\documentclass[english]{amsart}
\usepackage{a4wide}

\newcommand{\abs}[1]{\lvert#1\rvert}

\usepackage{amsfonts, amssymb, amsbsy, amsmath, mathrsfs, amsthm,enumerate}

\usepackage{graphicx}
\usepackage{url}
\usepackage{xcolor}
\usepackage{hyperref}
\usepackage{cleveref}

\newcommand\dx{\, d x}
\newcommand\dmu{\, d \mu}
\newcommand\dt{\, d t}
\newcommand\ds{\, d s}
\newcommand\R{\mathbb R}
\newcommand\dH{\, d \mathcal{H}^{n-1}}

\newcommand\wstar{
\overset{\ast}{\rightharpoondown}}
\newcommand{\edge}{\hspace{0.1em}\mbox{\LARGE$\llcorner$}\hspace{0.05em}}

\usepackage{mathrsfs}

\theoremstyle{definition}
\newtheorem{definition}{Definition}[section]

\newtheorem{remark}[definition]{Remark} 
\theoremstyle{plain}
\newtheorem{theorem}[definition]{Theorem}
\newtheorem{lemma}[definition]{Lemma}
\newtheorem{corollary}[definition]{Corollary}
\newtheorem{proposition}[definition]{Proposition}

\numberwithin{equation}{section}

\crefname{theorem}{Theorem}{Theorems}
\Crefname{theorem}{Theorem}{Theorems}
\crefname{lemma}{Lemma}{Lemmas}
\Crefname{lemma}{Lemma}{Lemmas}
\crefname{proposition}{Proposition}{Propositions}
\Crefname{proposition}{Proposition}{Propositions}
\crefname{corollary}{Corollary}{Corollaries}
\Crefname{corollary}{Corollary}{Corollaries}
\crefname{definition}{Definition}{Definitions}
\Crefname{definition}{Definition}{Definitions}
\crefname{remark}{Remark}{Remarks}
\Crefname{remark}{Remark}{Remarks}
\crefname{example}{Example}{Examples}
\Crefname{example}{Example}{Examples}

\RequirePackage{babel}

\def\Xint#1{\mathchoice
   {\XXint\displaystyle\textstyle{#1}}%
   {\XXint\textstyle\scriptstyle{#1}}%
   {\XXint\scriptstyle\scriptscriptstyle{#1}}%
   {\XXint\scriptscriptstyle\scriptscriptstyle{#1}}%
   \!\int}
\def\XXint#1#2#3{{\setbox0=\hbox{$#1{#2#3}{\int}$}
     \vcenter{\hbox{$#2#3$}}\kern-.5\wd0}}

\def\dashint{\Xint-}

\def\Yint#1{\mathchoice
    {\YYint\displaystyle\textstyle{#1}}%
    {\YYint\textstyle\scriptstyle{#1}}%
    {\YYint\scriptstyle\scriptscriptstyle{#1}}%
    {\YYint\scriptscriptstyle\scriptscriptstyle{#1}}%
      \!\iint}
\def\YYint#1#2#3{{\setbox0=\hbox{$#1{#2#3}{\iint}$}
    \vcenter{\hbox{$#2#3$}}\kern-.51\wd0}}
\def\longdash{{-}\mkern-3.5mu{-}} 

\def\fiint{\Yint\longdash}

\renewcommand{\div}{\operatorname{div}}
\DeclareMathOperator*{\esssup}{ess\,sup\,}

\DeclareMathOperator\BV{BV}

\DeclareMathOperator{\loc}{loc}
\DeclareMathOperator\supp{supp}
\DeclareMathOperator{\sign}{sign}
\DeclareMathOperator{\dist}{dist}

\title[Parabolic problems of linear growth]{Characterizations and properties of solutions to parabolic problems of linear growth}
\author{Theo Elenius}
\address{Department of Mathematics, Aalto University, P.O. Box 11100, FI-00076 Aalto, Finland}
\email{theo.elenius@aalto.fi}



\begin{document}
\begin{abstract}
    We consider notions of weak solutions to a general class of parabolic problems of linear growth, formulated independently of time regularity. Equivalence with variational solutions is established using a stability result for weak solutions. A key tool in our arguments is approximation of parabolic BV functions using time mollification and Sobolev approximations. We also prove a comparison principle and a local boundedness result for solutions. When the time derivative of the solution is in $L^2$ our definitions are equivalent with the definition based on the Anzellotti pairing.
\end{abstract}
\maketitle
\section{Introduction}
This article studies nonlinear parabolic equations of the type
\begin{equation}
    \partial_t u - \div \left( D_\xi f(x,Du)\right)= 0 \quad \text{on} \quad \Omega_T=\Omega\times(0,T), \label{eq 1}
\end{equation}
where $f:\Omega\times\R^n\to[0,\infty)$ is a functional of linear growth, $\Omega\subset\R^n$ is a bounded Lipschitz domain and $T>0$. Due to the linear growth of the Lagrangian, the natural function space for solutions is the space of functions of bounded variation. In this setting, the weak gradient $Du$ of a function $u$ is a vector-valued Radon measure. Hence, the usual definition of a weak solution via integration by parts does not apply for equations of the type \eqref{eq 1}.

The inherent difficulties are already present for the model equation  
\begin{equation}
    \label{TVF}
    \partial_t u - \div \left( \dfrac{Du}{\lvert Du\rvert} \right)= 0,
\end{equation}
which corresponds to the choice $f(x,\xi)=\lvert\xi\rvert$, known as the total variation flow (TVF). Two main approaches have been developed to define solutions to \eqref{eq 1}. On the one hand, $u$ is said to be a variational solution of \eqref{TVF} if $u$ satisfies variational inequalities of the type
\begin{align*}
    & \int_0^T \|Du\|(\Omega)\dt + \iint_{\Omega_T} u \partial_t \varphi \dx\dt \leq  \int_0^T \|D(u+\varphi)\|(\Omega)\dt, 
\end{align*}
for a suitable class of test functions $\varphi$. This variational approach to parabolic problems of linear growth dates back to Lichnewsky and Temam \cite{LicT78}, who applied it to the time-dependent minimal surface equation. For the total variation flow and more general functionals of linear growth, existence and uniqueness of variational solutions have been studied in a series of papers \cite{BDM2015,BoegelDuzSchevObstacle:2016,BoegelDuzSchevTime:2016,BDSS:2019}.

On the other hand, one can define weak solutions by applying the Anzellotti pairing \cite{Anzellotti:1984}. For this approach we refer to the monograph by Andreu, Caselles and Mazón \cite{Andreu-Caselles-Mazon:book}. The construction of the Anzellotti pairing is based on a product rule which, in the case of parabolic equations, requires the time derivative of the solution to lie in $L^2$. Hence, an a priori assumption on the time derivative is necessary in the definition of a weak solution. Moreover, weak solutions do not necessarily attain the prescribed boundary condition in the sense of traces. Consequently, the boundary data have to be taken in a suitably weak sense. One possibility is to require information on a weak normal trace related to a Gauss-Green formula for the Anzellotti pairing, which again depends on time regularity.

Existence of weak solutions to the Cauchy-Dirichlet problem for the total variation flow, with time-independent boundary data has been established in \cite{ABCM, Andreu-Caselles-Diaz-Mazon:2002}. The corresponding question has been studied for a more general family of equations in \cite{Andreu-Caselles-Mazon:2002, Moll2005}. For boundary data that is irregular in time it is not clear whether weak solutions have the required regularity in time to be able to apply the Anzellotti pairing. On the other hand, the definition of variational solutions requires no a priori regularity in time. Starting from minimal assumptions on solutions is natural in several contexts, for instance, when studying regularity theory or subsolutions and supersolutions.

The purpose of this article is to develop a notion of weak solution for parabolic problems of linear growth that does not rely on a priori time regularity, and to establish its equivalence with the variational approach. We investigate this question for the Cauchy-Dirichlet problem 
\begin{equation}
    \label{PDE}
    \begin{cases}
       \partial_tu-\div\big( D_\xi f(x,Du) \big)=0&\mbox{in }\Omega\times(0,T),\\
       u(x,t)=g(x,t)&\mbox{on }\partial\Omega\times(0,T),\\
       u(x,0)=u_o(x)&\mbox{in }\Omega.
     \end{cases}
\end{equation}
In this generality, it is not clear what the correct definition of weak solution is. In this work, we propose two notions of weak solutions, both of which are independent of the time regularity of the solution. To our knowledge, this is the first work in this direction and in this setting. The first notion is based on the approach using the Anzellotti pairing, and is given in \Cref{def: weak solution}. The dependency on the time derivative is eliminated in two ways. This is achieved, first, by a formal integration by parts in time for the Anzellotti pairing, and second, by replacing the normal trace condition for the boundary data with a condition using difference quotients of traces. The second characterization is an intermediate form combining the previous definition and the notion of variational solution, and is presented in the statement of \Cref{equivalences part 1}. The two characterizations of weak solutions are shown to be equivalent, and reduce to the classical definition based on the Anzellotti pairing, whenever the pairing is applicable. Both notions of weak solutions discussed in this article apply to equations with differentiable Lagrangians as well as to 1-homogeneous functionals, as in \cite{GornyMazon}.

We will also prove the equivalence of weak and variational solutions to \eqref{PDE}. Showing that weak solutions are variational solutions is straightforward, while the other direction is more difficult. For differentiable functionals the full equivalence is shown by differentiating the variational inequality in a suitable way. On the other hand, $1$-homogeneous functionals, such as the total variation, fail to be differentiable at the origin. For such functionals, we show an existence result for weak solutions. To do this, we approximate the functional with differentiable functionals and establish stability of weak solutions under this approximation in \Cref{stability theorem}. A comparison principle implies the uniqueness of variational solutions, and hence the existence result for weak solutions guarantees the equivalence of weak and variational solutions. The equivalence of weak and variational solutions for the total variation flow has previously been studied by applying duality methods in \cite{KinnunenScheven2022}.

For differentiable functionals a third approach to weak solutions is possible, based on the Euler-Lagrange equations for linear growth elliptic equations developed by Anzellotti in \cite{Anzellotti:1985}. In this setting, no Anzellotti pairing is needed, but the equation has to be tested with a larger class of functions than the Sobolev class. Instead, the allowed test functions are $\BV$ functions with a controlled jump part. In \Cref{prop: Anzellotti type euler equations}, we extend this approach to the parabolic setting. However, this approach seems to be applicable only in the case when the time derivative of the solution is a $L^2$ function.

A key tool in several arguments are two types of approximations of parabolic $\BV$ functions. The first is a mollification in time. We prove the area-strict convergence of such time mollifications. Combined with a Reshetnyak continuity theorem, this guarantees the convergence of the functional in the variational inequality under the time mollification. Secondly, we prove a parabolic version of a classical result concerning the approximation of functions of bounded variation with Sobolev functions. For such results in the time-independent setting we refer to the monographs \cite{AmbrosioFuscoPallara,Giusti1984} and \cite[Lemma 1]{KristensenRindler2010Young}. This approximation scheme plays an important role in avoiding the explicit use of the Anzellotti pairing.

The final contribution of this article is to prove the local boundedness of local variational solutions to \eqref{eq 1}. Based on an energy estimate in \Cref{prop: energy estimate}, we prove a quantitative estimate for local boundedness under an integrability assumption in \Cref{prop: local boundedness}. The main novelty is that our iteration technique does not require an a priori local boundedness assumption, which has been previously been the case for the total variation flow and singular parabolic equations of $p$-Laplacian type \cite{DiBeGV2012, DiBeGiaKla2017}. While continuity is not attainable, we discuss the existence of semicontinuous representatives for solutions based on our estimates. 

\subsection*{Acknowledgement} The  author would like to thank Juha Kinnunen for proposing the problem and for helpful discussions on the subject. 
\section{Preliminaries}
\label{sect: prelim}
\subsection{Functions of bounded variation}
In this article, $\Omega\subset\R^n$ is a bounded Lipschitz domain unless explicitly stated otherwise. A function $u\in L^1(\Omega)$ is said to be of bounded variation in $\Omega$ if its distributional gradient $Du$ is a finite $\R^n$-valued Radon measure on $\Omega$, and the space of functions of bounded variation will be denoted by $\BV(\Omega)$. The total variation measure of $Du$ will be denoted by $\|Du\|$, and is given by
\begin{equation*}
  \|Du\|(B)=\sup\bigg\{\sum_{i=1}^\infty |Du(B_i)|\colon
  B_i \mbox{ are pairwise disjoint Borel sets with }
  B=\bigcup_{i=1}^\infty B_i\bigg\},
\end{equation*}
for Borel sets $B\subset\Omega$. For an open set $U\subset\Omega$, we have
\begin{equation*}
    \| Du\|(U)= \sup\bigg\{ \int_\Omega u \div \Phi \dx : \Phi\in C_0^1(U,\R^n), \|\Phi\|_\infty \leq 1 \bigg\},
\end{equation*}
see \cite[Prop. 3.6]{AmbrosioFuscoPallara}. The distributional gradient $Du$ decomposes into an absolutely continuous part $D^a u$, with respect to the Lebesgue measure, and a singular part $D^su$. We shall use the notation $\nabla u$ for the Radon-Nikodym derivative of $D^a u$ with respect to the Lebesgue measure. Hence with this decomposition we have
\begin{equation*}
    Du = \nabla u \mathcal L^n + D^s u.
\end{equation*}
By \cite[Theorem 3.87]{AmbrosioFuscoPallara} there exist bounded inner and outer trace operators 
\begin{equation*}
    T_\Omega: \BV(\Omega)\to L^1(\partial\Omega) \quad \text{ and } \quad T_{\R^n\backslash \Omega}: \BV(\R^n\backslash\overline{\Omega})\to L^1(\partial\Omega).
\end{equation*}
Moreover, by \cite[Theorem 3.88]{AmbrosioFuscoPallara}, the trace operator $u\mapsto T_\Omega u$ is continuous between $\BV(\Omega)$ and $L^1(\partial\Omega)$ when $\BV(\Omega)$ is endowed with the topology induced by strict convergence. Moreover, the following extension result is given.
\begin{lemma}[{\cite [Lemma 3.89]{AmbrosioFuscoPallara}}]
\label{lemma: derivative of piecewise def function}
Let $u\in \BV(\Omega)$ and $v\in \BV(\R^n\backslash\overline{\Omega})$. Then the function
\begin{equation*}
    w(x)
    =\begin{cases}
        u(x), & \text{ for } x\in\Omega\\
        v(x) , & \text{ for } x\in \R^n\backslash\overline{\Omega}.
    \end{cases}
\end{equation*}
belongs to $\BV(\R^n)$, and its distributional gradient is given by
\begin{equation*}
    Dw= Du+Dv+\big(T_{\R^n\setminus\overline\Omega}v-T_\Omega
    u\big)\nu_\Omega\mathcal{H}^{n-1}\edge
    \partial\Omega,
\end{equation*}
where $\nu_\Omega$ denotes the generalized inner unit normal to $\Omega$. In the above we interpret $Du$ and $Dv$ as measures on the whole of $\R^n$, concentrated on $\Omega$ and $\R^n\backslash\overline{\Omega}$, respectively.
\end{lemma}
We recall the following lemma related to traces of $\BV$ functions. 
\begin{lemma}[{\cite[Lemma 7.2]{CDMLP1988}}]
\label{cut off trace lemma}
Let $\Omega_\epsilon=\{x\in\Omega:\dist(x,\partial\Omega)>\epsilon\}$ for $\epsilon>0$. Then for any $u\in \BV(\Omega)$ we have
\begin{equation*}
    \limsup_{\epsilon\to 0}\frac{1}{\epsilon}\int_{(\Omega\backslash\Omega_\epsilon)}\lvert u\rvert \dx \leq \int_{\partial\Omega}\lvert T_\Omega u\rvert \,d\mathcal{H}^{n-1}.
\end{equation*}
Moreover, there exists a constant $c=c(\partial\Omega)$ such that
\begin{equation*}
    \sup_{\epsilon>0} \frac{1}{\epsilon}\int_{\Omega\backslash\Omega_\epsilon}\lvert u\rvert\dx\leq c\|Du\|(\Omega). 
\end{equation*}
\end{lemma}
We will briefly recall the construction of the Anzellotti pairing and the associated Gauss-Green formula from \cite{Anzellotti:1984}. Neither will be explicitly present in our definition of a weak solution, however we use them for the discussion of the classical definition of a weak solution. Following \cite{Anzellotti:1984}, for $1\leq p<\infty$, we denote 
\begin{equation*}
    X_p(\Omega) = \{ z\in L^\infty(\Omega): \div z \in L^p(\Omega)\}.
\end{equation*}
and
\begin{equation*}
    X_\mu(\Omega) = \{z\in L^\infty(\Omega;\R^n) : \div z \text{ is a bounded measure on } \Omega \}.
\end{equation*}
For $u\in \BV(\Omega)\cap L^p(\Omega)$ and $z\in X_{p'}$, the Anzellotti pairing $(z,Du) $ is the distribution defined by 
\begin{equation*}
    (z,Du)(\varphi) = -\int_\Omega u \varphi \div z \dx - \int_\Omega u z\cdot\nabla \varphi\dx,
\end{equation*}
for $\varphi\in C_0^\infty(\Omega)$. In \cite{Anzellotti:1984} it is proved that $(z,Du)$ is a measure, with 
\begin{equation*}
    \lvert (z,Du)\rvert \leq \|z\|_\infty \|Du\|,
\end{equation*}
as measures. The absolutely continuous part is given by $(z,Du)^a = z\cdot \nabla u \mathcal L^n$. 

In \cite{Anzellotti:1984}, it is proved that there exists a linear operator $\gamma: X_\mu(\Omega)\to L^\infty(\partial\Omega)$ with
\begin{align*}
    \| \gamma(z)\|_\infty &  \leq \|z\|_\infty \\
    \gamma(z)(x)  = z(x)\cdot \nu_\Omega(x) \text{ for all } &  x\in\partial\Omega \text{ if } z\in C^1(\overline{\Omega},\R^n).
\end{align*}
The function $\gamma(z)(x)$ is interpreted as a weak trace of the normal component of $z$, and we denote $[z,\nu_\Omega](x)=\gamma(z)(x)$. With this notation the following Gauss-Green formula for the Anzellotti pairing is proven.
\begin{proposition}[{\cite [Thm. 1.9]{Anzellotti:1984}}]
For all $u\in \BV(\Omega)\cap L^p(\Omega)$ and $z\in X_{p'}(\Omega)$ we have
\begin{equation}
    \label{eq: Gauss Green in W11}
    \int_\Omega u \div z \dx + (z,Du)(\Omega) = \int_{\partial\Omega} [z, \nu_\Omega] T u \dH.
\end{equation}    
\end{proposition}

\subsection{Function of a measure}
Throughout this article we consider a Carathéodory function $f:\Omega\times\R^n\to [0,\infty)$ satisfying the following conditions. Firstly, we suppose there exists constants $\lambda,\Lambda>0$ such that
\begin{equation}
    \label{linear growth}
    \lambda \lvert \xi\rvert \leq f(x,\xi)\leq \Lambda (1+\lvert \xi\rvert),
\end{equation}
for all $(x,\xi)\in\Omega\times\R^n$. We require that
\begin{equation}
    \label{caratheodory and convex}
    \xi \mapsto f(x,\xi) \text{ is convex for a.e. $x\in\Omega$}
\end{equation}
Moreover, we assume the limit
\begin{equation}
    \label{recession function}
    f^\infty(x,\xi)=\lim_{y\to x, t\to 0^+, z\to \xi} t f\Big(y,\frac{z}{t}\Big) \,\, \text{ exists for all } (x,\xi)\in \overline{\Omega}\times(\R^n\backslash \{0\}),
\end{equation}
and defines a jointly continuous function in $(x,\xi)$. Note that from \eqref{recession function}, it follows that $f^\infty$ is positively $1$-homogeneous, that is,  $f^\infty(x,s\xi)=sf^\infty(x,\xi)$ for all $s>0$. We will finally assume that $f^\infty$ is symmetric in the sense that
\begin{equation}
    f^\infty(x,\xi)=f^\infty(x,-\xi).
\end{equation}
We note that it follows from \eqref{caratheodory and convex} that $f^\infty$ is convex in the second argument, and from \eqref{linear growth} it follows that 
\begin{equation*}
    \lambda \lvert \xi\rvert \leq f^\infty(x,\xi)\leq \Lambda \lvert \xi\rvert,
\end{equation*}
for all $(x,\xi)\in\Omega\times\R^n$. 

Given $\mu\in \mathcal M (\Omega;\R^n)$ we decompose $\mu=\mu^a \mathcal L^n +\mu^s$ into absolutely continuous and singular part, respectively, with respect to the Lebesgue measure and define the measure $f(\mu)$ by 
\begin{equation*}
    f(\mu)(B)=\int_B f(x,\mu^a(x))\dx+\int_B f^\infty \left(x, \dfrac{d\mu^s}{d\lvert \mu\rvert^s} \right) d\lvert \mu\rvert^s,
\end{equation*}
for Borel sets $B\subset \Omega$. 

The following is a specialized form of a classical lower semicontinuity result in $\BV$, which can be found for instance in \cite[Thm. 1.1]{Rindler2012}.
\begin{lemma}
\label{lemma: reshetnyak lsc}
Let $u\in \BV(\Omega)$ and let $u_i\in \BV(\Omega)$, $i\in\mathbb N$, be a sequence such that $u_i\to u$ in $L^1(\Omega)$ as $i\to\infty$. Then 
\begin{align*}
     f(Du)(\Omega)+ \int_{\partial\Omega}f^\infty\left(\cdot, \nu_\Omega\right) \lvert Tu\rvert \dH \leq \liminf_{i\to\infty} \left( f(Du)(\Omega)+ \int_{\partial\Omega}f^\infty\left(\cdot, \nu_\Omega\right) \lvert Tu\rvert \dH\right).
\end{align*}
\end{lemma}

We will also need a Reshetnyak continuity theorem. This version is a specialized form of \cite[Theorem 3.10]{Beck-Schmidt:2015}, see also \cite{KristensenRindler2010Young,KristensenRindler2010Relaxation}.
\begin{lemma}
\label{reshetnyak continuity}
Let $\Omega\subset\R^n$ be a bounded open set such that $\partial\Omega$ has zero $\mathcal L^n$ measure. Assume, $(\mu_k)_{k\in\mathbb N}$ weak$^*$-converges to $\mu$ in the space of finite $\R^n$ valued Radon measures on $\overline{\Omega}$. If
\begin{equation*}
    \lim_{k\to\infty} \lvert (\mathcal L^n,\mu_k)\rvert(\overline{\Omega})=\lvert (\mathcal L^n,\mu)\rvert(\overline{\Omega}),
\end{equation*}
then 
\begin{align*}
    \lim_{k\to\infty} & \left( \int_\Omega f\left(\cdot, \dfrac{d\mu_k^a}{d\mathcal L^n}\right)\dx + \int_{\overline\Omega} f^\infty\left(\cdot, \dfrac{d\mu_k^s}{d\lvert \mu_k^s\rvert}\right)\,d\lvert \mu_k^s\rvert \right) \\
    & = \int_\Omega f\left(\cdot, \dfrac{d\mu^a}{d\mathcal L^n}\right)\dx + \int_{\overline\Omega} f^\infty\left(\cdot, \dfrac{d\mu^s}{d\lvert \mu^s\rvert}\right)\,d\lvert \mu^s\rvert.
\end{align*}
\end{lemma}
Using the previously established notation, we have that
\begin{equation*}
    \lvert(\mathcal L^n,\mu)\rvert=\mathcal A(\mu),
\end{equation*}
as measures, where the area functional $\mathcal A$ is given by $\mathcal A(\xi)=\sqrt{1+\xi^2}$. A sequence $u_i\in \BV$ is said to converge area-strictly to $u \in \BV(\Omega)$ if $u_i\to u$ in $L^1(\Omega)$ as $i\to \infty$ and
\begin{equation*}
    \lim_{i\to\infty} \mathcal A(Du_i)(\Omega) = \mathcal A (Du)(\Omega).
\end{equation*}
The following is a corollary of the Reshetnyak continuity theorem for sequences in $\BV$.
\begin{corollary}
    \label{lemma:area strict convergences}
    Let $u\in \BV(\Omega)$ and let $u_i\in \BV(\Omega)$, $i\in\mathbb N$, be a sequence with $u_i\to u$ area-strictly as $i\to \infty$. Then for any $\varphi\in W^{1,1}(\Omega)$ we have $u_i+\varphi\to u+\varphi$ area-strictly as $i\to\infty$, and if $\varphi\in L^\infty(\Omega_T)$ then $\varphi u_i\to \varphi u$ area-strictly as $i\to\infty$.
\end{corollary}

\begin{proof}
Let us show that $u_i+\varphi\to u+\varphi$ area-strictly as $i\to\infty$, the proof for the other statement being similar. The convergence in $L^1(\Omega)$ is clear. By lower semicontinuity of the area functional, it is enough to show that 
\begin{equation*}
    \limsup_{i\to\infty} \mathcal A (D(u_i+\varphi))(\Omega)\leq \mathcal A (D(u+\varphi))(\Omega).
\end{equation*}
Let $\varphi_j\in C^\infty(\Omega)$ such that $\varphi_j\to \varphi$ in $W^{1,1}(\Omega)$ as $j\to\infty$. Using the estimate 
\begin{equation*}
    \mathcal A(\xi_1+\xi_2)\leq \mathcal A(\xi_{1})+\lvert\xi_2\rvert
\end{equation*} for all $\xi_1,\xi_2\in \R^n$, we find that
\begin{equation*}
    \mathcal A(D(u_i+\varphi))\leq \mathcal A(D(u_i+\varphi_j))+\|D(\varphi_j-\varphi)\|,
\end{equation*}
as measures on $\Omega$. Let furthermore $\eta\in C_0^\infty(\Omega)$ with $0\leq \eta\leq 1$ and estimate
\begin{equation}
    \label{eq: est for area strict conv added sobolev}
    \mathcal A(D(u_i+\varphi))\leq \eta \mathcal A(D(u_i+\varphi_j))+(1-\eta)(\|Du_i\|+\|D\varphi_h\|)+\|D(\varphi_j-\varphi)\|.
\end{equation}
 Consider the functional
\begin{equation*}
    f(x,\xi)=\eta(x) \sqrt{1+\lvert \xi+\nabla \varphi_j(x) \rvert^2}.
\end{equation*}
Then $f$ satisfies the conditions of \Cref{reshetnyak continuity} with $f^\infty(x,\xi) = \eta(x) \lvert\xi\rvert$. This shows that
\begin{equation*}
    \lim_{i\to \infty} \int_\Omega\eta \,d\mathcal A(D(u_i+\varphi_j)) = \lim_{i\to \infty} f(D(u_i))(\Omega)=f(Du)(\Omega)=\int_\Omega \eta \,d\mathcal A(D(u+\varphi_j)).
\end{equation*}
Applying the above in \eqref{eq: est for area strict conv added sobolev} we find that
\begin{align*}
    & \limsup_{i\to\infty} \mathcal A (D(u_i+\varphi))(\Omega) \\
    & \leq \int_\Omega\eta\,d \mathcal A(D(u+\varphi_j))+\int_\Omega(1-\eta)\,d(\|Du\|+\|D \varphi_h\|)+\|D(\varphi_j-\varphi)\|.
\end{align*}
Since $\eta$ was arbitrary we conclude 
\begin{align*}
    & \limsup_{i\to\infty} \mathcal A (D(u_i+\varphi))(\Omega)\leq \mathcal A(D(u+\varphi_j))(\Omega)+\|D(\varphi_j-\varphi)\|,
\end{align*}
and letting $j\to\infty$ we are done.
\end{proof}

For the purpose of studying weak solutions, we recall the following lemma on the differentiability of $f(\mu)$ from \cite{Anzellotti:1985}.
\begin{lemma}
\label{lemma: anzellotti derivative functional of measure}
Assume $f(x,\xi)$ is differentiable in $\xi$ for all $(x,\xi)\in \Omega\times\R^n$ and $f^\infty(x,\xi)$ is differentiable in $\xi$ for all $(x,\xi)\in \Omega\times(\R^n\backslash\{0\})$ with $\lvert D_\xi f\rvert\leq M$ and $\lvert D_\xi f\rvert \leq M$. Then $f(\mu)$ is differentiable at $\mu$ in the direction $\beta$ if and only if $\lvert \beta\rvert^s\ll\lvert \mu\rvert^s$, and one has 
\begin{equation*}
    \frac{d}{dh} f(\mu+h\beta)\Big\rvert_{h=0} = \int_\Omega D_\xi f(x,\mu^a)\cdot \beta^\alpha \dx + \int_\Omega D_\xi f^\infty\left(x,\frac{d\mu}{d\lvert \mu\rvert}\right)\cdot \frac{d\beta}{d\lvert\beta\rvert}\,d\lvert \beta\rvert^s.
\end{equation*}
\end{lemma}

\subsection{Parabolic function spaces}
Let $X$ be a Banach space and $1\leq p\leq \infty$. A map $v:[0,T]\to X $ is called strongly measurable if there exists a sequence of simple functions $v_k:[0,T]\to X$ such that $\|v(t)-v_k(t)\|_X\to 0$ as $k\to \infty$ for a.e. $t\in [0,T]$. The space $L^p(0,T;X)$ consists of equivalence classes of strongly measurable functions $v$ such that $t\mapsto \|v(t)\|_X\in L^p(0,T)$.

Let $\Omega\subset\R^n$ be an open set. Since $\BV(\Omega)$ is not separable, strong measurability turns out to be too restrictive for considering parabolic $\BV$ functions. However, $\BV(\Omega)$ can be identified as the dual of a separable space, whose elements are of the form $g-\div G$, where $g\in C_0^0(\Omega)$ and $G\in C_0^0(\Omega,\R^n)$. For $X=X_0'$, where $X_0$ is a separable space, a map $v:[0,T]\mapsto X$ is said to be weakly*-measurable if for every $\varphi\in X_0$ the function $t\mapsto \langle v(t),\varphi\rangle$ is measurable. This ensures that the map $t\mapsto \|v(t)\|_X$ is measurable, since 
\begin{equation*}
    \| v(t)\|_X= \sup \{ \langle v(t),\varphi\rangle ; \varphi\in X_0, \|\varphi\|_{X_0}\leq 1\},
\end{equation*}
and the space $X_0$ is separable. The weak*-Lebesgue space $L^p_{w*}(0,T;X)$ is then defined as the space of equivalence classes weak*-measurable functions $v$ with $t\mapsto \|v(t)\|_X\in L^p(0,T)$.

We recall the following fundamental lemma related to lower semicontinuity in parabolic BV spaces.
\begin{lemma}[{\cite [Lemma 2.3]{BoegelDuzSchevObstacle:2016}}]
    \label{low semcont of par BV}
    Let $u_i\in L^1_{w^*}(0,T;\BV(\Omega))$, $i\in\mathbb N$, be a sequence with $u_i\rightharpoondown u$ weakly in $L^1(\Omega_T)$ as $i\to \infty$ for some $u\in L^1(\Omega_T)$, and 
    \begin{equation*}
        \liminf_{i\to \infty}\int_0^T \|Du_i(t)\|(\Omega)\dt<\infty.
    \end{equation*}
    Then $u\in L^1_{w^*}(0,T;\BV(\Omega))$ and 
    \begin{equation*}
        \int_0^T \|Du(t)\|(\Omega)\dt \leq \liminf_{i\to \infty}\int_0^T\|Du_i(t)\|(\Omega)\dt.
    \end{equation*}
\end{lemma}

\subsection{Definition of variational and weak solutions}
Throughout this article, the boundary values are given by the trace of a function $g\in L^1_{w^*}(0,T;\BV(\Omega))$ and the initial values satisfy $u_o\in L^2(\Omega)$.  We begin with the variational approach. The form of the below definition is slightly different from the definition of variational solutions in \cite{BDM2015,BoegelDuzSchevTime:2016,BDSS:2019,BoegelDuzSchevObstacle:2016}, however the equivalence is guaranteed by \Cref{lemma: derivative of piecewise def function}.
\begin{definition}
A function $u\in L^1_{w^*}(0,T;\BV(\Omega))\cap L^\infty(0,T;L^2(\Omega))$ is a variational solution to \eqref{PDE} if
\begin{equation}
\label{var inequality}
\begin{split}
    & \int_0^\tau f(Du)(\Omega)\dt + \iint_{\partial\Omega\times(0,\tau)}\lvert Tg-Tu\rvert f^\infty(\cdot, \nu_\Omega)\dH\dt\\
    & \leq  \int_0^\tau  f(Dv)(\Omega)\dt + \iint_{\partial\Omega\times(0,\tau)}\lvert Tg-Tv\rvert f^\infty(\cdot, \nu_\Omega)\dH\dt \\ 
    & \quad +\iint_{\Omega_\tau} \partial_t v (v-u)\dx\dt +\frac{1}{2}\|v(0)-u_0\|^2_{L^2(\Omega)} - \frac{1}{2}\|(v-u)(\tau)\|^2_{L^2(\Omega)} ,
\end{split}
\end{equation}
holds for a.e. $\tau\in (0,T)$, and for every $v\in L^1_{w^*}(0,T;\BV(\Omega))\cap C([0,T],L^2(\Omega))$ with $\partial_t v\in L^2(\Omega_T)$.
\end{definition}
The following lemma shows that variational solutions attain the initial values.
\begin{lemma}
\label{lemma: var sol attains init val}
    \label{var sol initial value}
    Let $u$ be a variational solution of \eqref{PDE}. Then
    \begin{equation}
        \lim_{h\to 0}\frac{1}{h}\int_0^h\|u(t)-u_0\|_{L^2(\Omega)}^2 \dt = 0.
    \end{equation}
\end{lemma}
\begin{proof}
    Test \eqref{var inequality} with $v=u_{0,\epsilon}$, where $u_{0,\epsilon}$ denotes the standard mollification of $u_0$ and proceed similarly as in \cite[Lemma 2.8]{BoegelDuzSchevTime:2016}.
\end{proof}

We then begin our discussion of weak solutions. For simplicity, we will first consider the total variation flow \eqref{TVF}, which already exhibits the essential difficulties. Hence, we consider the problem 
\begin{equation}
    \label{eq: tvf Cauchy-Dirichlet problem}
    \begin{cases}
       \partial_tu-\div\bigg(\dfrac{Du}{\lvert Du\rvert}\bigg)=0&\mbox{in }\Omega\times(0,T),\\
       u(x,t)=g(x,t)&\mbox{on }\partial\Omega\times(0,T),\\
       u(x,0)=u_o(x)&\mbox{in }\Omega.
     \end{cases}
\end{equation}
Let us recall the approach using the Anzellotti pairing. Following \cite{Andreu-Caselles-Mazon:book}, a function $u\in C([0,T],L^2(\Omega))$ with $\partial_t u \in L^1(0,T;L^2(\Omega))$ is said to be a weak solution of \eqref{eq: tvf Cauchy-Dirichlet problem} if there exists a vector field $z\in L^\infty(\Omega_T)$ with $\|z\|_\infty \leq 1$ such that 
\begin{equation}
    \label{divergence as distributions}
    \div z = \partial_t u,
\end{equation}
in the sense of distributions on $\Omega_T$. The vector field $z$ should additionally satisfy the pairing condition 
\begin{equation}
    \label{strong pairing condition}
    (z(t),Du(t))=\|Du(t)\|  \text{ as measures, for a.e. $t\in (0,T)$},
\end{equation} 
and the boundary condition 
\begin{equation}
    \label{normal trace condition}
    [z(t),\nu_\Omega] \in \sign (Tg(t)-Tu(t)), \mathcal H^{n-1}- \text{a.e. for a.e. $t\in(0,T)$.}
\end{equation}
In this article, we consider the two sign functions defined by
\begin{equation*}
    \sign(a)
        =\begin{cases}
            \{1\} ,&\quad \text{if } a>0,\\
            [-1,1],&\quad \text{if } a=0,\\
            \{-1\},&\quad \text{if } a<0,
        \end{cases}
    \,\, \text{ and } \,\, \sign_0(a)
        =\begin{cases}
            1 ,&\quad \text{if } a>0,\\
            0,&\quad \text{if } a=0,\\
            -1,&\quad \text{if } a<0.
        \end{cases}
\end{equation*}

The vector field $z$ can be thought of as formally representing the quotient $\frac{Du}{\lvert Du\rvert}$. Given the construction of the Anzellotti pairing, the a priori assumption $\partial_t u\in L^1(0,T;L^2(\Omega))$ is clearly seen to be essential to state the pairing condition in the form \eqref{strong pairing condition}, and to make use of the boundary condition in the form \eqref{normal trace condition}.

Let us first comment on how to combine \eqref{divergence as distributions} and \eqref{normal trace condition}, still under the assumption $\partial_t u \in L^1(0,T;L^2(\Omega))$. The goal is to transform the conditions so as not to require an additional assumption on $\partial_t u$. Let $t\in (0,T)$. Using the Gauss-Green formula \eqref{eq: Gauss Green in W11}, \eqref{divergence as distributions}, and the assumption on the time derivative we have
\begin{equation*}
\begin{split}
    \int_{\Omega} w \partial_t u (t) \dx + \int_\Omega (z(t),Dw) & = \int_{\partial\Omega} [z(t),\nu_\Omega] Tw\dH \\
    &=\int_{\partial\Omega} \sign_0(Tg(t)-Tu(t)) Tw\dH, 
\end{split}
\end{equation*}
for any $w\in \BV(\Omega)$, satisfying the following trace condition
\begin{equation}
    \label{abs continuity of traces condition}
    Tg(t)(x)=Tu(t)(x) \implies Tw(t)(x)=0, \text{ for } \mathcal H^{n-1}\text{-a.e. } x\in \partial \Omega.
\end{equation}
The above condition appears naturally related to the differentiability of $\lvert \cdot \rvert$, as in \cite{Anzellotti:1985}. For $w\in L^1(0,T; W^{1,1}(\Omega))$ satisfying \eqref{abs continuity of traces condition} for a.e. $t\in (0,T)$, we integrate over $(0,T)$ to conclude
\begin{equation}
    \label{strong divergence condition with trace}
    \iint_{\Omega_T} w \partial_t u \dx\dt + \iint_{\Omega_T}z\cdot \nabla w\dx\dt = \iint_{\partial\Omega\times(0,T)} \sign_0(Tg-Tu) Tw\dH\dt, 
\end{equation}
Note that by taking $w$ with zero trace in the above we recover \eqref{divergence as distributions}. Moreover, testing \eqref{strong divergence condition with trace} with a test function whose trace is $g-u$, and applying Green's formula to the left hand side shows that
\begin{equation*}
    \iint_{\partial\Omega\times(0,T)} [z(t),\nu_\Omega] (Tg-Tu)\dH\dt = \iint_{\partial\Omega\times(0,T)} \lvert Tg-Tu\rvert\dH\dt.
\end{equation*}
Combining the above with the inequality $[z(t),\nu_\Omega] \leq 1$ shows that for a.e. $t\in (0,T)$ we have $[z(t),\nu_\Omega](Tg(t)-Tu(t))=\lvert Tg(t)-Tu(t)\rvert$, $\mathcal{H}^{n-1}$-a.e., which implies \eqref{normal trace condition}. 

We conclude that \eqref{strong divergence condition with trace} is equivalent to \eqref{divergence as distributions} together with \eqref{normal trace condition}. Using integration by parts, we thus find that the conjunction of \eqref{divergence as distributions} and \eqref{normal trace condition} is equivalent to
\begin{equation}
    \label{medium divergence condition with trace}
    \iint_{\Omega_T}z\cdot \nabla w\dx\dt -\iint_{\Omega_T} u \partial_t w \dx\dt = \iint_{\partial\Omega\times(0,T)} \sign_0(Tg-Tu) Tw\dH\dt, 
\end{equation}
for $w\in L^1(0,T;W^{1,1}(\Omega))$, compactly supported in time, with $\partial_t w \in L^2(\Omega_T)$ and satisfying \eqref{abs continuity of traces condition} for a.e. $t\in (0,T)$.

Let us also see how to recast the pairing condition. Integrating \eqref{strong pairing condition} over $(0,T)$ for a test function $\varphi\in C_0^\infty(\Omega_T)$ and using integration by parts shows
\begin{align*}
    \iint_{\Omega_T} \varphi \,d\|Du\|\dt &  = \iint_{\Omega_T} \varphi\,d (z(t),Du(t))\dt \\
    & = - \iint_{\Omega_T} u \varphi \div z + u z\cdot \nabla \varphi \dx\dt \\
    & = -\iint_{\Omega_T} u \varphi \partial_t u + u z\cdot \nabla \varphi \dx\dt \\
    & = \frac{1}{2}\iint_{\Omega_T} u^2 \partial_t \varphi \dx\dt -\iint_{\Omega_T} u z\cdot \nabla \varphi \dx\dt.
\end{align*}
It is not difficult to see that if $\partial_t u\in L^1(0,T;L^2(\Omega))$ then the condition 
\begin{equation}
    \label{tvf pairing condition}
    \iint_{\Omega_T} \varphi \,d\|Du\|\dt = \frac{1}{2}\iint_{\Omega_T} u^2 \partial_t \varphi \dx\dt -\iint_{\Omega_T} u z\cdot \nabla \varphi \dx\dt,
\end{equation}
for $\varphi\in C_0^\infty(\Omega_T)$, together with $\div z= \partial_t u$ implies \eqref{strong pairing condition}. Hence, in total we see that the conjunction of \eqref{medium divergence condition with trace} and \eqref{tvf pairing condition} is equivalent to the conditions in the definition of a weak solution. Moreover, neither of the conditions explicitly require the time derivative to state. 

However, assuming \eqref{medium divergence condition with trace} will not suffice in general, due to the restriction imposed on the traces in \eqref{abs continuity of traces condition}. The assumption does guarantee that $\div z = \partial_t u$ in the sense of distributions, as we can always test with smooth functions vanishing on the boundary. Yet, in the absence of time derivatives, we will also want to test the equation with a time mollification of $u-g$, which no longer satisfies the trace condition, and as such is not admissible.

Hence, we propose replacing \eqref{medium divergence condition with trace} with another condition, which takes the form
\begin{equation}
    \label{TVF divergence condition}
    \iint_{\Omega_T} z\cdot \nabla \varphi - u\partial_t \varphi \dx\dt \leq  \iint_{\partial\Omega\times(0,T)} \left(\lvert Tg-Tu+T\varphi\rvert - \lvert Tg-Tu\rvert \right)\, d \mathcal{H}^{n-1}\dt,
\end{equation}
for $\varphi\in L^1(0,T;W^{1,1}(\Omega))$, compactly supported in time, with $\partial_t\varphi\in L^1(0,T;L^2(\Omega))$. Let us note that if $\varphi\in C_0^\infty(\Omega_T)$ the right hand side of the above vanishes, and we can argue by linearity to conclude $\div z = \partial_t u$ in the sense of distributions. Moreover, if $\varphi$ satisfies the trace condition \eqref{abs continuity of traces condition}, then testing \eqref{TVF divergence condition} with $h\varphi$, dividing by $h$ and letting $h\to 0$ shows
\begin{equation}
    \iint_{\Omega_T} z\cdot \nabla \varphi - u\partial_t \varphi \dx\dt \leq  \iint_{\partial\Omega\times(0,T)} \sign_0(Tg-Tu) T\varphi\, d \mathcal{H}^{n-1}\dt.
\end{equation}
Again, by linearity, the above implies \eqref{medium divergence condition with trace}. Note that for taking the limit it is crucial that $\varphi$ satisfies the trace condition. On the other hand, it is not very hard to see that \eqref{medium divergence condition with trace} implies \eqref{TVF divergence condition}, if $\partial_t u\in L^1(0,T;L^2(\Omega))$.

Before stating the general defintion of weak solutions of \eqref{PDE}, recall that the convex conjugate of $f$ is defined by
\begin{equation*}
    f^*(x,z)=\sup_{z^*\in\R^n} z\cdot z^* - f(x,z^*),
\end{equation*}
and the subdifferential is given by
\begin{equation*}
    z\in \partial_\xi f(x,\nabla u) \iff z\cdot (\xi-\nabla u)\leq f(x,\xi)-f(x,\nabla u) \text{ for all } \xi\in\R^n.
\end{equation*}
\begin{definition}
\label{def: weak solution}
A function $u\in L^1_{w^*}(0,T;\BV(\Omega))\cap L^\infty(0,T;L^2(\Omega))$ is a weak solution of \eqref{PDE} if there exists a vector field $z\in L^\infty(\Omega_T,\R^n)$ such that
\begin{enumerate}[(i)]
\item $z\in \partial_\xi f(x,\nabla u)$ for a.e. $(x,t)\in\Omega_T$,
\item \begin{equation}
    \label{divergence condition}
        \begin{split}
        & \iint_{\Omega_T} z\cdot \nabla\varphi - u\partial_t \varphi \dx\dt \\
        & \leq  \iint_{\partial\Omega\times(0,T)} \left(\lvert T\varphi+Tg-Tu\rvert - \lvert Tg-Tu\rvert \right) f^\infty(\cdot, \nu_\Omega) \, d \mathcal{H}^{n-1}\dt,
    \end{split}
\end{equation}
for all $\varphi\in L^1(0,T;W^{1,1}(\Omega))$, compactly supported in time, with $\partial_t\varphi\in L^2(\Omega_T)$,
\item 
\begin{equation}
\label{pairing condition}
\begin{split}
    & \iint_{\Omega_T} \varphi \,df(Du)\dt + \iint_{\Omega_T} \varphi f^*(x,z)\dx\dt \\
    & = \frac{1}{2} \iint_{\Omega_T} u^2 \partial_t\varphi \dx\dt -\iint_{\Omega_T} u z\cdot \nabla\varphi\dx\dt,
\end{split}
\end{equation}
for all $\varphi\in C_0^\infty(\Omega_T)$, and 
\item \begin{equation}
    \label{initial value condition}
    \lim_{h\to 0}\frac{1}{h}\int_0^h\|u(t)-u_0\|_{L^2(\Omega)}^2 \dt = 0.
\end{equation}
\end{enumerate}
\end{definition}
\begin{remark}
\begin{enumerate}[(i)]
\item Writing \eqref{divergence condition} for $\varphi\in C_0^\infty(\Omega_T)$ and arguing by linearity shows that $\div z= \partial_t u$ in the sense of distributions. 
\item In the case $\partial_t u\in L^1(0,T;L^2(\Omega))$ the condition \eqref{pairing condition} is equivalent with 
\begin{equation*}
    \int_\Omega \varphi\, d f(Du(t)) + \int_\Omega f^*(x,z(x,t))\dx = -\int_\Omega \varphi u(t) \partial_t u(t) \dx - \int_\Omega u(t) z(t)\cdot\nabla \varphi\dx,
\end{equation*} 
for a.e. $t\in (0,T)$ and for $\varphi\in C_0^\infty(\Omega)$. As $\div z(t)=\partial_t u(t)$, the right hand side of the above is the standard Anzellotti pairing of $z$ and $u$. In this sense, the right hand side of \eqref{pairing condition} is a form of the Anzellotti pairing, partially integrated over time to eliminate the time derivative on $u$.
\item In the above definition, the initial condition is in an integrated form. This is a weaker assumption than requiring that $u(t)\to u_0$ in $L^2(\Omega)$ as $t\to 0$. The integrated form appears naturally starting from the variational definition, as seen in \Cref{lemma: var sol attains init val} and \cite[Lemma 2.8]{BoegelDuzSchevTime:2016}.
\end{enumerate}
\end{remark}

\section{Approximation results}
\label{section: approximations}
This section proves two approximation results which will be of later use in our study of weak and variational solutions to parabolic equations of linear growth. In the first part, we show the existence of sequences in $L^1(0,T;W^{1,1}(\Omega))$ approximating parabolic BV functions in the strict sense. In the second part we recall the exponential time mollification and show the area-strict convergence of this approximation.

\subsection{Approximation with parabolic Sobolev functions}
Let $\Omega\subset\R^n$ be an open set and let $u\in L^1_{w^*}(0,T;\BV(\Omega))$. We say that a sequence $u_i\in L^1_{w^*}(0,T;\BV(\Omega))$, $i\in\mathbb N$, converges strictly to $u$ in $\Omega_T$ if $u_i\to u$ in $L^1(\Omega_T)$ as $i\to\infty$ and
\begin{equation}
    \label{parabolic strict conv def}
    \lim_{i\to\infty}\int_0^T \left\lvert  \|Du\|(\Omega)- \|D u_i\|(\Omega) \right\rvert\dt = 0.
\end{equation}

Let $\Omega$ be bounded. A sequence $u_i\in L^1_{w^*}(0,T;\BV(\Omega))$, $i\in\mathbb N$, converges area-strictly to $u$ in $\Omega_T$ if $u_i\to u$ in $L^1(\Omega_T)$ as $i\to\infty$ and
\begin{equation}
    \label{parabolic area strict conv def}
    \lim_{i\to\infty}\int_0^T \left\lvert \mathcal A(Du)(\Omega)-\mathcal A(D u_i)(\Omega) \right\rvert\dt = 0.
\end{equation}

We present two lemmas on strict convergence in $L^1_{w^*}(0,T;\BV(\Omega))$. The first one is a parabolic version of \cite[Lemma 1]{KristensenRindler2010Young}. 
\begin{lemma}
\label{lemma: strict approx in parabolic bv}
Let $1\leq p<\infty$. Let $\Omega\subset\R^n$ be an open set and let $u\in L^1_{w^*}(0,T;\BV(\Omega))$. Then there exists a sequence $u_i\in L^1(0,T;W^{1,1}(\Omega))$, $i\in\mathbb N$, with the following properties.
\begin{enumerate}[(i)]
    \item $u_i\to u$ strictly in $\Omega_T$ as $i\to\infty$ and $u_i(t)\to u(t)$ strictly in $\Omega$ as $i\to\infty$, for a.e. $t \in(0,T)$.
    \item If $\Omega$ is bounded, $u_i\to u$ area-strictly in $\Omega_T$ as $i\to\infty$ and $u_i(t)\to u(t)$ area-strictly in $\Omega$ as $i\to\infty$, for a.e. $t \in(0,T)$.
    \item For every $i$, the zero extension of $u-u_i$ to $(\R^n\backslash\Omega) \times(0,T)$ satisfies $u-u_i\in L^1(0,T;\BV(\R^n))$ and for a.e. $t\in(0,T)$ 
    \begin{equation*}
        \|D(u-u_i)(t)\|(\partial\Omega)=0.
    \end{equation*}
    \item If $u\in L^p(\Omega_T)$, then $u_i\to u$ in $L^p(\Omega_T)$ as $i\to\infty$.
    \item If $\partial_t u\in L^p(\Omega_T)$, then $\partial_t u_i\in L^p(\Omega_T)$ and $\partial_t u_i\to \partial_t u$ in $L^p(\Omega_T)$ as $i\to\infty$. Similarly, if $\partial_t u \in L^1(0,T;L^p(\Omega))$ then $\partial_t u_i\in L^1(0,T;L^p(\Omega))$ and $\partial_t u_i\to \partial_t u$ in $L^1(0,T;L^p(\Omega))$ as $i\to\infty$.
    \item If for some $u_0\in L^p(\Omega)$ we have $u(t)\to u_0$ in $L^p(\Omega)$ as $t\to 0$, then there exists $u_{0,i}\in L^p(\Omega)$ with $u_i(t)\to u_{0,i}$ in $L^p(\Omega)$ as $t \to 0$. Moreover, $u_{0,i}\to u_0$ in $L^p(\Omega)$ as $i\to\infty$.
\end{enumerate}
\end{lemma}
\begin{proof}
The construction is similar to the one in \cite[Lemma 1]{KristensenRindler2010Young}. Let us first show how to obtain the basic construction, giving $(i)$ and $(iii)$, and then explain how to obtain the properties $(ii)$ and $(iv)-(vi)$ under their respective assumptions. Let $\epsilon>0$. For $(i)$, the lower semicontinuity of the $\BV$ norm implies that it is enough to find a function $u_\epsilon\in L^1(0,T;W^{1,1}(\Omega))$ such that 
\begin{align*}
    \iint_{\Omega_T} \lvert u_\epsilon -u\rvert \dx\dt \leq \epsilon,
\end{align*}
and
\begin{align*}
    \iint_{\Omega_T} \lvert \nabla u_\epsilon \rvert \dx\dt \leq \int_0^T\|Du\|(\Omega)\dt+\epsilon.
\end{align*}
Let $\Omega_i\subset\Omega$, $i\in \mathbb N$, be open sets such that $\Omega_i\Subset\Omega_{i+1}$ and $\bigcup_{i\in\mathbb N} \Omega_i=\Omega$. Let $U_0=\Omega_1$ and $U_i=\Omega_{i+1}\backslash \overline{\Omega}_{i-1}$ for $i=1,2,3,\dots$. The sets $U_i$ form an open cover of $\Omega$ with the property that each point $x$ belongs to at most two of the sets. Let $\eta_i$ be a partition of unity with respect to the covering $U_i$. That is, $\eta_i\in C_0^\infty(U_i)$, $0\leq \eta_i\leq 1$ and $\sum_{i=0}^\infty \eta_i = 1$ in $\Omega$.

Let $\varphi$ be a standard mollifier. For each $i\in \mathbb  N$ there exists $\epsilon_i$ such that 
\begin{equation}
    \label{eq: strict approximation supp contained}
    \supp (\varphi_{\epsilon_i}*\eta_i)\subset U_i,
\end{equation}
\begin{equation}
    \label{eq: strict approximation conv in l1}
    \iint_{\Omega_T} \lvert \eta_i u-\varphi_{\epsilon_i}*(\eta_i u)\rvert\dx\dt < \epsilon 2^{-i-1},
\end{equation}
and
\begin{equation}
    \label{eq: strict approximation choice for strict conv}
    \iint_{\Omega_T} \lvert \nabla \eta_i u-\varphi_{\epsilon_i}*(\nabla \eta_i u)\rvert\dx\dt <\epsilon 2^{-i-1}.
\end{equation} 
Define 
\begin{equation*}
    u_\epsilon(x,t)= \sum_{i=0}^\infty \varphi_{\epsilon_i}*(\eta_i u)(x,t)=\sum_{i=0}^\infty \varphi_{\epsilon_i}*(\eta_i u(t))(x).
\end{equation*}
By \eqref{eq: strict approximation conv in l1}
\begin{align*}
    \iint_{\Omega_T} \lvert u_\epsilon -u\rvert \dx\dt \leq \sum_{i=0}^\infty \int_{\Omega} \lvert \eta_i u-\varphi_{\epsilon_i}*(\eta_i u)\rvert\dx\dt <\epsilon.
\end{align*}
Moreover, we see that $u_\epsilon(\cdot,t)\in C^\infty(\Omega)$ and as $\sum_{i=0}^\infty \nabla \eta_i = 0$ we have
\begin{equation}
    \label{derivative of smooth approximation}
    \nabla u_\epsilon = \sum_{i=0}^\infty \varphi_{\epsilon_i} * (\eta_i Du) + \sum_{i=0}^\infty\left( \varphi_{\epsilon_i}*(u\nabla\eta_i)-u\nabla \eta_i \right).
\end{equation}
From this, we conclude that 
\begin{align*}
    \iint_{\Omega_T} \lvert \nabla u_\epsilon \rvert \dx\dt &  \leq \sum_{i=0}^\infty \iint_{\Omega_T} \varphi_{\epsilon_i}*(\eta_i\lvert Du\rvert)\dx\dt+\sum_{i=0}^\infty \iint_{\Omega_T} \lvert \nabla \eta_i u-\varphi_{\epsilon_i}*(\nabla \eta_i u)\rvert\dx\dt \\
    & \leq  \sum_{i=0}^\infty \iint_{\Omega_T} \eta_i\,d\lvert Du\rvert\dt+\epsilon \\
    & = \int_0^T\|Du\|(\Omega)\dt+\epsilon.
\end{align*}
To see that $(iii)$ holds we argue as in \cite{KristensenRindler2010Young}. Denote 
\begin{equation*}
        w
        =\begin{cases}
            u_\epsilon -u, & \text{ on } \Omega\\
            0 , & \text{ on } \R^n\backslash\overline{\Omega}.
        \end{cases}
\end{equation*}
Let $0<t<T$ be such that $u(t)\in \BV(\Omega)$ and
\begin{equation}
    \label{eq: strict approximation choose time at which series converges}
    \sum_{i=0}^\infty\int_{\Omega} \lvert \nabla \eta_i u(t)-\varphi_{\epsilon_i}*(\nabla \eta_i u(t))\rvert\dx <\infty.
\end{equation} 
This holds for a.e. $t\in (0,T)$ by \eqref{eq: strict approximation conv in l1} and \eqref{eq: strict approximation choice for strict conv} respectively. Let $\psi\in C_0^1(\R^n,\R^n)$ with $\|\psi\|_\infty\leq 1$. From \eqref{eq: strict approximation supp contained} we find that $\supp \varphi_{\epsilon_i}*(\eta_i u(t))\subset U_i$, and hence
\begin{equation}
\label{eq: extension is BV}
\begin{split}
    \int_{\R^n} w(t) \div \psi \dx & =\int_\Omega (u_\epsilon(t)-u(t))\div \psi\dx \\
    & = \sum_{i=0}^\infty \int_\Omega (\varphi_{\epsilon_i}*(\eta_i u(t))-\eta_i u(t)) \div \psi\dx \\
    & = \sum_{i=0}^\infty \left( \int_{U_i} \eta_i \psi \,d Du(t) - \int_{U_i} \psi \varphi_{\epsilon_i}*(\eta_i Du(t))\dx \right) \\
    & \quad + \sum_{i=0}^\infty \int_{U_i}\psi \left( u(t) \nabla\eta_i-\varphi_{\epsilon_i} *(u(t)\nabla\eta_i)\ \right)\dx.
\end{split}
\end{equation}
From \eqref{eq: strict approximation choose time at which series converges} we conclude 
\begin{equation*}
    \|Dw(t)\|(\R^n)\leq 2 \|Du(t)\|(\Omega)+\sum_{i=0}^\infty\int_{\Omega_T} \lvert \nabla \eta_i u(t)-\varphi_{\epsilon_i}*(\nabla \eta_i u(t))\rvert\dx\dt <\infty.
\end{equation*}
Moreover, integrating \eqref{eq: extension is BV} over $(0,T)$ shows that 
\begin{equation*}
    \int_0^T \|Dw(t)\|(\R^n)\dt \leq 2\int_0^T\|Du(t)\|+\epsilon,
\end{equation*}
and hence $w\in L^1(0,T;\BV(\R^n))$. To see that $\|Dw(t)\|(\partial\Omega)=0$ for a.e. $t\in (0,T)$ we let $k\in\mathbb N$ and take $U\supset\partial\Omega$ open with $U\cap U_i=\emptyset$ for all $i<k$. Let $\psi\in C_0^1(U,\R^n)$ with $\|\psi\|_\infty\leq 1$. From \eqref{eq: extension is BV} we have
\begin{align*}
    \int_{U} w(t) \div \psi \dx & \leq  \sum_{i=k}^\infty \left( \int_{U_i} \eta_i \,d \lvert Du(t)\rvert + \int_{U_i} \varphi_{\epsilon_i}*(\eta_i \lvert Du(t)\rvert\dx \right) \\
    & \quad + \sum_{i=k}^\infty \int_{U_i}\lvert u(t) \nabla\eta_i-\varphi_{\epsilon_i} *(u(t)\nabla\eta_i)\rvert\dx.
\end{align*}
Taking supremum over $\psi$ we arrive at
\begin{align*}
    \|Dw(t)\|(\partial\Omega)& \leq \|Dw(t)\|(U) \\
    & \leq \sum_{i=k}^\infty \left( \int_{U_i} \eta_i \,d \lvert Du(t)\rvert + \int_{U_i} \varphi_{\epsilon_i}*(\eta_i \lvert Du(t)\rvert)\dx \right) \\
    & \quad + \sum_{i=k}^\infty \int_{U_i}\lvert u(t) \nabla\eta_i-\varphi_{\epsilon_i} *(u(t)\nabla\eta_i)\rvert\dx,
\end{align*}
and letting $k\to \infty$ proves the claim, as the series in the right hand side of the above converge.

Let us consider the case when $\Omega\subset\R^n$ is bounded and see that $(ii)$ holds. By lower semicontinuity of the area functional, it is enough to show that
\begin{equation}
    \label{eq: area strict conv of smooth approx}
    \int_0^T \mathcal A(Du_{\epsilon})(\Omega)\dt \leq \int_0^T \mathcal A(Du)(\Omega)\dt +\epsilon.
\end{equation}
Applying the estimate $\mathcal A(\xi_1+\xi_2)\leq \mathcal A(\xi_{1})+\lvert\xi_2\rvert$ for all $\xi_1,\xi_2\in\R^n$, \eqref{derivative of smooth approximation} shows that 
\begin{align*}
    & \iint_{\Omega_T} \mathcal A(\nabla u_\epsilon) \dx\dt \\
    & \leq \sum_{i=0}^\infty \iint_{\Omega_T} \mathcal A( \varphi_{\epsilon_i}*(\eta_i Du))\dx\dt+\sum_{i=0}^\infty \iint_{\Omega_T} \lvert \nabla \eta_i u-\varphi_{\epsilon_i}*(\nabla \eta_i u)\rvert\dx\dt. 
\end{align*}
To estimate the first term in the above we use Jensen's inequality and the convexity of $\mathcal A$ to obtain
\begin{align*}
    \sum_{i=0}^\infty \iint_{\Omega_T} \mathcal A( \varphi_{\epsilon_i}*(\eta_i Du))\dx\dt & \leq \sum_{i=0}^\infty \iint_{\Omega_T}  \varphi_{\epsilon_i}*\mathcal A(\eta_i Du)\dx\dt \\
    & = \sum_{i=0}^\infty \int_0^T \mathcal A(\eta_i Du)(\Omega)\dt \\
    & \leq \sum_{i=0}^\infty \iint_{\Omega_T}\eta_i \,d \mathcal A(Du)\dt+\int_\Omega(1-\eta_i)\dx\dt \\
    & = \int_0^T \mathcal A(Du)(\Omega)\dt.
\end{align*}
Recalling \eqref{eq: strict approximation choice for strict conv} we have thus shown \eqref{eq: area strict conv of smooth approx}.

To show $(iv)$, we assume $u\in L^p(\Omega_T)$ and choose $\epsilon_i$ to satisfy in addition to the previous constraints also
\begin{equation*}
    \left(\iint_{\Omega_T} \lvert \eta_i u-\varphi_{\epsilon_i}*(\eta_i u)\rvert^p\dx\dt\right)^\frac{1}{p} < \epsilon 2^{-i-1},
\end{equation*}
which gives
\begin{equation*}
    \|u_\epsilon-u\|_{L^p(\Omega_T)}<\epsilon.
\end{equation*}
For $(v)$, if $\partial_t \varphi \in L^p(\Omega_T)$, then we see that 
\begin{equation*}
    \partial_t u_\epsilon = \sum_{i=0}^\infty \varphi_{\epsilon_i}*(\eta_i \partial_t u)
\end{equation*}
and we choose $\epsilon$ to satisfy
\begin{equation*}
    \left(\iint_{\Omega_T} \lvert \eta_i \partial_t u-\varphi_{\epsilon_i}*(\eta_i \partial_t u)\rvert^p\dx\dt\right)^\frac{1}{p} < \epsilon 2^{-i-1}.
\end{equation*}
A similar argument shows the claim for $\partial_t u\in L^1(0,T;L^p(\Omega_T))$. 

Finally, lets show $(vi)$. Suppose
\begin{equation*}
    \lim_{t\to 0}\|u(t)-u_0\|_{L^p(\Omega)} = 0,
\end{equation*}
and denote
\begin{equation*}
    u_{0,\epsilon} = \sum_{i=0}^\infty \varphi_{\epsilon_i}*(\eta_i u_0).
\end{equation*}
We claim that $u_\epsilon(t)\to u_{0,\epsilon}$ in $L^p(\Omega)$ as $\to 0$. To see this, note that as $\supp(\varphi_{\epsilon_i}*(\eta_i(u-u_0))\subset U_i$ and each $x\in \Omega$ is contained in at most two of the sets $U_i$ we have 
\begin{align*}
    \int_\Omega \lvert u_\epsilon(x,t)-u_{0,\epsilon}\rvert^p\dx & = \int_\Omega \left\lvert \sum_{i=0}^\infty \varphi_{\epsilon_i} *(\eta_i(u(x,t)-u_0(x))\right\rvert^p\dx \\
    & \leq 2^p \sum_{i=0}^\infty \int_\Omega \left\lvert \varphi_{\epsilon_i} *(\eta_i(u(x,t)-u_0(x))\right\rvert^p\dx \\
    & \leq 2^p\sum_{i=0}^\infty \int_\Omega \left\lvert (\eta_i(u(x,t)-u_0(x))\right\rvert^p\dx \\
    & \leq 2^p \|u(t)-u_0\|_{L^p(\Omega)}^p.
\end{align*}
To conclude that $u_{0,\epsilon}\to u_0$ in $L^p(\Omega)$ as $\epsilon\to 0$, we choose $\epsilon_i$ to satisfy 
\begin{equation*}
    \left(\int_\Omega \lvert \eta_i u_0- \varphi_{\epsilon_i}*(\eta_i u_0)\rvert^p\dx\right)^\frac{1}{p}<\epsilon 2^{-i-1},
\end{equation*}
which finishes the proof of $(vi)$.
\end{proof}

\subsection{Mollification in time}
Given $w\in L^1(\Omega)$ and $\delta>0$, define the exponential time mollification of $u\in L^1(\Omega_T)$ by
\begin{equation}
    \label{time mollification}
    u^\delta(x,t)=e^{-\frac{t}{\delta}}w+\frac{1}{\delta}\int_0^t e^{\frac{s-t}{\delta}} u(x,s)\ds.
\end{equation}
The function $w$ represents the initial value for the mollification.  In our application to parabolic problems of linear growth, we will want to choose $w=u_0$. However, we must  approximate $u_0$ to ensure $u^\delta\in L^1_{w^*}(0,T;\BV(\Omega))$, since the initial data is in general only in $L^2$. 

The following properties of the time mollification will be useful to us. For a proof we refer to \cite[Appendix B]{BDM2013}, see also \cite[Lemma 2.2]{KinnunenLindqvist2006}.
\begin{lemma}
    Let $1\leq p <\infty $, $u\in L^p(\Omega_T)$ and $w\in L^p(\Omega)$. Then the time mollification $u^\delta$ satisfies the following properties. 
    \begin{enumerate}[(i)]
        \item We have $u^\delta\in L^p(\Omega_T)$ with \begin{equation*}
            \| u^\delta\|_{L^p(\Omega_T)} \leq \|u\|_{L^p(\Omega_T)} + \delta^{\frac{1}{p}} \|w\|_{L^p(\Omega)}.
        \end{equation*} 
        Moreover, $u^\delta\to u$ in $L^p(\Omega_T)$ as $\delta\to 0$.
        \item The weak time derivative $\partial_t u^\delta$ exists with $\partial_t u^\delta\in L^p(\Omega_T)$ and satisfies
        \begin{equation}
            \label{derivative of time mollification}
            \partial_t u^\delta = \frac{1}{\delta}(u-u^\delta).
        \end{equation}
        \item If $u\in L^\infty(0,T;L^p(\Omega))$ then $u^\delta \in C([0,T],L^p(\Omega))$ and $u^\delta(0)=w$.
    \end{enumerate}
\end{lemma}
The following lemma proves the area-strict convergence of time mollifications. The strict convergence of the time mollification has already been observed in \cite[Lemma 2.6]{BDM2015}.
\begin{lemma}
    \label{lemma: time moll conv area strictly}
    Let $u\in L^1_{w^*}(0,T;\BV(\Omega))$, $w\in \BV(\Omega)\cap L^2(\Omega)$ and let $u^\delta$ be as in \eqref{time mollification}. Then $u^\delta\in L^1_{w^*}(0,T;\BV(\Omega))$ and for some sequence $\delta_i>0$, $i\in\mathbb N$, with $\delta_i\to 0$ as $i\to\infty$, we have  $u^{\delta_i} \to u$ area-strictly in $\Omega_T$ as $i\to\infty$.
\end{lemma}
\begin{proof}
It is shown in \cite[Lemma 2.6]{BDM2015} that $u^\delta\in L^1_{w^*}(0,T;\BV(\Omega))$. We already know $u^\delta \to u$ in $L^1(\Omega_T)$ as $\delta\to 0$. Thus we find a sequence $\delta_i$, with $\delta_i\to 0$ as i $\to\infty$ with $u^{\delta_i}(t)\to u(t)$ for a.e. $t$. Recall 
\begin{equation*}
    \mathcal{A}(Dv)(\Omega)=\sup\left\{ \int_\Omega \sum_{i=1}^{n} 
    v\partial_i\Phi_i +\Phi_{n+1}\dx; \Phi\in C^1(\Omega,\R^{n+1}), \|\Phi\|_\infty\leq 1 \right\}
\end{equation*}
for $v\in \BV(\Omega)$, cf. \cite[Def. 14.1]{Giusti1984}. The lower semicontinuity of the area functional gives
\begin{equation}
    \label{time mollification area lower smc}
    \mathcal{A}(Du(t))(\Omega)\leq\liminf_{i\to\infty} \mathcal A(Du^{\delta_i}(t))(\Omega). 
\end{equation}
for a.e. $t\in (0,T)$. On the other hand, for any $\Phi\in C^1(\Omega,\R^{n+1})$ with $ \|\Phi\|_\infty\leq 1$ we have
\begin{align*}
    & \int_\Omega \sum_{j=1}^{n} u^{\delta_i}(t)\partial_j\Phi_j +\Phi_{n+1}\dx \\
    & = \int_\Omega \sum_{j=1}^{n} \left(\int_0^t \delta_i^{-1} e^{\frac{s-t}{\delta_i}} u(s)\ds + e^{-\frac{t}{\delta_i }}w\right)\partial_j\Phi_j +\Phi_{n+1}\dx \\
    &  = \int_\Omega \sum_{j=1}^{n} \int_0^t \delta_i^{-1} e^{\frac{s-t}{\delta_i}} u(s)\partial_j\Phi_j +\Phi_{n+1}\ds\dx  \\
    & \quad +  e^{-\frac{t}{\delta_i}} \int_\Omega w \sum_{j=1}^{n}\partial_j\varphi_j\dx + e^{-\frac{t}{\delta_i}}\int_\Omega \Phi_{n+1}\dx \\
    & \leq \int_0^t \delta_i^{-1}e^\frac{s-t}{\delta_i} \mathcal A(Du(s))(\Omega)\ds + e^{-\frac{t}{\delta_i}}\lvert\Omega\rvert,
\end{align*}
and taking supremum over $\Phi$ we have 
\begin{equation*}
    \mathcal A(Du^{\delta_i}(t))(\Omega)\leq \int_0^t \delta_i^{-1}e^\frac{s-t}{\delta_i} \mathcal A(Du(s))(\Omega)\ds  + e^{-\frac{t}{\delta_i}}\lvert\Omega\rvert,
\end{equation*}
for a.e. $t\in (0,T)$. Integrating the above over $(0,T)$ shows that
\begin{align*}
    \limsup_{i\to\infty}\int_0^T \mathcal{A}(Du^{\delta_i})(\Omega)\dt & \leq \int_0^T \mathcal{A}(Du)(\Omega)\dt +\lim_{i\to\infty} \delta_i(1-e^{-\frac{T}{\delta_i}})(\|Dw\|(\Omega)+\lvert\Omega\rvert) \\
    & = \int_0^T \mathcal{A}(Du)(\Omega)\dt
\end{align*}
Together with \eqref{time mollification area lower smc} this implies \eqref{parabolic area strict conv def}.
\end{proof}

It is well known that the trace operator is continuous with respect to strict convergence. To conclude this section, we show the following lemma on the convergence of mollified trace terms.
\begin{lemma}
    \label{mollification of trace terms}
    Let $u\in L^1_{w^*}(0,T;\BV(\Omega))$. Then 
    \begin{equation}
        \label{conv of mollified trace}
        \lim_{\delta\to 0}\int_{0}^T\int_0^t\int_{\partial\Omega} \delta^{-1}e^{-\frac{s}{\delta}} \lvert Tu(t-s) - Tu(t)\rvert\, d \mathcal{H}^{n-1}\ds\dt=0.
    \end{equation}
    Moreover,
    \begin{equation}
        \label{conv of double mollified trace}
        \lim_{\delta\to 0}\int_{0}^T\int_0^t\int_{\partial\Omega} \delta^{-1}e^{-\frac{s}{\delta}} \lvert Tu^\delta(t-s) - Tu(t)\rvert\, d \mathcal{H}^{n-1}\ds\dt=0.
    \end{equation}
\end{lemma}
\begin{proof}
Let us prove \eqref{conv of mollified trace}, the proof of \eqref{conv of double mollified trace} being similar. Let $\varphi\in C_0^\infty(\Omega)$ with $0\leq \varphi\leq 1$. Given $s,t>0$,  denote 
\begin{equation*}
    w(s)=\varphi u(t)+(1-\varphi)u(s).    
\end{equation*}
Note that $Tw(s)=Tu(s)$ and $u(t)-w(s)=(1-\varphi)(u(t)-u(s))$. We thus have
    \begin{align*}
        & \int_{0}^T\int_0^t\int_{\partial\Omega} \delta^{-1}e^{-\frac{s}{\delta}} \lvert Tu(t-s) - Tu(t)\rvert\, d \mathcal{H}^{n-1}\ds\dt \\
        & = \int_{0}^T\int_0^t\int_{\partial\Omega} \delta^{-1}e^\frac{s-t}{\delta} \lvert Tw(s) - Tu(t)\rvert\, d \mathcal{H}^{n-1}\ds\dt \\
        & \leq c \int_{0}^T\int_0^t\delta^{-1}e^\frac{s-t}{\delta} \| D( w(s) - Tu(t))\|(\Omega)\ds\dt \\
        & \leq c \int_{0}^T\int_0^t\delta^{-1}e^\frac{s-t}{\delta}  \int_\Omega(1-\varphi)\,d (\|Du(s)\|+\|Du(t)\|) \\
        & \quad + \int_\Omega \lvert \nabla\varphi\rvert \lvert u(s)-u(t) \rvert\dx \ds\dt\\
        & \leq c \int_{0}^T\int_\Omega(1-\varphi)\,d\|Du(t)\|\dt + c \int_{0}^T\int_0^t\delta^{-1}e^\frac{s-t}{\delta}\int_\Omega \lvert \nabla\varphi\rvert \lvert u(s)-u(t) \rvert \dx\ds\dt,
    \end{align*}
    where in the last step we used the contraction property of the time mollification. Letting $\delta\to 0$ shows
    \begin{align*}
        & \limsup_{\delta\to 0}\int_{0}^T\int_0^t\int_{\partial\Omega} \delta^{-1}e^{-\frac{s}{\delta}} \lvert Tu(t-s) - Tu(t)\rvert\, d \mathcal{H}^{n-1}\ds\dt \\
        & \leq c \iint_{\Omega_T}(1-\varphi)\,d\|Du\|\dt+ c \lim_{\delta\to 0}c \int_{0}^T\int_0^t\delta^{-1}e^\frac{s-t}{\delta}\int_\Omega \lvert \nabla\varphi\rvert \lvert u(s)-u(t) \rvert \dx\ds\dt\\
        & = c \iint_{\Omega_T}(1-\varphi)\,d \|Du\|\dt,
    \end{align*}
    and since $\varphi$ was arbitrary we conclude 
    \begin{equation*}
        \lim_{\delta\to 0}\int_{0}^T\int_0^t\int_{\partial\Omega} \delta^{-1}e^{-\frac{s}{\delta}} \lvert Tu(t-s) - Tu(t)\rvert\, d \mathcal{H}^{n-1}\ds\dt=0.
    \end{equation*}
\end{proof}

\section{Comparison principle}
\label{section: comparison principle}
In this section we prove a comparison principle for variational solutions of \eqref{PDE}. For variational solutions to the total variation flow, this was shown under an assumption on the boundary data in \cite{BoegelDuzSchevTime:2016}. We will modify the technique to be able to use it for more general boundary data and for general functionals.

The following lemma is a version of \cite[Lemma 4.2]{BoegelDuzSchevTime:2016}, and will be necessary in the proof of the comparison principle.
\begin{lemma}\label{lemma for comp principle}
    Let $u$ be a variational solution of \eqref{PDE}. Let $w\in \BV(\Omega)\cap L^2(\Omega)$, $\delta>0$ and let $u^\delta$ be as in \eqref{time mollification}. Then
    \begin{equation*}
        \limsup_{\delta\to 0}\frac{1}{\delta} \iint_{\Omega_T}(u^\delta-u)^2\dx\dt \leq \frac{1}{2}\|w-u_0\|_{L^2(\Omega)}^2 
    \end{equation*}
\end{lemma}
\begin{proof}
    Test \eqref{var inequality} with $u^\delta$, use \eqref{derivative of time mollification} and let $\delta \to 0$.
\end{proof}
The next result is our comparison principle. The full statement is in terms of comparisons of extensions of the boundary values. However, we will only apply the result to conclude uniqueness for fixed boundary values.
\begin{lemma}\label{comparison principle}
    Let $u_0,v_0\in L^2(\Omega)$, $g,h\in L^1_{w^*}(0,T;\BV(\Omega))$ such that $g$ and $h$ can be extended to $\R^n\times(0,T)$ with $g\leq h$ a.e. in $(\R^n\backslash \Omega)\times(0,T)$. Suppose $u$ and $v$ are variational solutions of \eqref{PDE} for the initial and boundary data $(u_0,g)$ and $(v_0,h)$, respectively. Then
    \begin{equation*}
        \| (u-v)_+(\tau)\|_{L^2(\Omega}^2\leq \| (u_0-v_0)_+\|_{L^2(\Omega)}^2
    \end{equation*}
\end{lemma}
\begin{proof}
Let $\epsilon>0$, and let $u_{0,\epsilon}$, $v_{0,\epsilon}$ be standard mollifications of $u_0$ and $v_0$ respectively. For $\delta>0$, let $u^\delta,v^\delta$ be as in \eqref{time mollification} with $w$ replaced by $u_{0,\epsilon},v_{0,\epsilon}$, respectively. Testing $u$ with $\min(u^\delta,v^\delta)=u^\delta-(u^\delta-v^\delta)_+$ gives 
\begin{equation}
\label{comparison principle estimate 1}
\begin{split}
    & \int_0^\tau f(Du)(\Omega)\dt + \iint_{\partial\Omega\times(0,\tau)}\lvert Tg-Tu\rvert f^\infty(\cdot, \nu_\Omega)\dH\dt\\
    & \leq  \int_0^\tau  f(D\min(u^\delta,v^\delta))(\Omega)\dt + \iint_{\partial\Omega\times(0,\tau)}\lvert Tg-T \min(u^\delta,v^\delta)\rvert f^\infty(\cdot, \nu_\Omega)\dH\dt \\ 
    & \quad +\iint_{\Omega_\tau} \partial_t \min(u^\delta,v^\delta) (\min(u^\delta,v^\delta)-u)\dx\dt \\
    & \quad +\frac{1}{2}\|u_0-u_{0,\epsilon}+(u_{0,\epsilon}-v_{0,\epsilon})_+\|^2_{L^2(\Omega)} - \frac{1}{2}\|(\min(u^\delta,v^\delta)-u)(\tau)\|^2_{L^2(\Omega)}.
\end{split}
\end{equation}
Similarly, testing $v$ with $\max(u^\delta,v^\delta)=v^\delta+(u^\delta-v^\delta)_+$ gives
\begin{equation}
\label{comparison principle estimate 2}
\begin{split}
    & \int_0^\tau f(Dv)(\Omega)\dt + \iint_{\partial\Omega\times(0,\tau)}\lvert Th-Tu\rvert f^\infty(\cdot, \nu_\Omega)\dH\dt\\
    & \leq  \int_0^\tau  f(D\max(u^\delta,v^\delta))(\Omega)\dt + \iint_{\partial\Omega\times(0,\tau)}\lvert Th-T \max(u^\delta,v^\delta)\rvert f^\infty(\cdot, \nu_\Omega)\dH\dt \\ 
    & \quad +\iint_{\Omega_\tau} \partial_t \max(u^\delta,v^\delta) (\max(u^\delta,v^\delta)-u)\dx\dt \\
    & \quad +\frac{1}{2}\|v_0-v_{0,\epsilon} - (u_{0,\epsilon}-v_{0,\epsilon})_+\|^2_{L^2(\Omega)} - \frac{1}{2}\|(\max(u^\delta,v^\delta)-v)(\tau)\|^2_{L^2(\Omega)}. 
\end{split}
\end{equation}
We want to add together \eqref{comparison principle estimate 1} and \eqref{comparison principle estimate 2}. Let us first simplify the resulting terms. 

Let $u_1,u_2\in \BV(\R^n)$. Using the coarea formula for $f^\infty$, see for instance \cite[Remark 4.4]{AmarBellettini}, and the locality of the approximate gradient we have 
\begin{equation}
    \label{min max estimate for general functional}
    f(D\min(u_1,u_2))+f(D\max(u_1,u_2)) \leq f(Du_1)+f(Du_2).
\end{equation}
Hence, by the assumption on $g$ and $h$ we find that
\begin{equation}
\label{comparison principle estimate 3}
\begin{split}
    &f(D\min(u^\delta,v^\delta)(t))(\Omega)+f(D\max(u^\delta,v^\delta)(t))(\Omega) \\
    & \quad + \int_{\partial\Omega}\left( \lvert Tg(t)-T \min(u^\delta,v^\delta)(t)\rvert + \lvert Th(t)-T \max(u^\delta,v^\delta)(t)\rvert \right) f^\infty(\cdot, \nu_\Omega)\dH \\
    & \leq f(Du^\delta)(\Omega)+f(Dv^\delta)(\Omega) \\
    & \quad + \int_{\partial\Omega}\left( \lvert Tg(t)-T u^\delta(t)\rvert+\lvert Th(t)-T v^\delta\rvert \right) f^\infty(\cdot, \nu_\Omega)\dH, 
\end{split}
\end{equation}
for a.e. $t\in (0,T)$. For the time derivatives we have
\begin{align*}
    & \partial_t \min(u^\delta, v^\delta) (\min(u^\delta,v^\delta)-u)+ \partial_t \max(u^\delta,v^\delta)(\max(u^\delta,v^\delta)-v) \\
    & = \partial_t (u^\delta-v^\delta)_+ (u^\delta-v^\delta)_+ +\partial_t v^\delta(v^\delta-v)+\partial_t u^\delta(u^\delta-u) \\
    & \quad\quad + \partial_t (u^\delta-v^\delta)_+ (v^\delta-u^\delta+u-v) \\
    & \leq \frac{1}{2} \partial_t  (u^\delta-v^\delta)_+^2 + \partial_t (u^\delta-v^\delta)_+ (v^\delta-u^\delta+u-v),
\end{align*}
where in the second step we used \eqref{derivative of time mollification}. Adding \eqref{comparison principle estimate 1} and \eqref{comparison principle estimate 3} we arrive at
\begin{equation}
\label{comparison principle estimate 4}
\begin{split}
    & \int_0^\tau f(Du)(\Omega)+f(Dv)(\Omega)\dt \\
    & \quad + \iint_{\partial\Omega\times(0,\tau)}\left(\lvert Tg-Tu\rvert+\lvert Th-Tu\rvert \right) f^\infty(\cdot, \nu_\Omega)\dH\dt\\
    & \leq  \int_0^\tau  f(Du^\delta)(\Omega)+f(Dv^\delta)(\Omega)\dt \\ 
    & \quad + \iint_{\partial\Omega\times(0,\tau)}\left(\lvert Tg-T u^\delta\rvert +\lvert Th-T v^\delta\rvert \right) f^\infty(\cdot, \nu_\Omega)\dH\dt \\ 
    & \quad + \iint_{\Omega_\tau}\partial_t (u^\delta-v^\delta)_+ (v^\delta-u^\delta+u-v)\dx\dt \\
    & \quad + \frac{1}{2}\|(u-v)_+ (\tau)\|_{L^2(\Omega)}^2-\frac{1}{2}\|(u_{0,\epsilon}-v_{0,\epsilon})_+\|_{L^2(\Omega)}^2\\
    & \quad +\frac{1}{2}\|u_0-u_{0,\epsilon}+(u_{0,\epsilon}-v_{0,\epsilon})_+\|^2_{L^2(\Omega)} - \frac{1}{2}\|(\min(u^\delta,v^\delta)-u)(\tau)\|^2_{L^2(\Omega)} \\
    & \quad + \frac{1}{2}\|v_0-v_{0,\epsilon}-(u_{0,\epsilon}-v_{0,\epsilon})_+\|^2_{L^2(\Omega)} - \frac{1}{2}\|(\max(u^\delta,v^\delta)-v)(\tau)\|^2_{L^2(\Omega)}.
\end{split}
\end{equation}

For the remaining term containing the time derivative we see that
\begin{align*}
    & \limsup_{\delta\to 0} \iint_{\Omega_\tau}\partial_t (u^\delta-v^\delta)_+ (v^\delta-u^\delta+u-v)\dx\dt \\
    & \leq  \limsup_{\delta\to 0} \iint_{\Omega_\tau}\frac{1}{\delta}(v^\delta-u^\delta+u-v)^2\dx\dt \\
    & \leq  \limsup_{\delta\to 0} \frac{2}{\delta} \iint_{\Omega_\tau} (u^\delta-u)^2 +(v^\delta-v)^2 \dx\dt \\
    & \leq \|u_{0,\epsilon}-u_0\|_{L^2(\Omega)}^2 + \|v_{0,\epsilon}-v_0\|_{L^2(\Omega)}^2,
\end{align*}
where in the first step we used \eqref{derivative of time mollification} and in the final step we applied \Cref{lemma for comp principle}. Combining \Cref{lemma: time moll conv area strictly} and \Cref{reshetnyak continuity}, we find, after choosing some sequence $\delta_i>0$, $i\in\mathbb N$, with $\delta_i\to 0$ as $i\to\infty$, that $f(Du^\delta)(\Omega)\to f(Du)(\Omega)$ as $\delta\to 0$. Hence, letting $\delta\to 0$ in \eqref{comparison principle estimate 4} we absorb terms and arrive at
\begin{align*}
     \| (u-v)_+\|_{L^2(\Omega)}^2 &  \leq -\frac{1}{2}\|(u_{0,\epsilon}-v_{0,\epsilon})_+\|_{L^2(\Omega)}^2 +\frac{1}{2}\|u_0-u_{0,\epsilon}+(u_{0,\epsilon}-v_{0,\epsilon})_+\|^2_{L^2(\Omega)}  \\
    & \quad + \frac{1}{2}\|v_0-v_{0,\epsilon}-(u_{0,\epsilon}-v_{0,\epsilon})_+\|^2_{L^2(\Omega)}.
\end{align*}
Finally, letting $\epsilon\to 0$ completes the proof.
\end{proof}

\section{Weak solutions}
\label{section: weak solutions}
In this section we begin to study the equivalence of weak and variational solutions. The first result establishes the equivalence of our definition of weak solution with the definition based on the Anzellotti pairing, when the pairing is applicable. Note that if $f^*(x,z(x,t))<\infty$ then $z(x,t)\cdot \xi \leq f^\infty(x,\xi)$.
\begin{proposition}
    \label{equivalences for boundary condition}
    Let $u,g \in L^\infty(0,T;L^2(\Omega))$, with $\partial_t u, \partial_t g \in L^1(0,T;L^2(\Omega))$. Furthermore, let $z\in L^\infty(\Omega_T,\R^n)$ with $z(x,t)\cdot \xi \leq f^\infty(x,\xi)$ for all $((x,t),\xi)\in \overline{\Omega}_T\times\R^n$. Then the following are equivalent:
    \begin{enumerate}[(i)]
        \item The condition \eqref{divergence condition} holds for $u,g,z$.
        \item 
        It holds that
        \begin{align*}
             & \iint_{\Omega_T}z\cdot \nabla \varphi\dx\dt-\iint_{\Omega_T} u \partial_t \varphi \dx\dt \\
             & = \iint_{\partial\Omega\times(0,T)} \sign_0(Tg-Tu) T\varphi f^\infty(\cdot,\nu_\Omega)\dH\dt, 
        \end{align*}
        for every $\varphi\in L^1(0,T;W^{1,1}(\Omega))$, compactly supported in time with $\partial_t w\in L^1(0,T;L^2(\Omega))$, and such that for a.e. $t\in (0,T)$ and $\mathcal H^{n-1}$ a.e. $x\in\partial\Omega$, $Tg(x)=Tu(x)$ implies $T\varphi(x)=0$.
        \item $\div z = \partial_t u $ in the sense of distributions and 
        \begin{equation*}
            [z(t),\nu_\Omega]\in\sign(Tg(t)-Tu(t))f^\infty(\cdot, \nu_\Omega) \text{ $\mathcal H^{n-1}$-a.e.},  
        \end{equation*}
        for a.e. $t\in (0,T)$.
    \end{enumerate}
\end{proposition}
\begin{proof}
Let us first see that $(i)\implies (ii)$. Let $\varphi$ satisfy the assumptions in $(ii)$. Let $h>0$ and apply $h\varphi$ as a test function in \eqref{divergence condition}. Dividing by $h$ and letting $h\to 0$ shows
\begin{align*}    
    &\iint_{\Omega_T} z\cdot \nabla\varphi - u\partial_t \varphi \dx\dt \\
    &\leq \lim_{h\to 0} \iint_{\partial\Omega\times(0,T)} \left(\lvert T\varphi+Tg-Tu\rvert - \lvert Tg-Tu\rvert \right)/h f^\infty(\cdot, \nu_\Omega) \, d \mathcal{H}^{n-1}\dt \\
    & = \iint_{\partial\Omega\times(0,T)} \sign_0(Tg-Tu) T\varphi f^\infty(\cdot, \nu_\Omega) \, d \mathcal{H}^{n-1}\dt.
\end{align*}
Here, the assumption on the trace is necessary to allow us to pass to the limit. Applying the above inequality for $-\varphi$ and arguing by linearity shows $(ii)$.

Suppose then that $(ii)$ holds, and let us show that $(ii)\implies(i)$. Using \Cref{lemma: strict approx in parabolic bv}, we find $\Tilde u,\Tilde g\in L^1(0,T;W^{1,1}(\Omega))$ with $T\Tilde u= Tu$ and $T\Tilde g=Tg$, and $\partial_t \Tilde u, \partial_t \Tilde g\in L^1(0,T;L^2(\Omega))$. Let $\epsilon>0$ and denote
\begin{equation*}
    \Omega_\epsilon=\{x\in\Omega: \dist(x,\partial\Omega)>\epsilon\}.
\end{equation*}
Furthermore, let
\begin{equation*}
\eta_\epsilon(x)=\frac{\dist(x,\R^n\backslash \Omega)}{\dist(x,\R^n\backslash\Omega)+\dist(x,\Omega_\epsilon)},
\end{equation*} 
and note that $0\leq \eta_\epsilon \leq 1$, $\eta_\epsilon=0$ on $\R^n\backslash\Omega$, $\eta_\epsilon=1$ on $\Omega_\epsilon$. Let $\varphi$ be as in \eqref{divergence condition} and denote $\varphi_\epsilon=\eta_\epsilon \varphi + (1-\eta_\epsilon)(\Tilde u-\Tilde g)$. Testing $(ii)$ with $\varphi_\epsilon$ shows
\begin{equation}
\label{estimate for div cond from sign trace}
\begin{split}
    & \iint_{\Omega_T} \eta_\epsilon z\cdot \nabla \varphi +(1-\eta_\epsilon)\nabla (\Tilde u-\Tilde g) \dx\dt - \iint_{\Omega_T} u(\eta_\epsilon\partial_t\varphi + (1-\eta_\epsilon)\partial_t (\Tilde u-\Tilde g))\dx\dt \\
    & = \iint_{\Omega_T} (\Tilde u-\Tilde g-\varphi)z\cdot \nabla \eta \dx\dt - \iint_{\partial\Omega\times(0,T)} \lvert Tu-Tg\rvert f^\infty(\cdot, \nu_\Omega)\dH\dt  \\
    & \leq \iint_{\Omega_T} \lvert \Tilde u-\Tilde g-\varphi\rvert f^\infty(x,\nabla\eta_\epsilon) \dx\dt - \iint_{\partial\Omega\times(0,T)} \lvert Tu-Tg\rvert f^\infty(\cdot, \nu_\Omega) \dH\dt. 
\end{split}
\end{equation}

Let us examine the first term of the right hand side as $\epsilon\to 0$. For $t\in (0,T)$, define the $\R^n$ valued measures $\mu_\epsilon(t)$, $\mu(t)$ by
\begin{equation*}
    \mu_\epsilon(t)(B)=D\left( (1-\eta_\epsilon)(\Tilde u(t)-\Tilde g(t)-\varphi(t))\right)(B \cap \Omega),
\end{equation*}
and 
\begin{equation*}
    \mu(t)(B)= \int_{B\cap\partial\Omega} (T\varphi(t)+Tg(t) -Tu(t))\nu_\Omega\dH.
\end{equation*}
Then \Cref{cut off trace lemma} shows that
\begin{equation}
    \label{area strict conv of test functions}
    \lvert(\mu_\epsilon(t), \mathcal L^n)\rvert(\overline{\Omega})\to \lvert(\mu(t), \mathcal L^n)\rvert(\overline{\Omega}) \text{ as } \epsilon\to 0.
\end{equation}
Hence applying \Cref{reshetnyak continuity}, we find that
\begin{align*}
    \lim_{\epsilon\to 0}\iint_{\Omega_T} \lvert \Tilde u-\Tilde g-\varphi\rvert f^\infty(x,\nabla\eta_\epsilon) \dx\dt &  = \int_0^T \lim_{\epsilon\to o} f(\mu_\epsilon(t))\dt  \\
    & = \iint_{\partial\Omega\times(0,T)} \lvert Tg-Tu+T\varphi\rvert f^\infty(\cdot,\nu_\Omega) \dH\dt.    
\end{align*}
Here we are using the second part of \Cref{cut off trace lemma} and the dominated convergence theorem move the limit inside the integral. Thus letting $\epsilon\to 0$ in \eqref{estimate for div cond from sign trace} shows $(i)$.

To see that $(ii)\implies (iii)$, we note first that similarly as before, $(ii)$ implies $\div z = \partial_t u$ in the sense of distributions. Let again $\Tilde u,\Tilde g\in L^1(0,T;W^{1,1}(\Omega))$ with $T\Tilde u=Tu$, $T\Tilde g=Tg$. From $(ii)$ we have 
\begin{align*}
    & \int_{\Omega} z(t)\cdot \nabla (\Tilde g(t)-\Tilde u(t))\dx+\int_{\Omega} \div z(t) (\Tilde g(t)-\Tilde u(t))  \dx\dt \\
    & = \int_{\partial\Omega} \lvert Tg(t)-Tu(t)\rvert f^\infty(\cdot,\nu_\Omega)\dH,
\end{align*}
for a.e. $t\in (0,T)$. On the other hand, applying the Gauss-Green formula \eqref{eq: Gauss Green in W11} we find that
\begin{align*}
    & \int_{\Omega} z(t)\cdot \nabla (\Tilde g(t)-\Tilde u(t))\dx+\int_{\Omega} \div z(t) (\Tilde g(t)-\Tilde u(t))  \dx\dt \\
    & = \int_{\partial\Omega} [z(t),\nu_\Omega](Tg(t)-Tu(t))\dH,
\end{align*}
so that 
\begin{equation}
    \label{eq: eq of integrals for sign function}
    \int_{\partial\Omega} \lvert Tg(t)-Tu(t)\rvert f^\infty(\cdot,\nu_\Omega)\dH = \int_{\partial\Omega} [z(t),\nu_\Omega](Tg(t)-Tu(t))\dH.
\end{equation}
By \cite[Lemma 3.5]{GornyMazonBook} we have
\begin{equation*}
    [z(t),\nu_\Omega]\leq f^\infty(\cdot,\nu_\Omega) \text{ $\mathcal H^{n-1}$-a.e.},
\end{equation*}
which together with \eqref{eq: eq of integrals for sign function} implies 
\begin{equation*}
    [z(t),\nu_\Omega](Tg(t)-Tu(t))= \lvert Tg(t)-Tu(t)\rvert f^\infty(\cdot,\nu_\Omega) \text{ $\mathcal H^{n-1}$-a.e.},
\end{equation*}
and hence 
 \begin{equation*}
    [z(t),\nu_\Omega]\in\sign(Tg(t)-Tu(t))f^\infty(\cdot, \nu_\Omega) \text{ $\mathcal H^{n-1}$-a.e.} 
\end{equation*}
The final implication, $(iii)\implies (ii)$ follows simply by applying Green's formula.
\end{proof}

We begin our study of weak solutions without the assumption that the time derivative is an $L^2$ function. In this setting, the following lemma will be very useful. The first part is a form of \eqref{divergence condition}, when the test functions are not compactly supported in time, and the second is a time mollified version of \eqref{divergence condition}.
\begin{lemma}
Suppose $u$ is a weak solution of \eqref{PDE}. Then 
\begin{equation}
\label{equation with boundary terms}
\begin{split}
    & \iint_{\Omega\times(t_1,t_2)} z \cdot \nabla \varphi - u \partial_t\varphi \dx\dt + \left[ \int_{\Omega} u\varphi\dx\right]_{t=t_1}^{t_2} \\
    & \leq \iint_{\partial\Omega\times(t_1,t_2)} \left(\lvert T\varphi+Tg-Tu\rvert - \lvert Tg-Tu\rvert \right) f^\infty(\cdot, \nu_\Omega) \, d \mathcal{H}^{n-1}\dt,
\end{split}
\end{equation}
for all $\varphi$ satisfying the conditions of \eqref{divergence condition}, and for a.e. $0<t_1<t_2<T$. Furthermore, for any $\delta>0$ the time mollified equation
\begin{equation}
\label{mollified equation}
\begin{split}
    & \iint_{\Omega_T} z^{\delta}\cdot \nabla\varphi + \partial_t u^{\delta}\varphi\dx\dt 
    - \iint_{\Omega_T} \delta^{-1}e^{-\frac{s}{\delta}} u_0\varphi(s)\dx\ds \\
    & \leq \int_{0}^T\int_0^t\int_{\partial\Omega} \delta^{-1}e^{-\frac{s}{\delta}} \lvert T(u-g)(t-s)-T\varphi(t)\rvert f^\infty(\cdot, \nu_\Omega) \, d\mathcal H^{n-1}\ds\dt \\
    & \quad - \iint_{\partial\Omega\times(0,T)} (1-e^{-\frac{t}{\delta}}) \lvert Tu-Tg\rvert f^\infty(\cdot, \nu_\Omega) \, d \mathcal{H}^{n-1}\dt,
\end{split}
\end{equation}
holds for $\varphi\in L^1(0,T;W^{1,1}(\Omega))\cap L^2(\Omega_T)$, compactly supported in time.
\end{lemma}
\begin{proof}
    Let us derive \eqref{mollified equation}. Suppose first that $\varphi\in L^1(0,T;W^{1,1}_0(\Omega))\cap L^2(\Omega_T)$ with $\partial_t\varphi\in L^2(\Omega_T)$. By \eqref{equation with boundary terms}, for a.e. $s\in (0,T)$ we have 
    \begin{align*}
         &\int_s^T \int_\Omega z(x,t-s)\cdot \nabla \varphi(x,t) - u(x,t-s) \partial_t\varphi(x,t) \dx\dt \\
         & = \int_0^{T-s}\int_\Omega z(x,t)\cdot \nabla \varphi(x,t+s) - u(x,t) \partial_t\varphi(x,t+s) \dx\dt \\
         & = -\int_\Omega u(x,T-s) \varphi(x,T) \dx + \int_\Omega u(x,0)\varphi(x,s)\dx \\
         & \quad + \int_0^{T-s}\int_{\partial\Omega} \lvert Tu(t)-Tg(t)- T\varphi(t+s)\rvert - \lvert Tu(t)-Tg(t)\rvert  \, d \mathcal{H}^{n-1}\dt\\
         & =  \int_\Omega u_0\varphi(x,s)\dx + \int_0^{T-s}\int_{\partial\Omega} \lvert Tu(t)-Tg(t)- T\varphi(t+s)\rvert - \lvert Tu(t)-Tg(t)\rvert  \, d \mathcal{H}^{n-1}\dt.
    \end{align*}
    Here we used \eqref{initial value condition} and the assumption that $\varphi$ is compactly supported in time to deduce that the first boundary term vanishes. Multiply the above equation by $\delta^{-1} e^{-\frac{s}{\delta}}$ and integrate in $s$ over $(0,T)$. Switching the order of integration we have 
    \begin{align*}
         &\int_0^{T}  \delta^{-1}e^{-\frac{s}{\delta}}\int_s^T \int_\Omega z(x,t-s)\cdot \nabla \varphi(x,t) - u(x,t-s) \partial_t\varphi(x,t) \dx\dt\ds \\
         & = \int_0^{T} \int_{0}^{t}\int_\Omega  \delta^{-1}e^{-\frac{s}{\delta}} z(x,t-s)\cdot \nabla \varphi(x,t) - u(x,t-s) \partial_t\varphi(x,t) \dx\ds\dt \\
         & = \int_0^{T} \int_{0}^{t}\int_\Omega  \delta^{-1}e^{\frac{s-t}{\delta}} z(x,s)\cdot \nabla \varphi(x,t) -  \delta^{-1}e^{\frac{s-t}{\delta}}u(x,s) \partial_t\varphi(x,t) \dx\ds\dt \\
         & = \int_0^T \int_\Omega z^{\delta}(x,t)\cdot \nabla\varphi(x,t)- u^{\delta}(x,t)\partial_t\varphi(x,t)\dx\dt .
    \end{align*}
    Similarly, 
    \begin{align*}
        & \int_{0}^s \delta^{-1}e^{-\frac{s}{\delta}}\int_0^{T-s}\int_{\partial\Omega} \lvert Tu(t)-Tg(t)- T\varphi(t+s)\rvert - \lvert Tu(t)-Tg(t)\rvert  \, d \mathcal{H}^{n-1}\dt\ds\\
        & = \int_{0}^s \delta^{-1}e^{-\frac{s}{\delta}}\int_s^T\int_{\partial\Omega} \lvert T(u-g)(t-s)- T\varphi(t)\rvert - \lvert T(u-g)(t-s)\rvert  \, d \mathcal{H}^{n-1}\dt\ds \\
        & = \int_{0}^T\int_0^t\int_{\partial\Omega} \delta^{-1}e^{-\frac{s}{\delta}}\left( \lvert T(u-g)(t-s)- T\varphi(t)\rvert - \lvert Tu(t)-Tg(t)\rvert \right) \, d \mathcal{H}^{n-1}\ds\dt.
    \end{align*}
    Integrating by parts we thus have \eqref{mollified equation} for $\varphi$ with $\partial_t\varphi\in L^2(\Omega_T)$. Arguing by density and using the fact that the trace is continuous with respect to strict convergence we obtain the full statement.
\end{proof}

We are ready to prove our first result on the equivalence of weak and variational solutions, by showing that weak solutions are variational solutions. Moreover, we introduce a third condition, which is immediately shown to be equivalent to being a weak solution. This condition will be useful when considering the stability of weak solutions.
\begin{theorem}
\label{equivalences part 1}
Let $u\in L^1_{w^*}(0,T;\BV(\Omega))\cap L^\infty(0,T;L^2(\Omega))$.
Consider the following conditions:
\begin{enumerate}[(i)]
    \item $u$ is a weak solution of \eqref{PDE}.
    \item There exists $z\in L^\infty(\Omega_T,\R^n)$ with $z\in \partial_\xi f(x,\nabla u)$ for a.e. $(x,t)\in\Omega_T$ such that
    \begin{equation}
    \label{intermediate condition}
    \begin{split}
        &\int_0^{\tau} f(Du)(\Omega) + \iint_{\partial\Omega\times(0,\tau)}\lvert Tu-Tg\rvert f^\infty(\cdot, \nu_\Omega) \, d \mathcal{H}^{n-1}\dt + \iint_{\Omega_\tau} f^*(x,z)\dx\dt \\
        & \leq \iint_{\Omega_\tau} z\cdot \nabla v \dx \dt + \iint_{\partial\Omega\times(0,\tau)}\lvert Tv-Tg\rvert f^\infty(\cdot, \nu_\Omega) \, d \mathcal{H}^{n-1}\dt\\
        & \quad + \iint_{\Omega_{\tau}} \partial_t v (v-u)\dx\dt+\frac{1}{2}\|v(0)-u_0\|^2_{L^2(\Omega)}-\frac{1}{2}\|v(\tau)-u(\tau)\|_{L^2(\Omega)}^2,
    \end{split}
    \end{equation}
    holds for a.e. $\tau\in (0,T)$, and for every $v\in L^1(0,T;W^{1,1}(\Omega))\cap C([0,T],L^2(\Omega))$ with $\partial_t v\in L^2(\Omega_T)$
    \item $u$ is a variational solution of $\eqref{PDE}$.
\end{enumerate}
Then $(i)$ and $(ii)$ are equivalent, and imply $(iii)$.
\end{theorem}
\begin{proof}
Let us start by showing $(i)\implies (ii)$. Let $v$ be a comparison map as in $(ii)$. Let $\delta>0$, and let $u_i\in L^1(0,T;W^{1,1}(\Omega))$, $i\in\mathbb N$, be a sequence converging to $u^\delta$ as in \Cref{lemma: strict approx in parabolic bv}. Here, $u^\delta$ is the time mollification of $u$ as in \eqref{time mollification}, with $w=0$. Let $\zeta\in C_0^\infty(0,T)$. Applying $\zeta (u_i-v)$ as a test function in \eqref{mollified equation} gives
\begin{align*}
    & \iint_{\Omega_T} z^{\delta}\cdot \nabla(\zeta ( u_i-v)) \dx\dt+ \iint_{\Omega_T} \delta^{-1}e^{-\frac{s}{\delta}} u_0\zeta(s)(u_i(s)-v(s))\dx\ds \\
    & \quad +\iint \partial_t u^{\delta} \zeta (u_i-v)\dx\dt \\
    & \leq \int_{0}^T\int_0^t\int_{\partial\Omega} \delta^{-1}e^{-\frac{s}{\delta}} \zeta \lvert T(u-g)(t-s)- T(u^\delta-v)(t)\rvert f^\infty(\cdot, \nu_\Omega) \,d\mathcal H ^{n-1}\ds\dt \\
    & \quad - \iint_{\partial\Omega\times(0,T)} (1-e^{-\frac{t}{\delta}}) \zeta\lvert Tu-Tg\rvert f^\infty(\cdot, \nu_\Omega) \, d \mathcal{H}^{n-1}\dt.
\end{align*}
Rearranging the terms we have
\begin{equation}
\label{inserted equation 1}
\begin{split}
    & \iint_{\Omega_T} \zeta z^{\delta}\cdot \nabla u_i \dx\dt + \iint_{\Omega_T} \delta^{-1}e^{-\frac{s}{\delta}}u_0\zeta(s)(u_i(s)-v(s))\dx\ds \\
    & \leq \int_{0}^T\int_0^t\int_{\partial\Omega} \delta^{-1}e^{-\frac{s}{\delta}} \zeta \lvert T(u-g)(t-s)- T(u^\delta-v)(t)\rvert f^\infty(\cdot, \nu_\Omega)\,d\mathcal H ^{n-1}\ds\dt \\ 
    & \quad +\iint_{\Omega_T} \zeta  z^{\delta}\cdot \nabla v \dx\dt -\iint_{\Omega_T} \partial_t u^{\delta} \zeta (u_i-v)\dx\dt \\
    & \quad - \iint_{\partial\Omega\times(0,T)} (1-e^{-\frac{t}{\delta}})\zeta \lvert Tu-Tg\rvert f^\infty(\cdot, \nu_\Omega) \dH\dt.
\end{split}
\end{equation}
We want to pass to the limit $i\to\infty$ in the above. We begin by concidering the left hand side. Let $\eta\in C_0^\infty(\Omega)$ with $0\leq\eta\leq 1$. To estimate the left hand side of \eqref{inserted equation 1} from below we see that
    \begin{equation}  
        \iint_{\Omega_T} \eta\zeta z^\delta \cdot \nabla u_i\dx\dt - \iint_{\Omega_T} (1-\eta)\zeta \|z\|_\infty\left( \lvert\nabla u_i\rvert \right) \dx\dt \leq  \iint_{\Omega_T} \zeta z^{\delta}\cdot \nabla u_i \dx\dt.\label{eq: approximating with smooth func for applying pde}
    \end{equation}
    On the other hand, testing \eqref{mollified equation} with $\zeta\eta u_i$ gives
    \begin{align*}
        & \lim_{i\to\infty} \iint_{\Omega_T} \zeta \eta  z^{\delta}\cdot \nabla u_i \dx\dt  \\
        & = - \iint_{\Omega_T} \partial_t u^\delta \zeta\eta u^\delta \dx\dt- \iint_{\Omega_T} \delta^{-1}e^{-\frac{s}{\delta}} u_0\eta\zeta(s)u^\delta(s)\dx\ds \\
        & \quad - \iint_{\Omega_T} u^\delta \zeta z^{\delta}\cdot\nabla\eta\dx\dt\\
        & = \frac{1}{2}\iint_{\Omega_T} \left(u^\delta\right)^2 \partial_t \zeta\eta \dx\dt - \iint_{\Omega_T} \delta^{-1}e^{-\frac{s}{\delta}} u_0\eta\zeta(s)u^\delta(s)\dx\ds \\
        & \quad - \iint_{\Omega_T} u^\delta \zeta z^{\delta}\cdot\nabla\eta\dx\dt.
    \end{align*}
    Note that equality holds here because the test function is compactly supported. The second term on the right hand side of the above equation will converge to $0$ as $\delta\to 0$, as $u\in L^\infty(0,T;L^2(\Omega))$ and $\zeta$ is compactly supported. This implies
    \begin{equation}
    \label{eq: using measure condition}
    \begin{split}
        \lim_{\delta\to 0}\lim_{i\to\infty} \iint_{\Omega_T} \zeta \eta  z^{\delta}\cdot \nabla u_i \dx\dt & =\frac{1}{2}\iint_{\Omega_T} u^2 \partial_t \zeta\eta \dx\dt - \iint_{\Omega_T} u \zeta z\cdot\nabla\eta\dx\dt \\
        & = \iint_{\Omega_T} \eta\zeta\,d  f(Du)\dt + \iint_{\Omega_T} \eta\zeta f^*(\cdot,z)\dx\dt, 
    \end{split}
    \end{equation}
    where in the second step we used \eqref{pairing condition}. Combining \eqref{eq: using measure condition} and \eqref{eq: approximating with smooth func for applying pde} we have shown
    \begin{align*}
        & \iint_{\Omega_T} \eta\zeta\,d  f(Du)\dt + \iint_{\Omega_T} \eta\zeta f^*(\cdot,z)\dx\dt \\
        & \leq  \limsup_{\delta\to 0}\limsup_{i\to\infty}\iint_{\Omega_T} \zeta z^{\delta}\cdot \nabla u_i \dx\dt + \|z\|_\infty\iint_{\Omega_T} (1-\eta)\zeta  \,d \|Du\|\dt,
    \end{align*}
    for any $\eta\in C_0^\infty(\Omega)$ with $0\leq \eta\leq 1$. This implies 
    \begin{align*}
        & \int_{0}^T \zeta f(Du)(\Omega)\dt + \iint_{\Omega_T}\zeta f^*(\cdot,z)\dx\dt \\
        & \leq \limsup_{\delta\to 0}\limsup_{i\to\infty}\iint_{\Omega_T} \zeta z^{\delta}\cdot \nabla u_i \dx\dt .
    \end{align*}
    We pass to the limit in \eqref{inserted equation 1} to conclude
    \begin{equation}
    \label{inserted equation 2}
    \begin{split}
        & \int_0^T \zeta f(Du)(\Omega)\dt + \iint_{\Omega_T}\zeta f^*(\cdot,z)\dx\dt \\
        & \leq \limsup_{\delta\to 0} \limsup_{i\to\infty} \iint_{\Omega_T} \zeta z^{\delta}\cdot \nabla u_i \dx\dt \\
        & \leq \limsup_{\delta\to 0} \Bigg( \iint_{\Omega_T} \zeta  z^{\delta}\cdot \nabla v \dx\dt -\iint_{\Omega_T} \partial_t u^{\delta} \zeta (u^\delta-v)\dx\dt \\
        & \qquad+ \int_{0}^T\int_0^t\int_{\partial\Omega} \delta^{-1}e^{-\frac{s}{\delta}} \zeta \lvert T(u-g)(t-s)- T(u^\delta-v)(t))f^\infty(\cdot, \nu_\Omega)\rvert\ds\dt\\
        & \qquad - \iint_{\partial\Omega\times(0,T)} (1-e^{-\frac{t}{\delta}}) \lvert Tu-Tg\rvert f^\infty(\cdot, \nu_\Omega)  \, d \mathcal{H}^{n-1}\dt\Bigg).
    \end{split}
    \end{equation}
    Here, we used that the second term on the left hand side of \eqref{inserted equation 1} will converge to $0$, similarly as before. We calculate the term with the time derivatives as
    \begin{align*}
        & -\iint_{\Omega_T} \partial_t u^{\delta} \zeta (u^\delta-v)\dx\dt \\
        & = \iint_{\Omega_T} \zeta \partial_t v (v-u^\delta)\dx\dt-\iint_{\Omega_T} \partial_t (u^{\delta}-v) \zeta (u^\delta-v)\dx\dt \\
        & = \iint_{\Omega_T} \zeta \partial_t v (v-u^\delta)\dx\dt +\frac{1}{2}\iint_{\Omega_T} (u^{\delta}-v)^2 \partial_t\zeta \dx\dt.
    \end{align*} 
    By \Cref{mollification of trace terms} we have
    \begin{align*}
        & \lim_{\delta\to 0} \int_{0}^T\int_0^t\int_{\partial\Omega} \delta^{-1}e^{-\frac{s}{\delta}} \zeta\lvert T(u-g)(t-s)- T(u^\delta-v)(t))\rvert f^\infty(\cdot, \nu_\Omega)\dH\ds\dt\\
        & \leq \int_{0}^T\zeta f^\infty(\cdot, \nu_\Omega)\lvert Tg-Tv\rvert f^\infty(\cdot, \nu_\Omega) \, d \mathcal{H}^{n-1}\dt.
    \end{align*}
    We have shown that \eqref{inserted equation 2} becomes
    \begin{align*}
        & \int_0^T \zeta f(Du)(\Omega)\dt + \iint_{\partial\Omega\times(0,T)} \zeta \lvert Tu-Tg\rvert \dH\dt+ \iint_{\Omega_T} \zeta f^*(\cdot,z)\dx\dt\\
        & \leq  \iint_{\Omega_T} \zeta  z\cdot \nabla v \dx\dt + \iint_{\partial\Omega\times(0,T)} \zeta \lvert Tv-Tg \rvert f^\infty(\cdot, \nu_\Omega) \, d \mathcal{H}^{n-1}\dt \\
        & \quad +\iint_{\Omega_T} \zeta \partial_t v (v-u)\dx\dt + \frac{1}{2}\iint_{\Omega_T} (u-v)^2 \partial_t\zeta \dx\dt.  
    \end{align*}
    Let $\tau\in (0,T)$, $h>0$. Writing the above equation for $\zeta=\zeta_h$ given by
    \begin{equation}
    \label{eq: cut off in time}
        \zeta_h(t)
        =\begin{cases}
            \frac{1}{h}(t-h) ,&\quad 0<t < h,\\
            1,&\quad h\leq t\leq \tau,\\
            -\frac{1}{h}(t-\tau-h),&\quad \tau<t\leq \tau+h,\\
            0,&\quad\text{otherwise}.
        \end{cases}
    \end{equation}
    we have
    \begin{align*}
        & \int_{h}^{\tau}  f(Du)(\Omega)\dt + \iint_{\partial\Omega\times(h,\tau)} \lvert Tu-Tg\rvert \dH\dt + \iint_{\Omega\times(h,\tau)} f^*(\cdot,z)\dx\dt \\
        & \leq \iint_{\Omega_T} \zeta_h  z\cdot \nabla v \dx\dt+\iint_{\partial\Omega\times(0,T)} \zeta_h \lvert Tv-Tg \rvert f^\infty(\cdot, \nu_\Omega) \, d \mathcal{H}^{n-1}\dt \\
        & \quad + \iint_{\Omega_T} \partial_t v\zeta_h (v-u)\dx\dt+ \frac{1}{h}\int_{h}^{2h} \int_\Omega (u-v)^2 \dx\dt-\frac{1}{h}\int_{\tau}^{\tau+h} \int_\Omega (u-v)^2 \dx\dt.
    \end{align*}
    Letting $h\to 0$ shows that 
    \begin{align*}
        &\int_0^{\tau}  f(Du)(\Omega)\dt + \iint_{\partial\Omega\times(0,\tau)} \lvert Tu-Tg\rvert \dH\dt+\iint_{\Omega_\tau} f^*(\cdot,z)\dx\dt\\
        & \leq \iint_{\Omega_\tau} z\cdot v \dx \dt + \iint_{\partial\Omega\times(0,\tau)} \lvert Tv-Tg \rvert f^\infty(\cdot, \nu_\Omega)\, d \mathcal{H}^{n-1}\dt\\
        & \quad + \iint_{\Omega_{\tau}} \partial_t v (v-u)\dx\dt+\frac{1}{2}\|v(0)-u_0\|^2_{L^2(\Omega)}-\frac{1}{2}\|(v-u)(\tau)\|^2_{L^2(\Omega)},
    \end{align*}
    for a.e. $\tau\in (0,T)$. This finishes the proof of $(i)\implies (ii)$.

    Let us then prove that $(ii)\implies (i)$. The initial condition \eqref{initial value condition} can be obtained from \eqref{intermediate condition} by a similar argument as in the proof of \Cref{var sol initial value}. We begin by showing \eqref{divergence condition}. Let $\delta>0$ and let $u^\delta$ be the time mollification of $u$ as in \eqref{time mollification}, with $w\in \BV(\Omega)\cap L^2(\Omega)$. Let $u_i\in L^1(0,T;W^{1,1}(\Omega))$, $i\in\mathbb N$, be a sequence converging to $u^\delta$  as in \Cref{lemma: strict approx in parabolic bv}. Testing \eqref{intermediate condition} with $v=u_i-\varphi$, where $\varphi$ is as in \eqref{divergence condition}, gives
    \begin{align*}
        &\int_0^{T}  f(Du)(\Omega)\dt + \iint_{\partial\Omega\times(0,T)} \lvert Tu-Tg\rvert f^\infty(\cdot, \nu_\Omega)\dH\dt + \iint_{\Omega_T} f^*(x,z)\dx\dt\\
        & \leq \iint_{\Omega_T} z\cdot \nabla(u_i+e^{-\frac{t}{\delta}}u_{0,\epsilon} -\varphi) \dx \dt \\
        & \quad + \iint_{\partial\Omega\times(0,T)} \lvert Tu^\delta-T\varphi-Tg\rvert f^\infty(\cdot, \nu_\Omega)\, d \mathcal{H}^{n-1}\dt\\
        & \quad + \iint_{\Omega_T} \partial_t (u_i-\varphi) (u_i-\varphi-u)\dx\dt+\frac{1}{2}\|u_i(0)-u_0\|^2_{L^2(\Omega)}. 
    \end{align*}
    Estimating $z\cdot \nabla u_i\leq f(x,\nabla u_i)+f^*(x,z)$ and letting $i\to\infty$ we have
    \begin{equation}
    \label{estimate for interm to weak sol}
    \begin{split}
        &\int_0^{T}  f(Du)(\Omega)\dt + \iint_{\partial\Omega\times(0,T)} \lvert Tu-Tg\rvert f^\infty(\cdot, \nu_\Omega)\dH\dt+\iint_{\Omega_T} f^*(x,z)\dx\dt\\
        & \leq \int_0^T  f(Du^\delta)(\Omega) \dt +  \iint_{\Omega_T} f^*(x,z)-z\cdot\nabla \varphi\dx\dt \\
        & \quad + \iint_{\partial\Omega\times(0,T)} \lvert Tu^\delta-T\varphi-Tg \rvert f^\infty(\cdot, \nu_\Omega)\, d \mathcal{H}^{n-1}\dt \\
        & \quad + \iint_{\Omega_T} \partial_t (u^\delta-\varphi) (u^\delta-\varphi-u)\dx\dt+\frac{1}{2}\|w-u_0\|^2_{L^2(\Omega)}. 
    \end{split}
    \end{equation}
    For the time derivative we integrate by parts and use \eqref{derivative of time mollification} to conclude
    \begin{align*}
        & \limsup_{\delta\to 0}\iint_{\Omega_T} \partial_t (u^\delta-\varphi) (u^\delta-\varphi-u)\dx\dt \\
        & \leq \limsup_{\delta\to 0} \iint_{\Omega_T} u^\delta \partial_t\varphi -\partial_t\varphi (u^\delta-u)\dx\dt = \iint_{\Omega_T} u\partial_t\varphi \dx\dt.
    \end{align*}
    Letting $\delta\to 0$ we conclude
    \begin{align*}
        &\int_0^{T}  f(Du)(\Omega)\dt + \iint_{\partial\Omega\times(0,T)} \lvert Tu-Tg\rvert f^\infty(\cdot, \nu_\Omega)\dH\dt\\
        & \leq \int_0^T  f(Du) \dx \dt  - \iint_{\Omega_T}z\cdot\nabla \varphi\dx\dt \\
        & \quad + \iint_{\partial\Omega\times(0,T)} \lvert Tu-Tg-T\varphi\rvert f^\infty(\cdot, \nu_\Omega)\dH\dt\\
        & \quad + \iint_{\Omega_T} u \partial_t \varphi \dx\dt+\frac{1}{2}\|w-u_0\|^2_{L^2(\Omega)}.
    \end{align*}
    Here we applied \Cref{lemma: time moll conv area strictly} and \Cref{reshetnyak continuity} after choosing a sequence $\delta_i>0$, $i\in\mathbb N$, with $\delta_i\to 0$ as $i\to\infty$. Finally letting $w\to u_0$ in $L^2(\Omega)$ gives \eqref{divergence condition}. To show \eqref{pairing condition}, let $\varphi\in C_0^\infty(\Omega_T)$ with $\|\varphi\|_\infty\leq 1$. Let $\delta>0$, $w\in\BV(\Omega)\cap L^2(\Omega)$ and let $u_i\in L^1(0,T;W^{1,1}(\Omega))$, $i\in\mathbb N$, be a sequence converging to $u^\delta$ as in \Cref{lemma: strict approx in parabolic bv}. Apply $u_i-\varphi u_i$ as a test function in \eqref{intermediate condition} to obtain
    \begin{align*}
        &\int_0^{T}  f(Du)(\Omega)\dt + \iint_{\partial\Omega\times(0,T)} \lvert Tu-Tg\rvert f^\infty(\cdot, \nu_\Omega)\dH\dt + \iint_{\Omega_T} f^*(x,z)\dx\dt\\
        & \leq -\iint_{\Omega_T} u_i z\cdot \nabla\varphi  \dx \dt + \iint_{\Omega_T} (1-\varphi) (f(x,\nabla u_i) + f^*(x,z))\dx \dt\\
        & \quad + \iint_{\partial\Omega\times(0,T)} \lvert Tu^\delta-Tg\rvert f^\infty(\cdot, \nu_\Omega)\, d \mathcal{H}^{n-1}\dt\\
        & \quad + \iint_{\Omega_T} \partial_t (u_i-\varphi u_i)(u_i-\varphi u_i -u)\dx\dt+\frac{1}{2}\|u_i(0)-u_0\|^2_{L^2(\Omega)},
    \end{align*}
where we used the Fenchel inequality to estimate
\begin{align*}
    z\cdot \nabla((1-\varphi)u_i) & =-z\cdot \nabla \varphi u_i + (1-\varphi)z\cdot \nabla u_i \\
    & \leq -z\cdot \nabla \varphi u_i + (1-\varphi)(f(x,\nabla u_i)+f^*(x,z).
\end{align*}
Letting $i\to \infty$ we conclude that
\begin{equation}
\begin{split}
\label{estimate for measure condition}
    &\int_0^{T}  f(Du)(\Omega)\dt + \iint_{\partial\Omega\times(0,T)} \lvert Tu-Tg\rvert f^\infty(\cdot, \nu_\Omega)\dH\dt+\iint_{\Omega_T} f^*(x,z)\dx\dt\\
    & \leq -\iint_{\Omega_T} u^\delta z\cdot \nabla\varphi  \dx \dt + \iint_{\Omega_T} (1-\varphi)\,d  f(Du^\delta)\dt + \iint_{\Omega_T} (1-\varphi) f^*(x,z)\dx \dt\\
    & \quad + \iint_{\partial\Omega\times(0,T)} \lvert Tu^\delta-Tg\rvert f^\infty(\cdot, \nu_\Omega)\, d \mathcal{H}^{n-1}\dt\\
    & \quad + \iint_{\Omega_T} \partial_t (u^\delta-\varphi u^\delta)(u^\delta-\varphi u^\delta-u)\dx\dt+\frac{1}{2}\|w-u_0\|^2_{L^2(\Omega)}. 
\end{split}
\end{equation}
For the term containing the time derivative we have
\begin{equation}
\label{eq: time derivative simplifies}
\begin{split}
    & \iint_{\Omega_T} \partial_t (u^\delta-\varphi u^\delta)(u^\delta-u-\varphi u^\delta)\dx\dt \\
    & = \iint_{\Omega_T} \partial_t u^\delta (u^\delta-u-\varphi u^\delta)-\partial_t(\varphi u^\delta)(u^\delta-u-\varphi u^\delta) \dx\dt \\
    & = \iint_{\Omega_T} (1-\varphi)\partial_t u^\delta (u^\delta-u) - \frac{1}{2}(u^\delta)^2\partial_t\varphi-\partial_t\varphi u^\delta(u^\delta-u) \dx\dt \\
    & \leq \frac{1}{2}\iint_{\Omega_T} u^2 \partial_t\varphi\dx\dt-\partial_t\varphi u^\delta(u^\delta-u)\dx\dt, 
\end{split}
\end{equation}
where in the second step we integrated by parts and in the third step we used \eqref{derivative of time mollification}. Hence we have shown that
\begin{equation*}
    \limsup_{\delta\to 0} \iint_{\Omega_T} \partial_t (u^\delta-\varphi u^\delta)(u^\delta-u-\varphi u^\delta)\dx\dt\leq \frac{1}{2}\iint_{\Omega_T} u^2\partial_t\varphi\dx\dt.
\end{equation*}
Thus, letting $\delta\to 0$ in \eqref{estimate for measure condition} and absorbing terms gives
\begin{align*}
    &\iint_{\Omega_T} \varphi\,d  f(Du)(\Omega)\dt + \iint_{\Omega_T} \varphi f^*(x,z)\dx\dt\\
    & \leq -\iint_{\Omega_T} u^\delta z\cdot \nabla\varphi  \dx \dt  + \frac{1}{2}\iint_{\Omega_T} u^2 \partial_t\varphi\dx\dt+\frac{1}{2}\|w-u_0\|^2_{L^2(\Omega)}.
\end{align*}
Letting $w\to u_0$ in $L^2(\Omega)$ and arguing by linearity shows \eqref{pairing condition}.

Let us finally see that $(ii)\implies (iii)$. Let $v\in L^1_{w^*}(0,T;\BV(\Omega))\cap C([0,T],L^2(\Omega))$ with $\partial_t v\in L^2(\Omega_T)$. Let $v_i\in L^1(0,T;W^{1,1}(\Omega))$, $i\in\mathbb N $,  be a sequence converging to $v$ as in \Cref{lemma: strict approx in parabolic bv}. We may additionally assume $v_i(t)\to v(t)$ in $L^2(\Omega)$ as $i\to\infty$ for a.e. $t\in (0,T)$. Testing \eqref{intermediate condition} with $v_i$ and using the Fenchel inequality to estimate 
\begin{equation*}
    z\cdot \nabla v_i\leq f(x,\nabla v_i)+f^*(x,z)    
\end{equation*}
shows that
\begin{align*}
    & \int_0^\tau f(Du)(\Omega)\dt + \iint_{\partial\Omega\times(0,\tau)}\lvert Tg-Tu\rvert f^\infty(\cdot, \nu_\Omega)\dH\dt\\
    & \leq  \int_0^\tau  f(Dv_i)(\Omega)\dt + \iint_{\partial\Omega\times(0,\tau)}\lvert Tg-Tv\rvert f^\infty(\cdot, \nu_\Omega)\dH\dt \\ 
    & \quad +\iint_{\Omega_\tau} \partial_t v_i (v_i-u)\dx\dt +\frac{1}{2}\|v(0)-u_0\|^2_{L^2(\Omega)} - \frac{1}{2}\|(v_i-u)(\tau)\|^2_{L^2(\Omega)}.
\end{align*}
We let $i\to\infty$ in the above and use \Cref{reshetnyak continuity} to conclude that \eqref{var inequality} holds for a.e. $\tau\in (0,T)$, which completes the proof.
\end{proof}
\begin{remark}
\label{remark on subgradient in def of sol}
In the definition of a weak solution or \eqref{intermediate condition}, it is equivalent to assume $z\in \partial_\xi f(x,\nabla u)$ or $(x,t)\mapsto f^*(x,z(x,t))\in L^1(\Omega_T)$. To see this assume the latter holds. Let $\delta>0$, $u^\delta$ be as in \eqref{time mollification} with $w=0$ and let $u_i$ be a sequence converging to $u^\delta$ as in \Cref{lemma: strict approx in parabolic bv}. Let $\varphi\in C_0^\infty(\Omega_T)$ with $\varphi\geq 0$. As in the proof of \Cref{equivalences part 1}, \eqref{divergence condition} implies
\begin{align*}
    \lim_{\delta\to 0}\lim_{i\to\infty} \iint_{\Omega_T} \varphi  z^{\delta}\cdot \nabla u_i \dx\dt & =\frac{1}{2}\iint_{\Omega_T} u^2 \partial_t \varphi \dx\dt - \iint_{\Omega_T} u z\cdot\nabla\varphi\dx\dt,
\end{align*}
and thus from \eqref{pairing condition}
\begin{equation*}
    \iint_{\Omega_T} \varphi \,df(Du)\dt + \iint_{\Omega_T} \varphi f^*(x,z)\dx\dt = \lim_{\delta\to 0}\lim_{i\to\infty} \iint_{\Omega_T} \varphi  z^{\delta}\cdot \nabla u_i \dx\dt.
\end{equation*}
On the other hand, 
\begin{align*}
    & \lim_{\delta\to 0}\lim_{i\to\infty} \iint_{\Omega_T}\varphi z^\delta \cdot \nabla u_i\dx\dt \\
    & \leq \lim_{\delta\to 0}\lim_{i\to\infty}\left(  \|z\|_\infty \iint_{\Omega_T}\varphi \lvert \nabla u- \nabla u_i\rvert \dx\dt + \iint_{\Omega_T}\varphi z^\delta \cdot \nabla u\dx\dt\right).
\end{align*}
To evaluate the first term on the right hand side, we argue as in the proof of \Cref{lemma:area strict convergences} to conclude 
\begin{equation*}
    \lim_{i\to\infty} \iint_{\Omega_T} \varphi \lvert \nabla u -\nabla u_i\rvert\dx\dt = \iint_{\Omega_T} \varphi \lvert\nabla u-\nabla u^\delta\rvert \dx\dt + \iint_{\Omega_T} \varphi  \,d\|D^s u^\delta\| \dt,
\end{equation*}
and similarly taking $\delta\to 0$ in the above gives
\begin{equation*}
    \lim_{\delta\to 0}\lim_{i\to\infty} \iint_{\Omega_T} \varphi \lvert \nabla u -\nabla u_i\rvert\dx\dt = \iint_{\Omega_T} \varphi  \,d\|D^s u\| \dt.
\end{equation*}
We conclude that
\begin{align*}
    & \iint_{\Omega_T} \varphi \,df(Du)\dt + \iint_{\Omega_T} \varphi f^*(x,z)\dx\dt \\
    & \leq  \iint_{\Omega_T}\varphi z \cdot \nabla u\dx\dt + \|z\|_\infty \iint_{\Omega_T} \varphi  \,d\|D^s u\| \dt,
\end{align*}
for any $\varphi\in C_0^\infty(\Omega_T)$ with $\varphi\geq 0$. Thus for a.e. $t\in (0,T)$ we have
\begin{align*}
    \int_{\Omega_T} \varphi (f(\cdot,\nabla u(t))+f^*(x,z(t))-z(t)\cdot\nabla u(t))\dx \leq  \|z\|_\infty \int_{\Omega} \varphi  \,d\|D^s u(t)\|,
\end{align*} 
for every $\varphi\in C_0^\infty(\Omega)$ with $\varphi\geq 0$. This implies the inequality 
\begin{equation*}
    (f(\cdot,\nabla u(t))+f^*(x,z(t))-z(t)\cdot\nabla u(t)) \mathcal{L}^n \leq \|z\|_\infty \|D^s u(t)\|,
\end{equation*}
as measures. As the two sides of the above inequality are mutually singular we find that 
\begin{equation*}
    (f(\cdot,\nabla u(t))+f^*(x,z(t))-z(t)\cdot\nabla u(t)) \mathcal{L}^n \leq 0,
\end{equation*}
as a measure, or equivalently
\begin{equation}
    \label{eq: inequality for subgradient}
    f(x,\nabla u)+f^*(x,z(x,t))\leq z(x,t)\cdot\nabla u(x,t),
\end{equation}
for a.e. $(x,t)\in\Omega_T$. Let $\xi\in\R^n$ and use the Fenchel inequality to estimate 
\begin{equation*}
    z(x,t)\cdot \xi\leq f(x,\xi)+f^*(x,z(x,t)).
\end{equation*}
Combining with \eqref{eq: inequality for subgradient} we arrive at
\begin{equation*}
    z(x,t)\cdot \left(\xi-\nabla u(x,t)\right)\leq f(x,\xi)-f(x,\nabla u),
\end{equation*}
for a.e. $(x,t)\in\Omega_T$, which shows that $z\in\partial_\xi f(x,\nabla u)$.
\end{remark}

\subsection{Differentiable functionals}
So far, we have shown that weak solutions are variational solutions. For differentiable functionals, we establish that variational solutions are weak solutions by differentiating the variational inequality in a suitable way.
\begin{proposition}
\label{prop: equivalence for differentiable functionals}
Suppose $f(x,\xi)$ is differentiable in $\xi$ for a.e. $x\in\Omega$ and every $\xi\in\R^n$, and $f^\infty(x,\xi)$ is differentiable in $\xi$ for all $(x,\xi)\in \Omega\times(\R^n\backslash\{0\})$, with $\lvert D_\xi f\rvert \leq M$ and $\lvert D_\xi f^\infty\rvert \leq M$ for some $M>0$. Then if $u$ is a variational solution of \eqref{PDE} it is also a weak solution, with $z=D_\xi f(x,\nabla u)$.
\end{proposition}
\begin{proof}
We begin by showing \eqref{pairing condition}. Let $\varphi\in C_0^\infty(\Omega)$ and let $0<h<\frac{1}{\|\varphi\|_\infty}$. Let also $\delta>0$ and $\epsilon>0$. Denote $u^\delta$ for the time mollification of $u$ as in \eqref{time mollification} with $w\in \BV(\Omega)\cap L^2(\Omega)$. Applying $v=u^\delta -h\varphi u^\delta$ as test function in \eqref{var inequality} gives
\begin{equation}
\label{estimate for pairing cond}
\begin{split}
& \int_0^T f(Du)(\Omega)\dt + \iint_{\partial\Omega\times(0,T)}\lvert Tg-Tu\rvert f^\infty(\cdot, \nu_\Omega)\dH\dt\\
    & \leq  \int_0^T f (D(u^\delta-h\varphi u^\delta))(\Omega)\dt + \iint_{\partial\Omega\times(0,T)}\lvert Tg-Tu^\delta\rvert f^\infty(\cdot, \nu_\Omega)\dH\dt \\ 
    & \quad +\iint_{\Omega_T} \partial_t (u^\delta -h\varphi u^\delta) (u^\delta-h\varphi u^\delta-u)\dx\dt +\frac{1}{2}\|w-u_0\|^2_{L^2(\Omega)}.
\end{split}
\end{equation}
We want to let $\delta\to 0$ in the above. The term containing the time derivatives simplifies as in \eqref{eq: time derivative simplifies} to give 
\begin{align*}
    &\limsup_{\delta\to 0}\iint_{\Omega_T} \partial_t (u^\delta -h\varphi u^\delta) (u^\delta-h\varphi u^\delta-u)\dx\dt \leq \frac{h}{2}\iint_{\Omega_T} u^2\partial_t\varphi\dx\dt.
\end{align*}
By combining \Cref{lemma: time moll conv area strictly}, \Cref{lemma:area strict convergences} and \Cref{reshetnyak continuity} we have
\begin{equation*}
    \lim_{\delta\to 0}f(D(u^\delta-h\varphi u^\delta)(t))(\Omega)\to f(D(u-h\varphi u)(t))(\Omega),    
\end{equation*} 
after choosing a subsequence if necessary, for a.e. $t\in(0,T)$. Hence letting $\delta\to 0$ in \eqref{estimate for pairing cond} 
\begin{align*}
    & \int_0^T f(Du)(\Omega)\dt + \iint_{\partial\Omega\times(0,T)}\lvert Tg-Tu\rvert f^\infty(\cdot, \nu_\Omega)\dH\dt\\
    & \leq  \int_0^T f (D(u-h\varphi u))(\Omega)\dt + \iint_{\partial\Omega\times(0,T)}\lvert Tg-Tu^\delta\rvert f^\infty(\cdot, \nu_\Omega)\dH\dt \\ 
    & \quad +\frac{h}{2}\iint_{\Omega_T} u^2\partial_t\varphi\dx\dt +\frac{1}{2}\|w-u_0\|^2_{L^2(\Omega)}.
\end{align*}
Letting $w\to u_0$ in $L^2(\Omega)$ and dividing by $h$ gives
\begin{equation*}
    \int_0^T \frac{1}{h} f(D(u-h\varphi u))(\Omega)-f(Du)(\Omega)\dt\leq \frac{1}{2}\iint_{\Omega_T} u^2\partial_t\varphi\dx\dt.
\end{equation*}
We let $h\to 0$ and use \Cref{lemma: anzellotti derivative functional of measure} to conclude that
\begin{align*}
    \iint_{\Omega_T} D_\xi f(x,\nabla u)\cdot & \nabla (u\varphi) \dx\dt  + \iint_{\Omega_T} \varphi f^\infty\left(x,\dfrac{Du}{\lvert Du\rvert}\right)\,d\|D^s u\|\dt \\
    & \leq \frac{1}{2}\iint_{\Omega_{T}} u^2\partial_t \varphi \dx\dt.
\end{align*}
Arguing by linearity shows that
\begin{align*}
    &\iint_{\Omega_T} \varphi D_\xi f(x,\nabla u)\cdot \nabla u \dx\dt  + \iint_{\Omega_T} \varphi f^\infty\left(x,\dfrac{Du}{\lvert Du\rvert}\right)\,d\|D^s u\|\dt \\
    & = \frac{1}{2}\iint_{\Omega_{T}} u^2\partial_t \varphi \dx\dt  +\iint_{\Omega_T}u D_\xi f(x,\nabla u)\cdot \nabla\varphi\dx\dt.
\end{align*}
Denote $z=D_\xi f(x,\nabla u)$. We have $f^*(x,z)=z\cdot \nabla u -f(x,\nabla u)$ and thus we arrive at \eqref{pairing condition}.

Let us show \eqref{divergence condition}. Let $\varphi$ be a test function as in \eqref{divergence condition}. For $\epsilon>0$, let $\eta_\epsilon$ be as in the proof of \Cref{equivalences for boundary condition}. Let $\delta>0$, $w\in \BV(\Omega)\cap L^2(\Omega)$, $u^\delta$ be as in \eqref{time mollification} with $u^\delta(0)=w$, and let $g^\delta$ be the time mollification of $g$ with $g^\delta(0)=0$. Let $h>0$ and denote 
\begin{equation*}
    v_\epsilon=u^\delta-h\eta_\epsilon\varphi + h(1-\eta_\epsilon)(g^\delta-u^\delta).    
\end{equation*}
We will test \eqref{var inequality} with $v_\epsilon$ and then take $\epsilon\to 0$. Denote $v=u^\delta-h\varphi$, which clearly is the limit in $L^2(\Omega_T)$ of $v_\epsilon$ as $\epsilon\to 0$. Let us consider what happens to $f(Dv_\epsilon)$ as $\epsilon\to 0$. Similarly as in the proof of \Cref{equivalences for boundary condition}, for $t\in (0,T)$ define the $\R^n$ valued measures $\mu_\epsilon(t), \mu(t)$ by 
\begin{equation*}
    \mu_\epsilon(t)(B)=Dv_\epsilon(t)(B\cap \Omega),
\end{equation*}
and 
\begin{equation*}
    \mu(t)(B)= Dv(t)(B\cap \Omega) + \int_{B\cap\partial\Omega} h(T\varphi(t)+Tg^\delta(t) -Tu^\delta(t))\nu_\Omega\dH
\end{equation*}
for Borel sets $B\subset\overline{\Omega}$. Then as in the proof of \Cref{equivalences for boundary condition}, we use \Cref{cut off trace lemma}, to conclude  
\begin{equation}
    \label{area strict convergence of particular measures}
    \lim_{\epsilon\to 0}\lvert(\mu_\epsilon(t), \mathcal L^n)\rvert(\overline{\Omega})\to \lvert(\mu(t), \mathcal L^n)\rvert(\overline{\Omega}).
\end{equation}
Hence, applying \Cref{reshetnyak continuity}
\begin{align*}
    & \lim_{\epsilon\to 0}\int_0^T f(Dv_\epsilon)\dt = \int_0^T \lim_{\epsilon\to 0} f(\mu_\epsilon)\dt \\
    & =\int_0^T  f (Dv)\dt + \iint_{\partial\Omega\times(0,T)} h\lvert T\varphi+Tg^\delta-Tu^\delta\rvert f^\infty(\cdot, \nu_\Omega)\dH\dt.
\end{align*}
Thus testing \eqref{var inequality} with $v_\epsilon$ and letting $\epsilon\to 0$ gives
\begin{equation}
\label{var estimate for div cond}
\begin{split}
    & \int_0^T f(Du)(\Omega)\dt + \iint_{\partial\Omega\times(0,T)}\lvert Tg-Tu\rvert f^\infty(\cdot, \nu_\Omega)\dH\dt\\
    & \leq \limsup_{\epsilon \to 0} \Bigg( \int_{0}^T f (Dv_\epsilon)(\Omega)\dt  +\iint_{\Omega_T} \partial_t v_\epsilon (v_\epsilon-u)\dx\dt\Bigg) +\frac{1}{2}\|w-u_0\|^2_{L^2(\Omega)} \\
    & \quad + \iint_{\partial\Omega\times(0,T)}\lvert Tg-Tu^\delta +h(T g^\delta-T u^\delta)\rvert f^\infty(\cdot, \nu_\Omega)\dH\dt \\
    & =\int_{0}^T f (D(u^\delta-h\varphi))(\Omega)\dt  +\iint_{\Omega_T} \partial_t v (v-u)\dx\dt +\frac{1}{2}\|w-u_0\|^2_{L^2(\Omega)} \\
    & \quad + \iint_{\partial\Omega\times(0,T)}\lvert Tg-Tu^\delta +h(T g^\delta-T u^\delta)\rvert f^\infty(\cdot, \nu_\Omega)\dH\dt \\
    & \quad + \iint_{\partial\Omega\times(0,T)} h\lvert T\varphi+Tg^\delta-Tu^\delta\rvert f^\infty(\cdot, \nu_\Omega)\dH\dt. 
\end{split}
\end{equation}
For the time derivative we find  as in \eqref{eq: time derivative simplifies} that
\begin{align*}
    \limsup_{\delta\to 0} \iint_{\Omega_T} \partial_t (u^\delta-h\varphi)(u^\delta-h\varphi-u)\dx\dt \leq h\iint_{\Omega_T} u\partial_t\varphi\dx\dt.
\end{align*}
As in the previous part of the proof we have that $f(D(u^\delta-h\varphi))(\Omega)\to f(D(u-\varphi))(\Omega)$ as $\delta\to 0$, and hence we let $\delta\to 0$ in \eqref{var estimate for div cond} to conclude
\begin{align*}
    & \int_0^T f(Du)(\Omega)\dt + \iint_{\partial\Omega\times(0,T)}\lvert Tg-Tu\rvert f^\infty(\cdot, \nu_\Omega)\dH\dt\\
    & \leq \int_{0}^T f (D(u-h\varphi))(\Omega)\dt  +h\iint_{\Omega_T} u\partial_t\varphi\dx\dt +\frac{1}{2}\|w-u_0\|^2_{L^2(\Omega)} \\
    & \quad + (1-h)\iint_{\partial\Omega\times(0,T)}\lvert Tg-Tu\rvert f^\infty(\cdot, \nu_\Omega)\dH\dt \\
    & \quad + \iint_{\partial\Omega\times(0,T)} h\lvert T\varphi+Tg-Tu\rvert f^\infty(\cdot, \nu_\Omega)\dH\dt. 
\end{align*}
Taking $w\to u_0$ in $L^2(\Omega)$ and rearranging we have
\begin{align*}
    & -\int_0^T f(D(u-h\varphi)(\Omega)- f(Du)(\Omega)\dt- h\iint_{\Omega_T} u\partial_t\varphi\dx\dt \\
    &\leq h\iint_{\partial\Omega\times(0,T)} \lvert T\varphi+Tg-Tu\rvert f^\infty(\cdot, \nu_\Omega)\dH\dt \\
    & \quad -h\iint_{\partial\Omega\times(0,T)}\lvert Tg-Tu\rvert f^\infty(\cdot, \nu_\Omega)\dH\dt. 
\end{align*}
We divide the above estimate by $h$ and let $h\to 0$ to arrive at
\begin{align*}
    & \iint_{\Omega_T}D_\xi f(x,\nabla u)\cdot \nabla \varphi-u\partial_t\varphi\dx\dt \\
    &\leq \iint_{\partial\Omega\times(0,T)} \left( \lvert T\varphi+Tg-Tu\rvert -\lvert Tu-Tg\rvert \right) f^\infty(\cdot, \nu_\Omega)\dH\dt,
\end{align*}
which finishes the proof.
\end{proof}

\subsection{Anzellotti-type Euler equations}

In this subsection we study the weak formulation of \eqref{PDE} when $\partial_t u\in L^1(0,T;L^2(\Omega))$ and $f$ is differentiable with respect to the second argument. In this case, weak solutions can be characterized using an Euler-Lagrange equation in the style of \cite{Anzellotti:1985}.
\begin{proposition}
\label{prop: Anzellotti type euler equations}
Let $f$ be as in \Cref{prop: equivalence for differentiable functionals}, and suppose the boundary values satisfy $\partial_t g\in L^1(0,T;L^2(\Omega))$. Suppose $u\in L^1_{w^*}(0,T;\BV(\Omega))\cap L^\infty(0,T;L^2(\Omega))$ with $\partial_t u\in L^1(0,T;L^2(\Omega))$ and $u(0)=u_0$ in the sense of \eqref{initial value condition}. Then $u$ is a weak solution of \eqref{PDE} if and only if
\begin{equation}
\label{Euler-Lagrange with controlled jump}
\begin{split}
    & \iint_{\Omega_T} D_\xi f(x,\nabla u)\cdot \nabla \varphi\dx\dt + \iint_{\Omega_T} D_\xi f^\infty\left(x,\dfrac{Du}{\lvert Du\rvert}\right)\cdot \frac{D\varphi}{\lvert D\varphi\rvert}\,d \lvert D^s \varphi\rvert\dt \\
    & = \iint_{\Omega_T} u\partial_t\varphi\dx\dt+ \iint_{\partial\Omega\times(0,T)} \sign_0(Tg-Tu)T\varphi f^\infty(\cdot, \nu_\Omega) \dH\dt,
\end{split}
\end{equation}
for all $\varphi\in L^1_{w^*}(0,T;\BV(\Omega))$, compactly supported in time,  with $\partial_t \varphi\in L^1(0,T;L^2(\Omega))$ such that for a.e. $t$, we have $\|D^s \varphi(t)\| \ll \|D^s u(t)\|$ and for a.e. $t$, $T\varphi(t)(x)=0$ for $\mathcal{H}^{n-1}$ a.e. $x\in \{x\in\partial\Omega: Tu(t)(x)=Tg(t)(x)\}$.
\end{proposition}

\begin{proof}
    Let $u$ satisfy \eqref{Euler-Lagrange with controlled jump}. For $\varphi\in L^1(0,T;W^{1,1}(\Omega))$, \eqref{Euler-Lagrange with controlled jump} is the same as condition $(ii)$ from \Cref{equivalences for boundary condition}, for $z=D_\xi f(x,\nabla u)$. We have already seen that this implies \eqref{divergence condition}. To show \eqref{pairing condition}, let  $\varphi\in C_0^\infty(\Omega_T)$ and test \eqref{Euler-Lagrange with controlled jump} with $u\varphi$.

    Suppose $u$ is a variational solution. Let $\varphi$ be as in \eqref{Euler-Lagrange with controlled jump}. For $h>0$, apply $u+h\varphi$ as test function in \eqref{var inequality} to obtain
    \begin{align*}
        & \int_0^T  f(Du)(\Omega)\dt + \iint_{\partial\Omega\times(0,T)} \lvert Tu-Tg\rvert f^\infty(\cdot, \nu_\Omega)\,\dH\dt \\
        & \leq \int_0^T  f(D(u+h\varphi))(\Omega)\dt + \iint_{\partial\Omega\times(0,T)} \lvert Tu+hT\varphi-Tg\rvert f^\infty(\cdot, \nu_\Omega)\,\dH\dt \\
        & \quad + h \iint_{\Omega_T} \partial_t u \varphi\dx\dt
    \end{align*}
    Integrating by parts, rearranging and diving by $h$ we have
    \begin{align*}
        & \iint_{\Omega_T} u\partial_t \varphi \dx\dt - \iint_{\partial\Omega\times(0,T)} \frac{1}{h}\left(\lvert Tg-Tu-hT\varphi\rvert-\lvert Tg-Tu\rvert\right) f^\infty(\cdot, \nu_\Omega)\,\dH\dt. \\
        & \leq  \int_0^T \frac{1}{h}( f(D(u+h\varphi))(\Omega)- f(Du)(\Omega))\dt
    \end{align*}
    by the assumptions on $\varphi$ and \Cref{lemma: anzellotti derivative functional of measure} we may let $h\to 0$ to conclude
    \begin{align*}
        & \iint_{\Omega_T} u\partial_t \varphi \dx\dt +\iint_{\partial\Omega\times(0,T)} \sign_0(Tg-Tu) T\varphi f^\infty(\cdot, \nu_\Omega)\,\dH\dt. \\
        & \leq \iint_{\Omega_T} D_\xi f(x,\nabla u)\cdot \nabla \varphi\dx\dt + \iint_{\Omega_T} D_\xi f^\infty\left(x,\dfrac{Du}{\lvert Du\rvert}\right)\cdot \frac{D\varphi}{\lvert D\varphi\rvert}\,d \lvert D^s \varphi\rvert\dt.
    \end{align*}
    Finally, arguing by linearity shows \eqref{Euler-Lagrange with controlled jump}.
\end{proof}

\section{Stability and equivalence}
\label{section: stability}
In this section we finish the proof of the equivalence of weak and variational solutions to \eqref{PDE} for 1-homogeneous functionals $f$, by showing that variational solutions are weak solutions. Our approach is based on proving an existence result for weak solutions. To do this we study the regularized equation 
\begin{equation}
    \label{eq: regularized equation}
    \partial_t u - \div \left(D_\xi \sqrt{\mu^2+f(x,Du)^2}\right)=0,
\end{equation}
for which \Cref{prop: equivalence for differentiable functionals} is applicable. We show that weak solutions of \eqref{PDE} with functional $f$ are obtained as limits of solutions of \eqref{eq: regularized equation} when $\mu\to 0$. To conclude that the limit is also a weak solution, the formulation in \eqref{intermediate condition} plays a crucial role. 

We begin by showing the following elementary lemma on the lower semicontinuity of functionals in parabolic $\BV$ spaces.
\begin{lemma}
\label{lemma: lsc of functional on parabolic bv for weak conv}
Let $u_i\in L^1_{w^*}(0,T;\BV(\Omega))$, $i\in\mathbb N$, be a sequence with $u_i\rightharpoondown u$ weakly in $L^1(\Omega_T)$ as $i\to \infty$, for some $u\in L^1(\Omega_T)$. Then
\begin{equation}
    \label{eq: weak l1 lsc functional on parabolic bv space}
    \int_0^T f(Du)(\Omega)\dt \leq \liminf_{i\to \infty}\int_0^Tf(Du_i)(\Omega)\dt.
\end{equation}
\end{lemma}
\begin{proof}
Suppose the right hand side of \eqref{eq: weak l1 lsc functional on parabolic bv space} is finite. Then, after choosing a subsequence if necessary, we may assume that the limit
\begin{equation*}
    \lim_{i\to\infty} \int_0^Tf(Du_i)(\Omega)\dt,
\end{equation*}
exists. By Mazur's lemma we find a sequence of convex combinations
\begin{equation*}
    \Tilde u_i = \sum_{j=i}^{m_i} a_{i,j} u_j, \quad a_{i,j}\geq 0 \quad \text{and} \quad \sum_{j=i}^{m_i}a_{i,j}=1,
\end{equation*}
such that $\Tilde u_i\to u$ strongly in $L^1(\Omega_T)$ as $i\to\infty$. Moreover, choosing a subsequence if necessary, we may assume $\Tilde u_i(t)\to u(t)$ in $L^1(\Omega)$ as $i\to\infty$, for a.e. $t\in (0,T)$. Let $\epsilon>0$, for $i$ large enough we have 
\begin{equation*}
    \int_0^T f(Du_j)(\Omega)\dt \leq \lim_{k\to\infty} \int_0^T f(Du_k)(\Omega)\dt +\epsilon
\end{equation*}
for all $j\geq i$. Hence by convexity
\begin{equation*}
    \int_0^T f(D\Tilde u_i)(\Omega)\dt\leq \sum_{j=i}^{m_i} a_{i,j}\int_0^T f(Du_j)(\Omega)\dt \leq \lim_{k\to\infty} \int_0^T f(Du_k)(\Omega)\dt +\epsilon  
\end{equation*}
whenever $i$ is large enough. This shows that 
\begin{equation*}
    \liminf_{i\to\infty}\int_0^T f(D\Tilde u_i)(\Omega)\dt \leq \lim_{i\to\infty} \int_0^T f(Du_i)\dt.
\end{equation*}
As $\Tilde u_i(t)\to u(t)$ in $L^1(\Omega)$ as $i\to\infty$, we may apply \Cref{lemma: reshetnyak lsc}, for a.e. $t\in (0,T)$, to conclude
\begin{align*}
    \int_0^T f(Du)(\Omega)\dt &  \leq \int_0^T\liminf_{i\to\infty} f(D \Tilde u_i)(\Omega)\dt \\
    & \leq \liminf_{i\to\infty}\int_0^T f(D\Tilde u_i)(\Omega)\dt \\
    &\leq\lim_{i\to\infty}\int_0^T f(Du_i)(\Omega)\dt.
\end{align*}
\end{proof}
\begin{remark}
\label{remarK: lsc of functional with trace on par bv}
\Cref{lemma: lsc of functional on parabolic bv for weak conv} implies the lower semicontinuity of functionals of the type 
\begin{equation*}
    \int_0^T f(Du)(\Omega)\dt + \iint_{\partial\Omega\times(0,T)}\lvert Tu-Tg\rvert f^\infty(\cdot,\nu_\Omega)\dt,
\end{equation*}
by considering an extension of $g$ to $\R^n \backslash (0,T)$, and applying \Cref{lemma: derivative of piecewise def function}. For the properties of the extension operator $E:\BV(\Omega)\to\BV(\R^n)$ on parabolic spaces we refer to \cite[Lemma 2.3]{BoegelDuzSchevTime:2016}.
\end{remark}

We state and prove a general existence theorem for variational solutions to \eqref{PDE}. The main novelty is the general assumptions on the initial and boundary values, that is, $g\in L^1_{w^*}(0,T;\BV(\Omega))$ and $u_0\in L^2(\Omega)$.
\begin{theorem}
\label{thm: existence theorem}
    There exists a variational solution to \eqref{PDE}.
\end{theorem}
\begin{proof}
By \cite[Theorem 1.7]{BDSS:2019} there exists a solution to \eqref{PDE} if we require that $u_0\in L^2(\Omega)$ and $g\in L^2(0,T;W^{1,2}(\Omega))\cap C([0,T],L^2(\Omega))$, with $\partial_t g\in L^2(\Omega_T)$ and $g(0)=u_0$. We will upgrade this result using a similar technique as in the proof of \cite[Theorem 1.2]{BoegelDuzSchevTime:2016}, however crucially we do not test the variational inequality with the boundary value function. 
     
For $g\in L^1_{w^*}(0,T;\BV(\Omega))$, applying \Cref{lemma: strict approx in parabolic bv} we find $\Tilde g \in L^1(0,T;W^{1,1}(\Omega))$ with $Tg(t)=T\Tilde{g}(t)$ $\mathcal H^{n-1}$-a.e., for a.e. $t\in (0,T)$, and hence we may assume that $g\in L^1(0,T;W^{1,1}(\Omega))$. Let $g_i\in L^2(0,T;W^{1,2})\cap C([0,T],L^2(\Omega))$, $i\in\mathbb N$, be a sequence with $\partial_t g_i\in L^2(\Omega_T)$ such that $g_i\to g$ in $L^1(0,T;W^{1,1}(\Omega))$ as $i\to\infty$. To obtain the existence of such a sequence, truncate $g$ for boundedness, apply a standard mollification in time to conclude continuity in time, and finally apply a standard mollification with respect to the spatial variables.

To use the result of \cite{BDSS:2019} we modify $g_i$ and $u_0$ so that $g(0)=u_0$ holds. Let $u_{0,\epsilon}=u_0*\varphi_\epsilon$ denote the convolution of $u_0$ with the standard mollifier for the parameter $\epsilon>0$. Note that
\begin{equation*}
    \|Du_{0,\epsilon}\|(\Omega)\leq \frac{c}{\epsilon}\|u_0\|_{L^1(\Omega)}.
\end{equation*}
For $i\in\mathbb N$, let 
\begin{equation*}
    \zeta_i(t)
        =\begin{cases}
            \frac{1}{i}(\frac{1}{i}-t) ,&\quad 0<t < \frac{1}{i},\\
            0 ,&\quad t\geq \frac{1}{i},\\
        \end{cases}
\end{equation*}
$w_i = u_{0,1/\sqrt{i}}$, and $h_i(t)=\zeta_i(t) w_i +(1-\zeta_i(t))g_i(t)$. Then $h_i(0)=w_i$, $h_i\to g$ in $L^1(0,T;W^{1,1}(\Omega))$ as $i\to\infty$ and $w_i\to u_0$ in $L^2(\Omega)$ as $i\to\infty$. Let $u_i$ be the variational solution of \eqref{PDE} with initial and boundary data $(w_i,h_i)$. Let $v$ be any comparison map as in \eqref{var inequality}. Then for a.e. $\tau\in (0,T)$
\begin{equation}
\label{eq: existence stability estimate 1}
\begin{split}
    & \int_0^{\tau} f(Du_i)(\Omega)\dt + \iint_{\partial\Omega\times(0,\tau)} \lvert Tu_i-Th_i\rvert f^\infty(\cdot, \nu_\Omega) \dH\dt + \frac{1}{2}\|(v-u_i)(\tau)\|^2_{L^2(\Omega)} \\
    & \leq \int_0^\tau f(Dv)(\Omega)\dt + \iint_{\partial\Omega\times(0,\tau)}\lvert Tv-Th_i\rvert f^\infty(\cdot, \nu_\Omega)\dH\dt \\
    & \quad + \iint_{\Omega_\tau} \partial_t v (v-u_i)\dx\dt +\frac{1}{2}\|v(0)-w_i\|^2_{L^2(\Omega)} \\
    & \leq T \iint_{\Omega_T} \lvert \partial_t v\rvert^2 \dx\dt + c\int_{0}^{T} f(Dv)(\Omega)+ \|Dv\|(\Omega)+\|Dh_i\|(\Omega)\dt\\
    & \quad +\frac{1}{2}\|v(0)-w_i\|^2_{L^2(\Omega)} + \esssup_{0<\tau<T} \frac{1}{4}\|(v-u_i)(\tau)\|^2_{L^2(\Omega)}.
\end{split}
\end{equation}
Here we used Young's inequality and the boundedness of the trace operator. Taking essential supremum over $\tau$ in the left hand side and absorbing we conclude that
\begin{equation}
\label{existence stability estimate 2}
\begin{split}
    & \esssup_{0<\tau<T}\|(v-u_i)(\tau)\|^2_{L^2(\Omega)} \\
    & \leq c \iint_{\Omega_\tau} \lvert \partial_t v\rvert^2 \dx\dt + c\int_{0}^{T} f(Dv)(\Omega)+ \|Dv\|(\Omega)+\|Dh_i\|(\Omega)\dt\\
    & \quad\quad +c\|v(0)-w_i\|^2_{L^2(\Omega)}.
\end{split}
\end{equation}
As $h_i$ converges strictly to $g$,  \eqref{existence stability estimate 2} implies $u_i$ is a bounded sequence in $L^\infty(0,T;L^2(\Omega))$ and \eqref{eq: existence stability estimate 1} thus gives
\begin{equation*}
    \liminf_{i\to\infty} \int_0^T\|Du_i(t)\|(\Omega)\dt <\infty.
\end{equation*}
Hence, choosing a subsequence if necessary, still denoted by $u_i$, there exists a function $u\in L^\infty(0,T;L^2(\Omega))$ such that $u_i\wstar u$, and \Cref{low semcont of par BV} implies $u\in L^1_{w^*}(0,T;\BV(\Omega))$. Let us show $u$ is a variational solution to \eqref{PDE}. For this we need an averaged form of \eqref{var inequality}. Let $v$ be a comparison map and $\tau\in (0,T)$. Integrating \eqref{var inequality} over $(\tau,\tau+h)$ for $h>0$ gives
\begin{align*}
    &  \int_0^{\tau} f(Du_i)(\Omega)\dt + \iint_{\partial\Omega\times(0,\tau)} \lvert Tu_i-Th_i\rvert f^\infty(\cdot, \nu_\Omega)\dH\dt \\
    & \quad + \frac{1}{h}\int_\tau^{\tau+h} \frac{1}{2}\|(v-u_i)(s)\|^2_{L^2(\Omega)}\ds \\
    & \leq  \int_{0}^{\tau+h} f(Dv)(\Omega)\dt  + \iint_{\partial\Omega\times(0,\tau+h)} \lvert Tv-Th_i\rvert f^\infty(\cdot, \nu_\Omega)\dH\dt \\
    & \quad +\frac{1}{h}\int_\tau^{\tau+h} \iint_{\Omega_s} \partial_t v (v-u_i)\dx\dt \ds+\frac{1}{2}\|v(0)-w_i\|^2_{L^2(\Omega)}.
\end{align*}
Applying \Cref{lemma: lsc of functional on parabolic bv for weak conv} and arguing as in \Cref{remarK: lsc of functional with trace on par bv}, we find that 
\begin{align*}
    & \int_0^{\tau} f(Du)(\Omega)\dt + \iint_{\partial\Omega\times(0,\tau)} \lvert Tu-Tg\rvert f^\infty(\cdot, \nu_\Omega)\dH\dt \\
    & \leq \liminf_{i\to \infty} \left(\int_0^{\tau} f(Du_i)(\Omega)\dt + \iint_{\partial\Omega\times(0,\tau)} \lvert Tu_i-Th_i\rvert f^\infty(\cdot, \nu_\Omega)\dH\dt\right).
\end{align*}
Hence, using the lower semicontinuity of the $L^2$ norm on $\Omega\times(\tau,\tau+h)$ we have
    \begin{align*}
        & \int_0^{\tau} f(Du)(\Omega)\dt + \iint_{\partial\Omega\times(0,\tau)} \lvert Tu-Tg\rvert f^\infty(\cdot, \nu_\Omega)\dH\dt \\ 
        & \quad +\frac{1}{h}\int_\tau^{\tau+h} \frac{1}{2}\|(v-u)(s)\|^2_{L^2(\Omega)}\ds \\
        & \leq \liminf_{i\to \infty} \left(\int_0^{\tau} f(Du_i)(\Omega)\dt + \iint_{\partial\Omega\times(0,\tau)} \lvert Tu_i-Th_i\rvert f^\infty(\cdot, \nu_\Omega)\dH\dt\right) \\
        & \quad +  \liminf_{i\to \infty} \frac{1}{h}\int_\tau^{\tau+h} \frac{1}{2}\|(v-u_i)(s)\|^2_{L^2(\Omega)}\ds \\
        & \leq \int_{0}^{\tau+h} f(Dv)(\Omega)\dt\ + \iint_{\partial\Omega\times(0,\tau+h)} \lvert Tv-Tg\rvert f^\infty(\cdot, \nu_\Omega)\dH\dt \\
        & \quad + \frac{1}{h}\int_\tau^{\tau+h}\iint_{\Omega_s} \partial_t v (v-u)\dx\dt\ds+\frac{1}{2}\|v(0)-u_{0}\|^2_{L^2(\Omega)}.
\end{align*}
Letting $h\to 0$ the Lebesgue differentiation theorem shows \eqref{var inequality} for a.e. $\tau\in(0,T)$.
\end{proof}
The following is our stability theorem for weak solutions under an area-type approximation of $1$-homogeneous functionals. The assumption that $f$ is $1$-homogeneous is necessary in our proof. However, it is unknown whether the corresponding result holds in general.
\begin{theorem}
\label{stability theorem}
Suppose $f$ is $1$-homogeneous and $D_\xi f$ exists for all $(x,\xi)\in \Omega\times(\R^n\backslash \{0\})$ such that for some $M>0$ we have $\lvert D_\xi f(x,\xi)\rvert \leq M$ for all $(x,\xi)\in \Omega\times(\R^n\backslash \{0\})$. For $\mu>0$ denote $f_\mu(x,\xi)=\sqrt{\mu^2+f(x,\xi)^2}$. Let $u_i$, $i\in\mathbb N$, be a sequence of weak solutions to
\begin{equation*}
\begin{cases}
    \partial_t u-\div\bigg( D_\xi f_{\mu_i}(x,Du) \bigg)=0&\mbox{in }\Omega\times(0,T),\\
       u(x,t)=g(x,t)&\mbox{on }\partial\Omega\times(0,T),\\
       u(x,0)=u_o(x)&\mbox{in }\Omega,
\end{cases}
\end{equation*} 
where $\mu_i\to 0$ as $i\to\infty$. Then, there exists  $u\in L^1_{w^*}(0,T;\BV(\Omega))\cap L^\infty(0,T;L^2(\Omega))$ such that $u_{\mu_i}\wstar u$ weakly* in $L^\infty(0,T;L^2(\Omega))$ as $i\to\infty$. The function $u$ is a weak solution of \eqref{PDE}, for the functional $f$ and the data $(u_0,g)$. 
\end{theorem}
\begin{proof}
Let $\mu_i$, $u_i$, and $f_{i}=f_{\mu_i}$ be as in the statement. Let $v$ be any admissible test function in the definition of variational solution. Then for a.e. $\tau\in (0,T)$
\begin{equation}
\label{stability estimate 1}
\begin{split}
    & \int_0^{\tau} f_i(Du_i)\dt + \iint_{\partial\Omega\times(0,T)} \lvert Tu_i-Tg\rvert\,d\mathcal{H}^{n-1}\dt + \frac{1}{2}\|(v-u_i)(\tau)\|^2_{L^2(\Omega)} \\
    & \leq \iint_{\Omega_\tau} \partial_t v (v-u_i)\dx\dt + \int_{0}^{\tau} f_i(Dv)\dt + \iint_{\partial\Omega\times(0,\tau)} \lvert Tv-Tg\rvert\,d\mathcal{H}^{n-1}\dt\\
    & \quad +\frac{1}{2}\|v(0)-u_0\|^2_{L^2(\Omega)} \\
    & \leq T \iint_{\Omega_T} \lvert \partial_t v\rvert^2 \dx\dt + \int_{0}^{T} f_i(Dv)\dt + \iint_{\partial\Omega\times(0,T)} \lvert Tv-Tg\rvert\,d\mathcal{H}^{n-1}\dt\\
    & \quad +\frac{1}{2}\|v(0)-u_0\|^2_{L^2(\Omega)} + \esssup_{0<\tau<\infty}\frac{1}{4}\|(v-u_i)(\tau)\|^2_{L^2(\Omega)}, 
\end{split}
\end{equation}
where in the last step we used Young's inequality. Taking essential supremum over $\tau$ in the left hand side shows
\begin{equation}
\label{stability estimate 2}
\begin{split}
    & \esssup_{0<\tau<T}\|(v-u_i)(\tau)\|^2_{L^2(\Omega)} \\
    & \leq c \iint_{\Omega_T} \lvert \partial_t v\rvert^2 \dx\dt + c\int_{0}^{T} \mathcal{A}(Dv)(\Omega)\dt \\
    & \quad + c\iint_{\partial\Omega\times(0,T)} \lvert Tv-Tg\rvert\,d\mathcal{H}^{n-1}\dt +c\|v(0)-u_0\|^2_{L^2(\Omega)},
\end{split}
\end{equation}
for a.e. $\tau\in (0,T)$. Then \eqref{stability estimate 2} implies $u_i$ is a bounded sequence in $L^\infty(0,T;L^2(\Omega))$ and then from \eqref{stability estimate 1}
\begin{equation*}
    \liminf_{i\to\infty} \int_0^T\|Du_i(t)\|(\Omega)\dt <\infty.
\end{equation*}
Hence, choosing a subsequence if necessary, still denoted by $u_i$, there exists a function $u\in L^\infty(0,T;L^2(\Omega))$ such that $u_i\wstar u$ in $L^\infty(0,T;L^2(\Omega))$ as $i\to\infty$. \Cref{low semcont of par BV} implies $u\in L^1_{w^*}(0,T;\BV(\Omega))$. Let us show $u$ is a weak solution to \eqref{PDE} by showing that \eqref{intermediate condition} holds. By \Cref{prop: equivalence for differentiable functionals}, the vector fields $z_i= D_\xi f_i(x,\nabla u_i)$ satisfy \eqref{intermediate condition} with respect to the solution $u_i$. We have 
\begin{equation*}
    D_\xi f_i(x,\xi)= \frac{f(x,\xi)}{\sqrt{\mu_i^2+f(x,\xi)^2}} D_\xi f(x,\xi),
\end{equation*}
for $\xi \neq 0$ and $D_\xi f_i(x,0)=0$. The boundedness of the $D_\xi f$ implies that $z_i$ is a bounded sequence in $L^\infty(\Omega_T;\R^n)$. Thus, after taking a subsequence, still denoted by $z_i$, there exists $z\in L^\infty(\Omega_T,\R^n)$ such that $z_i\wstar z$ as $i\to\infty$. Let us verify first that $f^*(x,z)=0$ for a.e. $(x,t)\in \Omega_T$.
    
For every $i$, we have $f_i^*(x,z_i)<\infty$ for a.e. $(x,t)$ and $f_i^\infty=f$, which implies $f^*(x,z_i)=0$, for a.e. $(x,t)$. Hence, we conclude that for any cube $Q\subset\Omega_T$ and $\xi\in\R^n$
\begin{equation*}
    \frac{1}{\abs Q}\iint_Q z \cdot \xi \dx\dt = \lim_{i\to \infty} \frac{1}{\abs Q} \iint_Q z_i \cdot \xi \dx\dt \leq  \frac{1}{\abs Q}\iint_Q f(x,\xi) \dx\dt.
\end{equation*}
Lebesgue's differentiation theorem implies $z\cdot \xi\leq f(x,\xi)$ for a.e. $(x,t)\in\Omega_T$, and consequently $f^*(x,z)=0$ for a.e. $(x,t)\in \Omega_T$. Moreover, as 
\begin{equation*}
    z_i\cdot \xi -f_i(x,\xi)\geq z_i\cdot \xi -f(x,\xi)-\mu_i   
\end{equation*}
we find that 
\begin{equation}
    \label{conjugate vanishes in limit}
        -\mu_i \leq f_i^*(x,z)\leq 0.
\end{equation}
We will apply a similar averaging procedure as in the proof of \Cref{thm: existence theorem} to conclude \eqref{intermediate condition}. Let $v$ be a comparison map as in \eqref{intermediate condition}. Integrate \eqref{intermediate condition} over $(\tau,\tau+h)$ for $h>0$, while using \eqref{conjugate vanishes in limit} and $f(Du_i)(\Omega)\leq f_i(Du_i)(\Omega)$ to conclude
\begin{align*}
    &  \int_0^{\tau} f(Du_i)(\Omega)\dt + \iint_{\partial\Omega\times(0,\tau)} \lvert Tu_i-Tg\rvert f^\infty(\cdot, \nu_\Omega)\dH\dt \\
    & \quad + \frac{1}{h}\int_\tau^{\tau+h} \frac{1}{2}\|(v-u_i)(s)\|^2_{L^2(\Omega)}\ds + \mu_i \lvert\Omega_T\rvert \\
    & \leq  \frac{1}{h}\int_\tau^{\tau+h}\iint_{\Omega_s}z_i\cdot\nabla v\dx\dt\ds  + \iint_{\partial\Omega\times(0,\tau+h)} \lvert Tv-Tg\rvert f^\infty(\cdot, \nu_\Omega)\dH\dt \\
    & \quad +\frac{1}{h}\int_\tau^{\tau+h} \iint_{\Omega_s} \partial_t v (v-u_i)\dx\dt \ds+\frac{1}{2}\|v(0)-u_{0}\|^2_{L^2(\Omega)} + \mu_i \lvert\Omega_T\rvert.
\end{align*}
Applying \Cref{lemma: lsc of functional on parabolic bv for weak conv}, the lower semicontinuity of the $L^2$ norm on $\Omega\times(\tau,\tau+h)$ and the previous estimate we have
\begin{align*}
    & \int_0^{\tau} f(Du)(\Omega)\dt + \iint_{\partial\Omega\times(0,\tau)} \lvert Tu-Tg\rvert f^\infty(\cdot, \nu_\Omega)\dH\dt \\ 
    & \quad +\frac{1}{h}\int_\tau^{\tau+h} \frac{1}{2}\|(v-u)(s)\|^2_{L^2(\Omega)}\ds \\
    & \leq \liminf_{i\to \infty} \left(\int_0^{\tau} f(Du_i)(\Omega)\dt + \iint_{\partial\Omega\times(0,\tau)} \lvert Tu_i-Tg\rvert f^\infty(\cdot, \nu_\Omega)\dH\dt\right) \\
    & \quad +  \liminf_{i\to \infty} \frac{1}{h}\int_\tau^{\tau+h} \frac{1}{2}\|(v-u_i)(s)\|^2_{L^2(\Omega)}\ds \\
    & \leq \frac{1}{h}\int_\tau^{\tau+h}\iint_{\Omega_s}z\cdot\nabla v\dx\dt\ds + \iint_{\partial\Omega\times(0,\tau+h)} \lvert Tv-Tg\rvert f^\infty(\cdot, \nu_\Omega)\dH\dt \\
    & \quad + \frac{1}{h}\int_\tau^{\tau+h}\iint_{\Omega_s} \partial_t v (v-u)\dx\dt\ds+\frac{1}{2}\|v(0)-u_{0}\|^2_{L^2(\Omega)}.
\end{align*}
Letting $h\to 0$ the Lebesgue differentiation theorem implies that $u$ satisfies \eqref{intermediate condition} for a.e. $\tau\in (0,T)$ with the vector field $z$. By \Cref{remark on subgradient in def of sol} we find that $z\in\partial_\xi f(x,\nabla u)$ and we are done.
\end{proof}
Combining the existence result for variational solutions from \Cref{thm: existence theorem} with the equivalence result for differentiable functionals from \Cref{prop: equivalence for differentiable functionals}, shows the existence of a sequence of weak solutions as in the statement of \Cref{stability theorem}. Thus, combining with the previous theorem shows the existence of weak solutions to \eqref{PDE} for $1$-homogeneous functionals. Moreover, by \Cref{equivalences part 1} any weak solution is a variational solution. Uniqueness of variational solutions is guaranteed by the comparison principle in \Cref{comparison principle}, and hence we conclude the following corollary.
\begin{corollary}
    Let $f$ be as in \Cref{stability theorem} and $g\in L^1_{w^*}(0,T;\BV(\Omega))$, $u_0\in L^2(\Omega)$. Then $u$ is a variational solution of \eqref{PDE} if and only if it is a weak solution.
\end{corollary}

\section{Local boundedness of solutions}
\label{section: local boundedness}
The goal of this section is to study the local boundedness of solutions to \eqref{eq 1}. As this is a local property, the boundary values and initial value of our solution do not play an essential role. Hence, it is natural to ask for a notion of local solution to \eqref{PDE}. We present one such definition below.
\begin{definition}
A function $u\in L^1_{w^*}(0,T;\BV_{\loc}(\Omega))\cap L^\infty(0,T;L^2_{\loc}(\Omega))$ is a local variational solution to \eqref{PDE} in $\Omega_T$ if for all Lipschitz domains $\Omega'\Subset \Omega$ the inequality
\begin{equation}
\label{local var inequality}
\begin{split}
    & \int_{\tau_1}^{\tau_2} f(Du)(\Omega')\dt + \frac{1}{2}\|(v-u)(\tau_2)\|^2_{L^2(\Omega')}  \\
    & \leq  \int_{\tau_1}^{\tau_2}  f(Dv)(\Omega')\dt + \iint_{\partial\Omega'\times(\tau_1,\tau_2)}\lvert Tu-Tv\rvert f^\infty(\cdot, \nu_{\Omega'})\dH\dt \\ 
    & \quad +\iint_{\Omega'\times(\tau_1,\tau_2)} \partial_t v (v-u)\dx\dt +\frac{1}{2}\|(v-u)(\tau_1)\|^2_{L^2(\Omega')},
\end{split}
\end{equation}
holds for a.e. $0<\tau_1<\tau_2<T$ and for all $v\in L_{w^*}^1(0,T;\BV(\Omega))\cap C([0,T],L^2(\Omega))$ with $\partial_t v\in L^1(0,T;L^2(\Omega))$.
\end{definition}
To prove a local boundedness result we will need the following parabolic energy estimate.
\begin{proposition} 
\label{prop: energy estimate}
Let $u$ be a local variational solution and $\alpha\geq 1$ with $u\in L^{\alpha+1}_{\loc}(\Omega_T)$. Then there is a constant $c=c(\alpha,\lambda,\Lambda)$ such that for every $k\in\mathbb R$, $\Omega'\times(\tau_1,\tau_2)\Subset\Omega_T$ we have
\begin{equation}
\label{energy estimate} 
\begin{split}
    & \esssup_{\tau_1 < t<\tau_2}\int_{\Omega'}\varphi(u-k)_+^{1+\alpha} \dx + \int_{\tau_1}^{\tau_2}\| D(\varphi (u-k)^\alpha_+)\|(\Omega')\dt\\ 
    & \leq c\iint_{\Omega'\times(\tau_1,\tau_2)}\lvert\nabla\varphi\rvert(u-k)^\alpha_+ \dx\dt + c\iint_{\Omega'\times(\tau_1,\tau_2)}(\partial_t\varphi)_+ (u-k)_+^{1+\alpha}\dx\dt \\
    & \quad + c \iint_{(\Omega'\times(\tau_1,\tau_2))\cap \{u> \lambda\}}\varphi \dx\dt, 
\end{split}
\end{equation}
whenever $\varphi\in C_0^\infty(\Omega'\times(\tau_1,\tau_2))$, $\varphi\geq 0$.
\end{proposition}
\begin{proof}
Let $\Omega'\times(\tau_1,\tau_2)\Subset\Omega_T$ and $\tau_1<t_1<t_2<\tau_2$ to be chosen later. Let $\varphi\in C_0^\infty(\Omega'\times(\tau_1,\tau_2))$ with $\varphi\geq 0$. We may assume $0\leq \varphi\leq 1$.

Let $h>k$, $\delta>0$ and denote $\Tilde\varphi=\varphi/(\alpha (h-k)^{\alpha-1})$. We will apply $v=u^\delta-\Tilde\varphi(\min(u^\delta,h)-\min(u^\delta,k))^\alpha$ as test function in \eqref{local var inequality}. Denote $w_\delta=\min(u^\delta,h)-\min(u^\delta,k)$ and $w=\min(u,h)-\min(u,k)$. For the time derivative that appears we use \eqref{derivative of time mollification} and integrate by parts to obtain
\begin{align*}
    & \iint_{\Omega_\times(t_1,t_2)} \partial_t v(v-u)\dx\dt \\
    & \leq \iint_{\Omega_\times(t_1,t_2)} -\partial_t w \Tilde\varphi w_\delta^\alpha -\partial_t(u^\delta-h)_+(k-h)\Tilde\varphi\dx\dt  \\
    & \quad + \frac{1}{2}\| (\varphi w_\delta^\alpha) (t_2)\|_{L^2(\Omega}^2-\frac{1}{2}\| (\varphi w_\delta^\alpha)(t_1)\|_{L^2(\Omega}^2 \\
    & = \iint_{\Omega\times(t_1,t_2)} \frac{1}{\alpha+1}w_\delta^{\alpha+1}\partial_t\Tilde\varphi + (k-h)(u^\delta-h)_+\partial_t\Tilde\varphi\dx\dt  \\
    & \quad - \left[\int_{\Omega'}  \frac{1}{\alpha+1}w_\delta^{\alpha+1}\Tilde\varphi \dx\right]_{t=t_1}^{t_2} - \left[\int_{\Omega'} (k-h)(u^\delta-h)_+\Tilde\varphi \dx\right]_{t=t_1}^{t_2} \\
    & \quad + \frac{1}{2}\| (\Tilde\varphi w_\delta)^\alpha(t_2)\|_{L^2(\Omega}^2-\frac{1}{2}\| (\Tilde\varphi w_\delta^\alpha )(t_1)\|_{L^2(\Omega}^2.
\end{align*}
Hence, testing \eqref{local var inequality} with $v$ and letting $\delta\to 0$ gives
\begin{equation}
\label{var ineq for energy estimate} 
\begin{split}
    & \int_{t_1}^{t_2} f(Dw)(\Omega)\dt \leq  \int_{t_1}^{t_2}  f(D(w-\Tilde\varphi w^\alpha))(\Omega)\dt \\ 
    &  +\iint_{\Omega\times(t_1,t_2)} \frac{1}{\alpha+1}w^{\alpha+1}\partial_t\Tilde\varphi + (k-h)(u-h)_+\partial_t\Tilde\varphi\dx\dt  \\
    & - \left[\int_{\Omega'}  \frac{1}{\alpha+1}w^{\alpha+1}\Tilde\varphi \dx\right]_{t=t_1}^{t_2} - \left[\int_{\Omega'} (k-h)(u-h)_+\Tilde\varphi \dx\right]_{t=t_1}^{t_2}. 
\end{split}
\end{equation}
Let $w_i\in L^1(0,T;W^{1,1}(\Omega))$, $i\in\mathbb N$, be a sequence with $w_i\to w$ as in \Cref{lemma: strict approx in parabolic bv}. We may assume $0\leq w_i\leq h-k$ and that $\chi_{w_i\geq 0}\to \chi_{w\geq 0}$. Then by convexity
\begin{align*}
    & \int_{t_1}^{t_2}f(D(w-\varphi w^\alpha )(\Omega) \dt \\
    & \leq \liminf_{i\to\infty} \iint_{\Omega\times(t_1,t_2)} f(\nabla (w_i-\varphi w_i^\alpha))\dx\dt \\
    & \leq \liminf_{i\to\infty}\iint_{\Omega\times(t_1,t_2)} \left(1-\Tilde\varphi\alpha w_i^{\alpha-1}\right)f(\nabla w_i) + \Lambda\lvert \nabla \Tilde\varphi\rvert w_i^\alpha+\Lambda\Tilde\varphi \chi_{\{w_i>0\}}\dx\dt \\
    & \leq \int_{t_1}^{t_2}f(Dw)(\Omega) \dt + \Lambda\iint_{\Omega\times(t_1,t_2)} \lvert \nabla\Tilde\varphi \rvert w^\alpha +\Tilde\varphi \chi_{\{w\geq  0\}} \dx\dt \\
    & \quad - \limsup_{i\to\infty} \iint_{\Omega\times(t_1,t_2)}\Tilde\varphi \alpha w_i^{\alpha-1} \dx\dt.
\end{align*}
In the above estimation it is important that $\Tilde\varphi\alpha w_i^{\alpha-1}\leq 1$. By lower semicontinuity of the $\BV$-norm and the assumptions on $f$ we have
\begin{align*}
    & \int_{t_1}^{t_2} \|D (\Tilde\varphi w^\alpha)\|(\Omega)\dt\\
    & \leq \limsup_{i\to\infty} \iint_{\Omega\times(t_1,t_2)}\Tilde\varphi\alpha w_i^{\alpha-1}\lvert \nabla w_i\rvert\dx\dt + \iint_{\Omega\times(t_1,t_2)} \lvert\nabla\Tilde\varphi\rvert w^\alpha\dx\dt\\
    & \leq \limsup_{i\to\infty} \frac{1}{\lambda} \iint_{\Omega\times(t_1,t_2)}\Tilde\varphi\alpha w_i^{\alpha-1}f(\nabla w_i)\dx\dt + \iint_{\Omega\times(t_1,t_2)} \lvert\nabla\Tilde\varphi\rvert w^\alpha\dx\dt.
\end{align*}
Inserting the two previous inequalities into \eqref{var ineq for energy estimate} and absorbing terms we arrive at
\begin{align*}
    & \lambda \int_{t_1}^{t_2} \|D (\Tilde\varphi w^\alpha)\|(\Omega)\dt \leq  (\Lambda+\lambda)\iint_{\Omega\times(t_1,t_2)} \lvert \nabla\Tilde\varphi \rvert w^\alpha +\Tilde\varphi \chi_{\{w\geq  0\}} \dx\dt \\ 
    & \quad +\iint_{\Omega\times(t_1,t_2)} \frac{1}{\alpha+1}w^{\alpha+1}\partial_t\Tilde\varphi + (k-h)(u-h)_+\partial_t\Tilde\varphi\dx\dt  \\
    & \quad - \left[\int_{\Omega'}  \frac{1}{\alpha+1}w^{\alpha+1}\Tilde\varphi \dx\right]_{t=t_1}^{t_2} - \left[\int_{\Omega'} (k-h)(u-h)_+\Tilde\varphi \dx\right]_{t=t_1}^{t_2}. 
\end{align*}
We remark that it is also possible to use the Chain rule in $\BV$ instead of the approximation scheme to conclude the above inequality, see \cite[Theorem 3.99]{AmbrosioFuscoPallara}. Multiplying the above by $\alpha(h-k)^{\alpha-1}$ and letting $h\to \infty$ we have
\begin{align*}
    & \lambda \int_{t_1}^{t_2} \|D (\varphi w^\alpha)\|(\Omega)\dt \leq  (\Lambda+\lambda)\iint_{\Omega\times(t_1,t_2)} \lvert \nabla\varphi \rvert w^\alpha +\varphi \chi_{\{w \geq 0\}}\dx\dt \\ 
    & \quad +\iint_{\Omega\times(t_1,t_2)} \frac{1}{\alpha+1}w^{\alpha+1}\partial_t\varphi \dx\dt - \left[\int_{\Omega'}  \frac{1}{\alpha+1}(u-k)_+^{\alpha+1}\varphi \dx\right]_{t=t_1}^{t_2},
\end{align*}
for a.e. $\tau_1<t_1<t_2<\tau_2$. Applying the above estimate to a sequence $k_i>k$ with $k_i\to k$ gives
\begin{equation}
\label{energy estimate before time choice}
\begin{split}
    & \lambda \int_{t_1}^{t_2} \|D (\varphi w^\alpha)\|(\Omega)\dt \leq  (\Lambda+\lambda)\iint_{\Omega\times(t_1,t_2)} \lvert \nabla\varphi \rvert w^\alpha +\varphi \chi_{\{w > 0\}}\dx\dt \\ 
    & \quad +\iint_{\Omega\times(t_1,t_2)} \frac{1}{\alpha+1}w^{\alpha+1}\partial_t\varphi \dx\dt - \left[\int_{\Omega'}  \frac{1}{\alpha+1}(u-k)_+^{\alpha+1}\varphi \dx\right]_{t=t_1}^{t_2}, 
\end{split}
\end{equation}
for a.e. $\tau_1<t_1<t_2<\tau_2$. By taking $t_1=\tau_1$, $t_2=\tau_2$ we obtain
\begin{equation}
\label{energy estimate for bv part}
\begin{split}
    & \lambda \int_{\tau_1}^{\tau_2} \|D (\varphi w^\alpha)\rvert(\Omega)\dt \leq  (\Lambda+\lambda)\iint_{\Omega\times(\tau_1,\tau_2)} \lvert \nabla\varphi \rvert w^\alpha +\varphi\chi_{\{w> 0\}} \dx\dt \\ 
    & \quad +\iint_{\Omega\times(\tau_1,\tau_2)} \frac{1}{\alpha+1}w^{\alpha+1}(\partial_t\varphi)_+ \dx\dt. 
\end{split}
\end{equation}
On the other hand, choosing $t_1=\tau_1$ and letting $t_2\in (\tau_1,\tau_2)$, \eqref{energy estimate before time choice} implies
\begin{equation}
\label{energy estimate for time part}
\begin{split}
    \esssup_{\tau_1<t<\tau_2}\int_{\Omega'} \frac{1}{\alpha+1}(u-k)_+^{\alpha+1}\varphi\dx & \leq  (\Lambda+\lambda)\iint_{\Omega\times(\tau_1,\tau_2)} \lvert \nabla\varphi \rvert w^\alpha +\varphi \chi_{\{w> 0\}}\dx\dt \\ 
    & \qquad +\iint_{\Omega\times(\tau_1,\tau_2)} \frac{1}{\alpha+1}w^{\alpha+1}\partial_t\varphi \dx\dt. 
\end{split}
\end{equation}
Finally, adding \eqref{energy estimate for bv part} and \eqref{energy estimate for time part} shows \eqref{energy estimate}.
\end{proof}

In what follows we use the notation
\begin{equation*}
    Q^-_{\rho,\theta}(z,t_0)=B(z,\rho)\times(t_0-\theta\rho,t_0),
\end{equation*}
for $(z,t_0)\in\Omega_T$, $\rho>0$, $\theta>0$. Given a measure $\mu$, measurable set $A$ with $\mu(A)>0$, and an integrable function $f$, the integral average of $f$ over $A$ is denoted by
\begin{equation*}
    \dashint_A f \dmu = \frac{1}{\mu(A)}\int_A f\dmu. 
\end{equation*}
We now arrive at our local boundedness result. The statement assumes that the local weak solution be a priori locally integrable to a power $r>n$, and hence the result can be seen as an $L^r-L^\infty$ regularisation effect for $r>n$.
\begin{proposition}
\label{prop: local boundedness}
Let $u$ be be a weak solution and assume $u\in L^r_{\loc}(\Omega_T)$ for some $r>n$. Then there exists a constant $c=c(n,r,\lambda,\Lambda)$ such that for any $k_0\in\R$,
\begin{equation}
\label{ess sup estimate}
\begin{split}
    & \esssup_{Q^-_{\rho/2,\theta}(z,t_0)} (u-k_0)_+ \\ 
    & \leq c \bigg( \frac{( 1+\frac{1}{\xi} )^{1+\frac{1}{n}}}{\theta} \bigg)^{\frac{n}{r-n}}\left(\fiint_{Q_0} (u-k_0)_+^r \dx\dt+\theta\rho\right)^\frac{1}{r-n}+\rho+\xi\theta, 
\end{split}
\end{equation}
whenever $Q^-_{\rho,\theta}(z,t_0)\Subset\Omega_T$.
\end{proposition}
\begin{proof}
Let $i\in\mathbb N$ and 
\begin{equation*}
    k_i=(1-2^{-i})k+k_0, 
\end{equation*}
where \begin{equation}
    \label{k first choice}
    k\geq \max(\xi\theta,\rho),
\end{equation}
is to be chosen later. Let also
\begin{equation*}
    \rho_i = \frac{1}{2}\rho+2^{-i-1}\rho \, \text{ and } \, Q_i = Q^-_{\rho_i,\theta}(z,t_0).
\end{equation*}
By Hölder's inequality
\begin{equation}
\label{de giorgi estimate 1}
\begin{split}
    & \fiint_{Q_{i+1}} (u-k_{i+1})_+^2 \dx\dt \\
    & \leq \dashint_{t_0-\theta\rho_{i+1}}^{t_0}\left(\dashint_{B\left(z,\rho_{i+1}\right)} (u-k_{i+1})_+^\frac{n}{n-1}\dx\right)^{1-\frac{1}{n}} \left( \frac{\lvert\{u(t)>k_{i+1}\} \cap B(z,\rho) \rvert}{\lvert B(z,\rho)\rvert} \right)^{\frac{1}{n}-\frac{1}{r}} \dt \\
    & \quad \cdot \esssup_{t_0-\theta\rho_{i+1} < t < t_0  } \left(\dashint_{B(z,\rho_{i+1})} (u-k_{i+1})_+^{r}\dx\right)^{\frac{1}{r}}.
\end{split}
\end{equation}
By \eqref{energy estimate} we have
\begin{equation}
\label{estimate 1 for ess sup r term}
\begin{split}
    & \esssup_{t_0-\theta\rho_{i+1} < t < t_0  } \dashint_{B(z,\rho_{i+1})} (u-k_{i+1})_+^{r}\dx \\
    & \leq \frac{c}{\rho\lvert B(z,\rho)\rvert}\iint_{Q_0}(u-k_{i+1})^{r-1}_+ \dx\dt + \frac{c}{\theta\rho\lvert B(z,\rho)\rvert}\iint_{Q_0} (u-k_{i+1})_+^{r}\dx\dt \\
    & \quad + c\frac{\lvert \{u> k_{i+1}\} \cap Q_0\rvert}{\lvert B(z,\rho)\rvert}.
\end{split}
\end{equation}
We estimate
\begin{equation}
\label{de giorgi iteration increasing exponent}
\begin{split}
    \iint_{Q_0}(u-k_{i+1})^{r-1}_+ \dx\dt & \leq \frac{2^{i+1}}{k}\iint_{Q_0}(u-k_i)^{r}_+ \dx\dt \\
    & \leq \frac{2^{i+1}}{\xi\theta}\iint_{Q_0}(u-k_0)^{r}_+ \dx\dt,
\end{split}
\end{equation}
where we used \eqref{k first choice}. Thus \eqref{estimate 1 for ess sup r term} becomes
\begin{align*}
    \esssup_{t_0-\theta\rho_{i+1} < t < t_0  }  \dashint_{B(z,\rho_{i+1})} (u-k_{i+1})_+^{r}\dx\leq c 4^i \left( 1+\frac{1}{\xi} \right) \left( \fiint_{Q_0} (u-k_0)_+^{r}\dx\dt + \theta\rho \right).
\end{align*}
Inserting this estimate in \eqref{de giorgi estimate 1} we so far have
\begin{equation}
\label{de giorgi estimate 2}
\begin{split}
& \fiint_{Q_{i+1}} (u-k_{i+1})_+^2 \dx\dt \\
    & \leq c 4^{\frac{i}{r}} D \bigg( 1+\frac{1}{\xi} \bigg)^\frac{1}{r} \dashint_{t_0-\theta\rho_{i+1}}^{t_0}\left(\dashint_{B\left(z,\rho_{i+1}\right)} (u-k_{i+1})_+^\frac{n}{n-1}\dx\right)^{1-\frac{1}{n}} \dt \\
    & \quad \cdot \esssup_{t_0-\theta\rho_{i+1} <t<t_0} \left( \frac{\lvert\{u(t)>k_{i+1}\} \cap B(z,\rho) \rvert}{\lvert B(z,\rho)\rvert} \right)^{\frac{1}{n}-\frac{1}{r}},
\end{split}
\end{equation}
where 
\begin{equation*}
    D = \left(\fiint_{Q_0} (u-k_{0})_+^{r}\dx\dt+\theta\rho\right)^\frac{1}{r}.
\end{equation*}
Estimating
\begin{equation}
    \label{estimate for level set}
    \lvert\{u(t)>k_{i+1}\} \cap B(z.\rho_{i+1})\rvert \leq \frac{4^{i+1}}{k^2}\int_{B(z,\rho_{i+1})}(u-k_i)^2_+(t)\dx,
\end{equation}
in \eqref{de giorgi estimate 2} we have
\begin{equation}
\label{de giorgi estimate 3}
\begin{split}
    & \fiint_{Q_{i+1}} (u-k_{i+1})_+^2 \dx\dt \\
    & \leq c \left(4^{1+\frac{1}{n}-\frac{1}{r}}\right)^i D \left( 1+\frac{1}{\xi} \right)^\frac{1}{r}  k^{-2\left(\frac{1}{n}-\frac{1}{r}\right)}\dashint_{t_0-\theta\rho_{i+1}}^{t_0}\left(\dashint_{B\left(z,\rho_{i+1}\right)} (u-k_{i+1})_+^\frac{n}{n-1}\dx\right)^{1-\frac{1}{n}} \dt \\
    & \quad \cdot \left( \esssup_{t_0-\theta\rho_{i+1} < t < t_0  } \dashint_{B(z,\rho_{i+1})} (u-k_{i+1})_+^{2}\dx\right)^{\frac{1}{n}-\frac{1}{r}}. 
\end{split}
\end{equation}
Applying the Sobolev inequality and \eqref{energy estimate} we find that 
\begin{align*}
    & \dashint_{t_0-\theta\rho_{i+1}}^{t_0}\left(\dashint_{B\left(z,\rho_{i+1}\right)} (u-k_{i+1})_+^\frac{n}{n-1}\dx\right)^{1-\frac{1}{n}} \dt \\
    & \leq c 2^i\fiint_{Q_i}(u-k_{i+1})_+ \dx\dt + \frac{c 2^i}{\theta}\fiint_{Q_i} (u-k_{i+1})_+^{2}\dx\dt+ c \rho \frac{\lvert Q_i\cap \{u> k_{i+1}\}\rvert}{\lvert Q_i\rvert } \\
    & \leq \frac{c 4^i}{\theta\xi}\fiint_{Q_i}(u-k_{i})_+^2 \dx\dt + \frac{c}{\theta\rho}\fiint_{Q_i} (u-k_i)_+^{2}\dx\dt.
\end{align*}
Here, we used \eqref{k first choice} and \eqref{estimate for level set} again. Moreover, \eqref{energy estimate} gives
\begin{align*}
    \esssup_{t_0-\theta\rho_{i+1} < t < t_0  } \dashint_{B(z,\rho_{i+1})} (u-k_{i+1})_+^{2}\dx \leq c \left( 1+\frac{1}{\xi} \right) 4^i\fiint_{Q_i}(u-k_{i})_+^2 \dx\dt.
\end{align*}
Using these estimates in \eqref{de giorgi estimate 3} we arrive at
\begin{equation}
\label{de giorgi estimate 4}
\begin{split}
    & \fiint_{Q_{i+1}} (u-k_{i+1})_+^2 \dx\dt \\
    & \leq c b^i \frac{\left( 1+\frac{1}{\xi} \right)^{1+\frac{1}{n}}}{\theta}  D k^{-2(1/n-1/r)}\left(\fiint_{Q_i}(u-k_{i})_+^2 \dx\dt\right)^{1+\frac{1}{n}-\frac{1}{r}}. 
\end{split}
\end{equation}
Denote
\begin{equation*}
    Y_i=\fiint_{Q_{i}} (u-k_{i})_+^2 \dx\dt.
\end{equation*}
Then \cite[Lemma 5.1]{DiBeGV2012} gives $Y_i\to 0$ as $i\to \infty$ if 
\begin{equation*}
    Y_{1}\leq C\bigg( D \frac{( 1+\frac{1}{\xi} )^{1+\frac{1}{n}}}{\theta} \bigg)^{-\frac{nr}{r-n}}k^{2}
\end{equation*}
On the other hand, arguing as in \eqref{de giorgi iteration increasing exponent} we find that 
\begin{equation}
    \label{condition for geometric conv lemma}
    Y_1 \leq c k^{2-r} \fiint_{Q_0} (u-k_0)_+^r \dx\dt
\end{equation}
so that $Y_i\to 0$ as $i\to \infty$ if 
\begin{equation*}
    k^r \geq C \bigg( D \frac{( 1+\frac{1}{\xi} )^{1+\frac{1}{n}}}{\theta} \bigg)^{\frac{nr}{r-n}}\fiint_{Q_0} (u-k_0)_+^r \dx\dt.
\end{equation*}
We estimate
\begin{align*}
    D^{\frac{nr}{r-n}}\fiint_{Q_0} (u-k_0)_+^r \dx\dt & =  \left(\fiint_{Q_0} (u-k_{0})_+^{r}\dx\dt+\theta\rho\right)^\frac{n}{r-n}\fiint_{Q_0} (u-k_0)_+^r \dx\dt \\
    & \leq \left( \fiint_{Q_0} (u-k_0)_+^r +\theta\rho \dx\dt \right)^{\frac{r}{r-n}} 
\end{align*}
And conclude that 
\begin{equation*}
    k \geq C \bigg( \frac{( 1+\frac{1}{\xi} )^{1+\frac{1}{n}}}{\theta} \bigg)^{\frac{n}{r-n}}\left(\fiint_{Q_0} (u-k_0)_+^r \dx\dt+\theta\rho\right)^\frac{1}{r-n},
\end{equation*}
is sufficient to guarantee \eqref{condition for geometric conv lemma}. On the other hand, keeping in mind \eqref{k first choice} we choose 
\begin{equation*}
    k = C \bigg( \frac{( 1+\frac{1}{\xi} )^{1+\frac{1}{n}}}{\theta} \bigg)^{\frac{n}{r-n}}\left(\fiint_{Q_0} (u-k_0)_+^r \dx\dt+\theta\rho\right)^\frac{1}{r-n}+\rho+\xi\theta.
\end{equation*}
and the convergence $Y_i\to 0$ as $i\to \infty$ shows \eqref{ess sup estimate}.
\end{proof}
Let us make some comments on the method used to prove \Cref{prop: local boundedness}. The technique is an adaptation of the De Giorgi iteration, in particular it is similar to techniques used for the singular parabolic p-Laplace equation \cite[Prop. A.2.1]{DiBeGV2012} and the total variation flow \cite[Prop. B.1]{DiBeGiaKla2017}. However, there is a crucial difference. In \cite{DiBeGV2012,DiBeGiaKla2017}, the iteration is done for solutions which are a priori assumed to be qualitatively locally bounded. Based on approximations of solutions with locally bounded solutions, this can then be upgraded to a similar statement as in \Cref{prop: equivalence for differentiable functionals}. Our technique avoids this assumption. The downside is that the resulting estimate is nonoptimal. Once the local boundedness is established, applying the technique of \cite{DiBeGV2012,DiBeGiaKla2017} shows a better estimate, which is of the form  
\begin{align*}
    & \esssup_{Q^-_{\rho/2,\theta}(z,t_0)} (u-k_0)_+ \\ 
    & \leq c \bigg( \frac{( 1+\frac{1}{\xi} )^{1+\frac{1}{n}}}{\theta} \bigg)^{\frac{n}{r-n}}\left(\fiint_{Q_0} (u-k_0)_+^r \dx\dt\right)^\frac{1}{r-n}+\rho+\xi\theta.
\end{align*}

Let us also remark that an assumption on the integrability of solutions is necessary for local boundedness. A straightforward calculation as in \cite{DiBeGiaKla2017} shows that 
\begin{equation*}
    u(x,t)=(1-t)_+ \frac{n-1}{\lvert x\rvert},
\end{equation*}
is a solution to \eqref{TVF} in $\R^n\times(0,\infty)$, but $u$ is not locally bounded. 
\begin{remark}
A function satisfying the conclusion of \Cref{prop: local boundedness} has an upper semicontinuous representative. More precisely, the function 
\begin{equation*}
    \Tilde u (x,t) = \lim_{\rho\to 0} \esssup_{Q^-_{\rho,\theta}(x,t)} u
\end{equation*}
agrees with $u$ a.e. It is straightforward to verify that $\Tilde u$ is upper semicontinuous. To see that $u=\Tilde u$ a.e., let $(x,t)$ be a Lebesgue point of $u$. Applying \eqref{ess sup estimate} with $k_0=u(x,t)$ shows that
\begin{align*}
    & \esssup_{Q^-_{\rho,\theta}(x,t)} u - u(x,t) \\
    & \leq c \bigg( \frac{( 1+\frac{1}{\xi} )^{1+\frac{1}{n}}}{\theta} \bigg)^{\frac{n}{r-n}}\left(\fiint_{Q_0} (u-u(x,t))_+^r \dx\dt+\theta\rho\right)^\frac{1}{r-n}+\rho+\xi\theta.
\end{align*}
Since $(x,t)$ is a Lebesgue point we let $\rho\to 0$ and conclude
\begin{equation*}
    \Tilde u(x,t)-u(x,t)\leq \xi\theta,
\end{equation*}
and letting $\xi\to 0$ we have $\Tilde u(x,t)\leq u(x,t)$. On the other hand, $u(x,t)\leq \Tilde u(x,t)$ clearly holds a.e., which shows the claim. 
\end{remark}

\textbf{Data availability statement} The manuscript has no associated data.

\textbf{Conflict of interest statement} On behalf of all authors, the corresponding author states that there is no conflict of interest.


\begin{thebibliography}{99}
\bibitem{AmarBellettini}
M. Amar and G. Bellettini, A notion of total variation depending on a metric with discontinuous coefficients, \emph{Ann. Inst. H. Poincar\'e{} C Anal. Non Lin\'eaire}, 11(1):91--133, 1994.

\bibitem{AmbrosioFuscoPallara}
L. Ambrosio, N. Fusco, and D. Pallara, {\it Functions of bounded variation and free discontinuity problems}, \emph{Oxford Mathematical Monographs}, Oxford University press, Oxford, 2000.

\bibitem{ABCM}
F.~Andreu, C.~Ballester, V.~Caselles, and J. M.~Maz{\'o}n,
The Dirichlet problem for the total variation flow.
  \emph{J. Funct. Anal.} 180(2):347--403, 2001.
  
   
\bibitem{Andreu-Caselles-Diaz-Mazon:2002}
   F.~Andreu, V.~Caselles, J.~I.~D\'iaz, and J. M.~Maz\'on,
   Some qualitative properties for the total variation flow.
   \emph{J. Funct. Anal.} 188(2):516--547, 2002. 

\bibitem{Andreu-Caselles-Mazon:2002}
F. Andreu-Vaillo, V. Caselles, and J.~M. Maz\'on, A parabolic quasilinear problem for linear growth functionals, \emph{Rev. Mat. Iberoamericana}, 18(1):135--185, 2002.
   
\bibitem{Andreu-Caselles-Mazon:book}
\newblock F.~Andreu, V.~Caselles, and J. M.~Maz\'on,
\newblock {\it Parabolic quasilinear equations minimizing linear growth functionals}.
\newblock{\em Progress in Mathematics}, 223, Birkh\"auser, Basel, 2004.

\bibitem{Anzellotti:1984}
G.~Anzellotti, 
\newblock Pairings between measures and bounded functions and compensated compactness. 
\newblock  \emph{Ann. Mat. Pura Appl.} 135(4):293--318, 1984.

\bibitem{Anzellotti:1985}
G.~Anzellotti, 
\newblock The Euler equation for functionals with linear growth. 
\newblock  \emph{Trans. Amer. Math. Soc.} 290(2):483--501, 1985.


\bibitem{Beck-Schmidt:2015}
L.~Beck and T.~Schmidt,
\newblock Convex duality and uniqueness for $\BV$-minimizers.
\newblock \emph{J. Funct. Anal.} 268(10): 3061--3107, 2015.
   
\bibitem{BDM2013}V. B\"ogelein, F. Duzaar, and P. Marcellini,
   Parabolic systems with $p,q$-growth: a variational approach.
   \emph{Arch. Ration. Mech. Anal.}, 210(1):219--267, 2013. 
   

\bibitem{BDM2015}V. B\"ogelein, F. Duzaar, and P. Marcellini,
   A time dependent variational approach to image restoration.
   \emph{SIAM J. Imaging Sci.} 8(2):968--1006, 2015. 
   
\bibitem{BoegelDuzSchevObstacle:2016}
    V.~B\"ogelein, F.~ Duzaar, and C.~Scheven,
    The obstacle problem for the total variation flow.
    \emph{Ann. Sci. \'Ec. Norm. Sup\'er.} (4) 49(5):1143--1188, 2016.

\bibitem{BoegelDuzSchevTime:2016}
    V.~B\"ogelein, F.~ Duzaar, and C.~Scheven,
    The total variation flow with time dependent boundary values.
    \emph{Calc. Var. Partial Differential Equations} 55(4):Art. 108, 2016.

\bibitem{BDSS:2019}
  \newblock V.~B\"ogelein, F.~Duzaar, L.~Sch\"atzler, and C.~Scheven,
  \newblock Existence for evolutionary problems with linear growth by
  stability methods.
  \newblock \emph{J. Differential Equations} 266(11):7709--7748, 2019.

\bibitem{CDMLP1988}
  \newblock M.~Carriero, G.~Dal Maso, A.~Leaci, and E.~Pascali,
  \newblock Relaxation of the nonparametric Plateau problem with an obstacle. 
  \newblock \emph{J. Math. Pures Appl.} (9) 67(4):359--396, 1988.
    
  \bibitem{DiBeGV2012}  
    E.~DiBenedetto, U. Gianazza, and V. Vespri, {\it Harnack's inequality for degenerate and singular parabolic equations.} Springer Monographs in Mathematics. Springer, New York, 2012.

  \bibitem{DiBeGiaKla2017}  
  E.~DiBenedetto, U.~Gianazza, and C.~Klaus,  
  A necessary and sufficient condition for the continuity of local minima of parabolic variational integrals with linear growth.
  \emph{Adv. Calc. Var.} 10(3): 209--221, 2017.

\bibitem{Giusti1984}
E. Giusti, {\it Minimal surfaces and functions of bounded variation}. \emph{Monographs in Mathematics}, 80, Birkhäuser Verlag, Basel, 1984.

\bibitem{GornyMazon}
W. G\'{o}rny and J. M.~Maz\'on, A duality-based approach to gradient flows of linear growth functionals, \emph{Publ. Mat.} 69(2): 341--365, 2025.

\bibitem{GornyMazonBook}
W. G\'{o}rny and J. M.~Maz\'on, {\it Functions of Least Gradient}, \newblock{\em Monographs in Mathematics}, 110, Birkh\"auser, Basel, 2024.


\bibitem{KinnunenLindqvist2006}
J. Kinnunen and P. Lindqvist, Pointwise behaviour of semicontinuous supersolutions to a quasilinear parabolic equation. \emph{Ann. Mat. Pura Appl.} (4) 185(3): 411--435, 2006. 

\bibitem{KinnunenScheven2022}
J. Kinnunen and C. Scheven, On the definition of solution to the total variation flow. \emph{Calc. Var. Partial Differential Equations}, 61(1): Paper No. 40, 20, 2022. 

\bibitem{KristensenRindler2010Young}
J. Kristensen and F. Rindler, Characterization of generalized gradient Young measures generated by sequences in $W^{1,1}$ and $\mathrm{BV}$. \emph{Arch. Ration. Mech. Anal.}, 197(2):539--598, 2010. 

\bibitem{KristensenRindler2010Relaxation}
J. Kristensen and F. Rindler, Relaxation of signed integral functionals in $\mathrm{BV}$. \emph{Calc. Var. Partial Differential Equations}, 37(1-2):29--62, 2010.


\bibitem{LicT78}
A. Lichnewsky and R. Temam, Pseudosolutions of the time-dependent minimal surface problem.
    \emph{J. Differential Equations} 30(3):340--364, 1978.

\bibitem{Moll2005}
J. S. Moll, The anisotropic total variation flow. \emph{Math. Ann.}, 332(1):177--218, 2005.

\bibitem{Rindler2012}
F. Rindler, Lower semicontinuity and Young measures in $\mathrm{BV}$ without Alberti's rank-one theorem. \emph{Adv. Calc. Var.}, 5(2):127-159, 2012.
\end{thebibliography}
\end{document}